\documentclass[11pt,letterpaper]{amsart}

\usepackage[utf8]{inputenc}
\usepackage{color}
\usepackage{mathrsfs}
\usepackage{dsfont}
\usepackage[all]{xy}
\usepackage{amsmath,amssymb,amsthm}
\usepackage{mathtools}
\usepackage[top=2cm, bottom=2cm, left=1.5cm, right=1.5cm]{geometry}

\usepackage[pdfusetitle,hidelinks]{hyperref}
\newcommand{\iref}[2]{\hyperref[#1]{#2}}
\newcommand{\secref}[1]{\hyperref[#1]{§\ref*{#1}}}
\newcommand{\thmref}[1]{\hyperref[#1]{Theorem~\ref*{#1}}}
\newcommand{\propref}[1]{\hyperref[#1]{Proposition~\ref*{#1}}}
\newcommand{\lemref}[1]{\hyperref[#1]{Lemma~\ref*{#1}}}
\newcommand{\corref}[1]{\hyperref[#1]{Corollary~\ref*{#1}}}

\newtheorem{theorem}{Theorem}[section]
\newtheorem{lemma}[theorem]{Lemma}
\newtheorem{proposition}[theorem]{Proposition}
\newtheorem{corollary}[theorem]{Corollary}

\newtheorem{remark}[theorem]{Remark}

\numberwithin{equation}{section}

\title[Squarefree values of random polynomials]
{Squarefree values of random polynomials are Gaussian mixtures}
\author{Omer Simhi}
\date{}
\begin{document}

\begin{abstract}
Fix $d\ge1$ and choose a polynomial $P$ uniformly at random from the set of integer polynomials of degree $d$ whose coefficients have absolute value at most $H$. We study $E_{P}$, the error in the count of squarefree values $P\left(n\right)$ for $n\leq x$. In the range $x=H^{o\left(1\right)}$, $x\to\infty$, we prove that $E_{P}\left(x\right)/x^{\frac{1}{4}}$ converges to a Gaussian mixture that is not itself Gaussian, $\sqrt{C\left(G\right)}Z$, where $G$ is the weak limit, as $H\rightarrow\infty$, of the prime-indexed vector $\left(P\bmod p^{2}\right)_{p}$, $C$ is a nonnegative function and $Z$ is an independent standard Gaussian.  This answers a question of Sofos. Finally, in the appendix, we establish that
$\log \mathbb{E}[C(G)^m]\asymp_d m^2/(\log m)^2$
as $m\to\infty$ and deduce $-\log\mathbb P\bigl(|\sqrt{C(G)}Z|>t\bigr)
\asymp_d(\log t)^2(\log\log t)^2$ as 
$t\to\infty.$

\end{abstract}
\maketitle

\section{Introduction}\label{sec:introduction}

\begingroup
\renewcommand{\thefootnote}{}
\footnotetext{\textbf{Declaration of AI usage.}
After completing an advanced draft of the manuscript,
the author used ChatGPT 5.6 Sol (OpenAI) for proofreading,
grammatical and typographical corrections, and refining
the proofs of \lemref{lem:partial-sum-moment-bound}
and \lemref{lem:app-finite-prime-k-majorant}.
The mathematical ideas and original arguments are due
to the author.}
\endgroup

Fix an integer polynomial $P(t)=\sum_{j=0}^{d}a_jt^j\in\mathbb{Z}[t]$ and take an integer $x\ge1$. Recall that $\mu^{2}$ is the indicator of squarefree numbers and set 
\[ E_{P}(x):=\sum_{n\leq x}\mu^{2}(P(n))-\mathfrak{S}_{P}x, \quad \mathfrak{S}_{P}:=\prod_{p}\left(1-\frac{\rho_{P}\left(p^{2}\right)}{p^{2}}\right),\]
where $\rho_{P}\left(q\right):=\#\left\{ n\bmod q:P\left(n\right)\equiv 0\bmod q\right\}$. 
It remains a deep open problem to prove that
$E_P(x)=o(x)$ as $x\to\infty$ for arbitrary
`admissible' polynomials $P$. Unconditionally,
this is known for $\deg P\le 3$ by Hooley
\cite{key-21}. More recently, a preprint by
OpenAI \cite{key-28} claims to
establish the result for irreducible quartic
polynomials. Under the assumption
of the \textit{abc} conjecture, Granville
\cite{key-3} established the result for
polynomials of arbitrary degree.

Several authors have studied $E_P(x)$ on average over the coefficient box
\[ \mathscr{P}_{d}(H) := \left\{ P\in\mathbb{Z}[t]:\deg P=d,\ \vert P\vert\leq H\right\} ,\quad\vert P\vert:=\max_{0\leq j\leq d}\lvert a_{j}\rvert, \]
see Filaseta \cite{key-7}, and more recently Shparlinski \cite{key-8}, Browning--Shparlinski \cite{key-9} and Jelinek \cite{key-10}. Specifically, they endow the finite set $\mathscr{P}_{d}(H)$ with a uniform probability measure, and study $E_P$ as a random variable, for $P$ distributed according to this measure.

Browning--Shparlinski \cite[Theorem 1.1]{key-9} obtained a power-saving first-moment estimate in a power range $x\le H^{\alpha_d}$, and Jelinek proved a corresponding second-moment bound. Recently, Sofos obtained the first asymptotic formula for the variance \cite[Theorem 1.1]{key-1-1}. With 
\[ \gamma_{0} := -4\frac{\zeta\left(-\frac{1}{2}\right)}{\zeta(2)^{2}}\prod_{p}\left(1-\frac{2-p^{-\frac{1}{2}}+p^{-1}}{\left(p+1\right)^{2}}\right)=0.192\ldots, \]
he proved, for fixed $d\ge3$ and a suitable power range $x\leq H^{\beta_d}$ that there exists $\delta>0$ such that
\[
\frac{1}{(2H)^{d+1}}\sum_{P\in\mathscr{P}_{d}(H)}E_{P}(x)^{2}=\gamma_{0}x^{\frac{1}{2}}+O\left(x^{\frac{1}{2}-\delta}\right),
\] suggesting that $x^{\frac{1}{4}}$ is the natural scale of $E_{P}\left(x\right)$. In the paper, he posed the question of studying the distribution of $E_P(x)/x^{1/4}$ in the large-$x$ limit. We fully answer this question in this paper.

\subsection{The main result}\label{sec:main-result}
For an integer $q\ge1$, put 
\[ V_{q} := \left\{ g\in(\mathbb{Z}/q\mathbb{Z})[t] : \deg g\leq d\right\}, \quad V_{\square} := \prod_{p}V_{p^{2}}. \] 
The set $V_q$ is a finite abelian group; we endow it with the uniform probability measure (which agrees with the Haar measure) and with the discrete topology. It follows that $V_{\square}$ is compact and can be equipped with a probability Haar measure $dg$. Let $Z\sim N(0,1)$ be a standard normal random variable. Throughout the paper, we use $\overset{d}{\longrightarrow}$ to denote convergence in distribution (weak convergence). Our main result is the following. Let $P$ be uniformly distributed on $\mathscr{P}_{d}(H)$ and set 
\[
G_{H}:=(P\bmod p^{2})_{p},\quad X_{x,H}:=\frac{E_{P}(x)}{x^{\frac{1}{4}}}. \]
\begin{theorem}\label{thm:gaussian-mixture-limit}
Fix $d \ge 1$. If $x=H^{o(1)}$, then, as $x,H\rightarrow\infty$,\[
\left(G_{H},X_{x,H}\right)\overset{d}{\longrightarrow}\left(G,\sqrt{C\left(G\right)}Z\right),
\] where $G$ is distributed according to $dg$, $Z$ is independent of $G$ and $C\colon V_{\square}\rightarrow[0,\infty)$ is the nonnegative measurable function constructed in \secref{sec:limiting-variance-construction}. Moreover, $C(g)$ has positive variance. Consequently, the limiting distribution is not Gaussian.
\end{theorem}
We could have simply stated $X_{x,H}\overset{d}{\longrightarrow}\sqrt{C\left(G\right)}Z$, but the more general formulation above makes it clear that the random variance $C(G)$ arises from $G_H$.

\begin{remark}\label{rem:mixture}
Prior to our work, Gaussian mixture limits have appeared very few times in number theory, solely in the recent works on random
multiplicative functions of Gorodetsky and Wong \cite{key-26}
and Hardy \cite{key-27}.
\end{remark}

\subsection{Squarefrees in short intervals}
In recent work, Gorodetsky--Mangerel--Rodgers \cite{key-14} (hereafter GMR) studied the distribution of $\sum_{x\le n < x+H}\mu^2(n) - H/\zeta(2)$ through the method of moments, where $x$ is an integer chosen uniformly at random from $\{1,2,\ldots,X\}$. In the range $H=X^{o(1)}$ they obtained a standard Gaussian distribution after normalizing by $H^{1/4}$ times a certain arithmetic constant (see \cite{key-17,key-14} for earlier works on the variance, and \cite{key-23} for earlier work on moments).

We observe that replacing $X$ and $H$ in the problem studied by GMR, is similar to studying the distribution of $E_P$, when $P$ is now a random \textit{monic} linear polynomial (instead of a random polynomial from $\mathscr{P}_{1}\left(H\right)$): indeed, 
\[ \sum_{x\le n < x+H}\mu^2(n) - \frac{H}{\zeta(2)} = \sum_{0\le m<H} \mu^2(P(m)) - \mathfrak{S}_{P} H\]
where $P(t):=t+x$. This suggested to us that the methods of \cite{key-14} could play a useful role in establishing \thmref{thm:gaussian-mixture-limit}.

\subsection{Structure of the proof}\label{sec:main-objects-proof-structure}
As in GMR, it is natural to approach the distribution of $E_P(x)$ via the method of moments. Writing $\mathbb{E}_{H}$ for expectation with respect to the uniform measure on $\mathscr{P}_{d}\left(H\right)$, for each fixed integer $m\geq1$ we consider the $m$th moment\[
\mathrm{MP}_{m}\left(x,H\right):=\mathbb{E}_{H}\left[E_{P}\left(x\right)^{m}\right]=\frac{1}{\left|\mathscr{P}_{d}\left(H\right)\right|}\sum_{P\in\mathscr{P}_{d}\left(H\right)}E_{P}\left(x\right)^{m}.
\]
However, in contrast with the setting of GMR \cite{key-14}, it turns out that we cannot deduce the distribution directly from the moments $\mathrm{MP}_{m}$ alone. The reason is that, unlike the Gaussian limit in their setting, our Gaussian mixture does not satisfy standard sufficient conditions for moment determinacy, such as Carleman's criterion \cite[Chapter III]{key-20}: see \secref{app:C-moment-growth} for discussion. Instead, we need to establish asymptotics for weighted moments, where we weigh by functions of finitely many local reductions of $P$: $\Phi\left(\left(P\mod p^{2}\right)_{p\le w}\right)$, for arbitrary $w\ge2$ and arbitrary bounded $\Phi$. This constitutes the main arithmetic step of the argument and occupies the greater part of the paper. For $w\ge2$, write for $g=\left(g_{p}\right)_{p}\in V_{\square}$\begin{equation}
\mathcal{F}_{w}:=\sigma\left(g_{p}:p\leq w\right),\label{eq:sigmaAlgbra}
\end{equation} where $\sigma\left(\cdots\right)$ means $\sigma$-algebra. Also let $\mu_{m}:=\mathbb{E}\left(N\left(0,1\right)^{m}\right)$ which is $\left(m-1\right)!!$ if $m$ is even and $0$ otherwise. Our main technical result is the following.

\begin{proposition}\label{prop:finite-coordinate-weighted-moments}
Fix integers $m\geq1$ and $w\geq2$, and let $\Phi$ be a bounded $\mathcal{F}_{w}$-measurable function on $V_{\square}$. If $H\ge x^{C_{d}m^{2}}$ for some constant $C_{d}>0$, then \[
\lim_{x\rightarrow\infty}\mathbb{E}_H\left[\Phi\left(G_{H}\right)X_{x,H}^{m}\right]=\mu_{m}\int_{V_{\square}}\Phi\left(g\right)C\left(g\right)^{\frac{m}{2}}dg.
\]
\end{proposition}
The proof of \propref{prop:finite-coordinate-weighted-moments}
builds on some of the ideas in GMR's argument, which in turn
relied on the work of Montgomery and Vaughan \cite{key-16},
Montgomery and Soundararajan \cite{key-25} and an
observation of Nunes \cite{key-24}.

In \secref{sec:polynomial-local-comparison}--%
\secref{sec:paired-main-term-limiting-variance}
we establish the $\Phi\equiv 1$ case of
\propref{prop:finite-coordinate-weighted-moments},
thereby obtaining the asymptotics of ordinary moments.
Then, in \secref{sec:conditional-moments-convergence}
we explain the modifications required to obtain
\propref{prop:finite-coordinate-weighted-moments}
in its full generality. Moreover, in
\secref{sec:conditional-moments-convergence}
we spell out in detail the deduction of
\thmref{thm:gaussian-mixture-limit} from
\propref{prop:finite-coordinate-weighted-moments},
which, while being standard to probabilists,
will be useful to number theorists as earlier
works in the area rarely involved a Gaussian mixture, and in the few cases they did (see Remark~\ref{rem:mixture}), the proof was not approached via the (weighted) moment method underlying our proof of \thmref{thm:gaussian-mixture-limit}.

\subsubsection{The starting point: comparison with the local model}
\label{sec:local-model-subsubsection}
The proof begins by passing from $E_P$ to a random variable $E_g$, which we call a \textit{local model} for $E_P$. It has more independence than $E_P$ does, yet the moments of $E_g$ and $E_P$ are close as we now explain.

Let
\[
V_{\square}^{\sharp}:=\prod_p V_{p^2}^{\sharp},
\quad V_{p^2}^{\sharp}:=V_{p^2}\setminus\{0\},
\]
equipped with the product probability measure, and let $g=(g_p)_p\in V_{\square}^{\sharp}$. Define
\begin{equation}
E_g(x):=
\sum_{n\le x}\prod_p\mathds{1}_{g_p(n)\not\equiv0\bmod p^2}
-x\prod_p\left(1-\frac{\rho_{g_p}(p^2)}{p^2}\right).
\label{eq:local-discrepancy}
\end{equation}
The tuple $g$ serves as a model for $G_H$, with the
advantage that its coordinates $g_p$ are independent.

We extend $E_g(x)$ by zero outside $V_\square^\sharp$. The Haar mass of $V_\square^\sharp$ in $V_\square$ is $\zeta(2d+2)^{-1}$; consequently, an integral over the normalized product measure on $V_\square^\sharp$ must be multiplied by this factor when passing to full Haar measure.

We now explain how we relate $E_P$ to $E_g$. For squarefree $k$, put
\[
\Delta_P(k;x):=\sum_{n\le x}
\left(\mathds{1}_{k^2\mid P(n)}-\frac{\rho_P(k^2)}{k^2}\right),
\]
and define $\Delta_g(k;x)$ analogously. 
After excluding an elementary exceptional set of polynomials (see \lemref{lem:exceptional-polynomials-negligible}), M\"obius inversion yields 
\begin{equation}\label{eq:EpMobius}
E_P(x)=\sum_{k\ge1}\mu(k)\Delta_P(k;x)
\end{equation}
for every remaining polynomial $P$.
On expanding the $m$th power of \eqref{eq:EpMobius}, we obtain a summation of $m$-tuples $\mathbf{k}=(k_1,\ldots,k_m)$. We truncate it according to the size of the lcm $K:=[k_1,\ldots,k_m]$, and estimate the contribution of large $K$.

For each fixed tuple $(k_1,\ldots,k_m)$, $\prod_i\Delta_P(k_i;x)$ depends only on the coefficients of $P$ modulo $K^2$. With a weight function $\Phi$ depending on primes $p\le w$, the required modulus is $[K^2,Q_w]$, where $Q_w:=\prod_{p\le w}p^2$. As $H$ grows, $P$ equidistributes modulo $[K^2,Q_w]$ (see \lemref{lem:polynomial-residue-equidistribution}), which allows us to compare the truncated expression directly with the local-model counterpart. The  contribution of the discarded terms on both sides of (the sums of) $\Delta_P(k;x)^m$ and $\Delta_g(k;x)^m$, namely those with large lcm, is discarded easily using divisor bounds.

In \secref{sec:polynomial-local-comparison}, this gives \propref{prop:polynomial-local-moment-comparison}: for every fixed $m,w$ and bounded $\mathcal F_w$-measurable $\Phi$, whenever $H\ge x^{C_dm^2}$,
\begin{equation}
\mathbb E_H[\Phi(G_H)E_P(x)^m]
=
\frac1{\zeta(2d+2)}\int_{V_\square^\sharp}\Phi(g)E_g(x)^m\,dg
+o_{m,d,w}\!\left(\|\Phi\|_\infty x^{m/4}\right).
\label{eq:intro-polynomial-local-comparison}
\end{equation}
Note that on the right hand of \eqref{eq:intro-polynomial-local-comparison}, there is no dependence on the height parameter $H$, and it is only influencing the error between the $m$th moment and \eqref{eq:intro-polynomial-local-comparison}.

\begin{remark}
In addition to the convergence in distribution established in \thmref{thm:gaussian-mixture-limit}, our proof yields
the following stable convergence\footnote{
Let $(\Omega,\mathscr F,\mathbb P)$ be a probability
space and $\mathcal G\subseteq\mathscr F$ a
sub-$\sigma$-algebra. Real-valued random variables
$Y_x$ converge $\mathcal G$-stably to $Y$
(possibly on an extension of $\Omega$) if, for every
bounded $\mathcal G$-measurable random variable $W$
and bounded continuous $f:\mathbb R\to\mathbb R$,
we have $\mathbb E[Wf(Y_x)]\to\mathbb E[Wf(Y)]$.
}
statement for the local model introduced above.
Let $G$ be Haar-distributed on $V_\square$ and put
$Y_x:=x^{-1/4}E_G(x)$. Then
\[
Y_x
\xrightarrow[x\to\infty]{\sigma(G)\text{-stably}}
\sqrt{C(G)}Z,
\]
where $Z$ is a standard Gaussian independent of $G$.
This follows from the weighted local-model moment
analysis in the proof of
\propref{prop:finite-coordinate-weighted-moments}
and the conditional Gaussian identification in
\propref{prop:conditional-gaussian-identification},
using \lemref{lem:finite-coordinate-haar-structure}
to extend the convergence to arbitrary bounded
measurable weights.
\label{rem:local-stable-convergence}
\end{remark}

\subsubsection{Fourier analysis of the moments of the local model}

Having reduced the moments of $E_P$ to those of $E_g$,
we now focus on the $\Phi\equiv1$ case of \propref{prop:finite-coordinate-weighted-moments}.
To obtain asymptotics for these moments, we introduce a
Fourier-side representation of the centered sums
$\Delta_g(k;x)$, inspired by the one in
GMR \cite[Proposition 2.1]{key-14}.
Put $e(t):=e^{2\pi it}$ and, for $q\ge1$, define
\[
\mathcal R_g(q):=
\{\nu\bmod q:g_p(\nu)\equiv0\bmod p^2
\text{ for every }p\mid q\},
\quad
E_x(t):=\sum_{n\le x}e(nt),
\quad
R_g(\alpha,q):=
\sum_{\nu\in\mathcal R_g(q)}e\left(\frac{\alpha\nu}{q}\right),
\]
where $q$ is the square of a squarefree integer.
By orthogonality of additive characters, for every
squarefree $k$ we have the expansion
\[
\Delta_g(k;x)
=\frac{\mu^2(k)}{k^2}
\sum_{\ell=1}^{k^2-1}
E_x\left(\frac{\ell}{k^2}\right)
R_g(-\ell,k^2).
\]
By \eqref{eq:intro-polynomial-local-comparison}, it remains
to study the moments of $E_g(x)$. Recall that
$E_g(x)=\sum_{k\ge1}\mu(k)\Delta_g(k;x)$.
Raising this identity to the $m$th power and averaging
over $g$, we restrict the resulting sum to tuples
$\mathbf k=(k_1,\ldots,k_m)$ with
$K:=[k_1,\ldots,k_m]\le z$, obtaining the expression
\[
\mathcal L_{m,z}(x)
:=\frac1{\zeta(2d+2)}
\sum_{\substack{k_1,\ldots,k_m\ge1\\K\le z}}
\prod_{i=1}^m\mu(k_i)
\int_{V_\square^\sharp}
\prod_{i=1}^m\Delta_g(k_i;x)\,dg.
\]
In \secref{sec:local-harmonic-expansion}, we substitute
the Fourier expansion of $\Delta_g(k_i;x)$ into this
expression. Expanding the product and interchanging
the finite sums with the integral gives
$\mathcal L_{m,z}(x)=\mathcal M_{m,z}(x)$, where
\begin{equation}\label{eq:MathcalMP_m,z}
\mathcal M_{m,z}(x)
:=\sum_{\substack{k_1,\ldots,k_m\ge1\\K\le z}}
\prod_{i=1}^m\frac{\mu(k_i)}{k_i^2}
\sum_{\substack{0<\ell_i<k_i^2\\1\le i\le m}}
\prod_{i=1}^m E_x\left(\frac{\ell_i}{k_i^2}\right)
\mathfrak A_{\mathbf k}(\boldsymbol\ell),
\end{equation}
and
\begin{equation}
\mathfrak A_{\mathbf k}(\boldsymbol\ell)
:=\frac1{\zeta(2d+2)}
\int_{V_\square^\sharp}
\prod_{i=1}^m R_g(-\ell_i,k_i^2)\,dg.
\label{eq:ArithmeticFactor}
\end{equation}
Thus the $m$th moment is reduced to a sum over
$m$-tuples of rational frequencies $\ell_i/k_i^2$,
with the arithmetic factor
$\mathfrak A_{\mathbf k}(\boldsymbol\ell)$
arising from the averaging over $g$.
Since the distribution of $g$ is invariant under
translations $g(t)\mapsto g(t+a)$, the arithmetic
factor $\mathfrak A_{\mathbf k}(\boldsymbol\ell)$
vanishes unless
\[
\frac{\ell_1}{k_1^2}+\cdots+
\frac{\ell_m}{k_m^2}\in\mathbb Z,
\]
which we refer to as the \emph{zero-sum} relation.
By \propref{prop:local-model-moment-representation}
and \eqref{eq:intro-polynomial-local-comparison},
we therefore have that for $H\ge x^{C_dm^2}$
\[
\mathrm{MP}_m(x,H)
=\mathcal M_{m,z}(x)
+O_{m,d,\varepsilon}(x^m z^{-1+\varepsilon})
+o_{m,d}(x^{m/4}).
\]
This is the starting point for our analysis of the
moments.

\subsubsection{Non-paired contribution}
One expects that the main contribution to \eqref{eq:MathcalMP_m,z} comes from \textit{exactly paired} tuples: tuples where each $\ell_i/k_i^2$ is equal to the negative of $\ell_j/k_j^2$ for some $j\ne i$ (and the $\{i,j\}$ pairs form a partition of $\{1,2,\ldots,m\}$); such a contribution of course only exists when $m$ is even. In \secref{sec:nonpaired-frequency-analysis} we prove \propref{prop:nonpaired-frequency-power-saving}, which shows that non-exactly-paired tuples contribute negligibly to \eqref{eq:MathcalMP_m,z}. To do so, we recall that GMR dealt with the simpler sum
\begin{equation}\label{eq:MathcalMP_m,z2}
\sum_{\substack{k_1,\ldots,k_m\ge1\\K\le z}}
\prod_{i=1}^m\frac{\mu(k_i)}{k_i^2}
\sum_{\substack{0<\ell_i<k_i^2\ (1\le i\le m)\\
\sum_i\ell_i/k_i^2\in\mathbb Z}}
\prod_{i=1}^m E_x\left(\frac{\ell_i}{k_i^2}\right),
\end{equation}
and also bounded its non-paired contribution. Our argument has two parts. When the arithmetic factor $\mathfrak A_{\mathbf k}(\boldsymbol\ell)$ is small, a case we call `good', we reduce our desired bound to the bounds already established in GMR. When it is large, a case we call `bad', we exploit its prime-by-prime factorization to identify the corresponding local divisibility constraints on $\mathbf k,\boldsymbol\ell$, and then use a weighted higher-moment argument to recover the required saving. This is done in \lemref{lem:bad-frequency-weighted-sum-bound} and \lemref{lem:arithmetic-factor-higher-moment-bound}, which are combined in \lemref{lem:bad-frequency-power-saving}.

\subsubsection{A reformulation of the paired contribution}
\label{sec:reformulation_paired}

By the estimates of the preceding subsection, for a
suitable choice of $z$, it remains to evaluate the
contribution to \eqref{eq:MathcalMP_m,z} coming from
exactly paired frequencies. To this end, we exploit
the product structure of the indicator defining
$E_g(x)$ to compute its Fourier coefficients.
In \secref{sec:reduced-local-moments}, we first restrict
the products in \eqref{eq:local-discrepancy} to primes
$p\le y$, putting
\[
E_g^{(y)}(x):=
\sum_{n\le x}\mathds{1}_{g,y}(n)
-x\prod_{p\le y}
\left(1-\frac{\rho_{g_p}(p^2)}{p^2}\right),
\quad
\mathds{1}_{g,y}(n):=
\prod_{p\le y}
\mathds{1}_{g_p(n)\not\equiv0\bmod p^2}.
\]
The function $\mathds{1}_{g,y}$ is periodic modulo
$Q_y:=\prod_{p\le y}p^2$ and hence admits a finite
Fourier expansion. Moreover, for fixed $x$,
we show that $E_g^{(y)}(x)\to E_g(x)$ in every
fixed $L^q=L^q(dg)$ as $y\to\infty$.

By M\"obius inversion, we have
\[
E_g^{(y)}(x)=
\sum_{\substack{k\ge1\\p\mid k\Rightarrow p\le y}}
\mu(k)\Delta_g(k;x).
\]
Expanding the summands in Fourier space, raising
to the $m$th power and integrating over
$V_\square^\sharp$ with the factor
$\zeta(2d+2)^{-1}$, we recover all the terms in
\eqref{eq:MathcalMP_m,z} whenever $y\ge z$.
By \lemref{lem:paired-large-lcm-tail-local-model},
the contribution of the remaining exactly paired
tuples, with $K>z$, is
$O_{m,d,\varepsilon}(x^mz^{-1+\varepsilon})$,
uniformly in $y$. The same estimate justifies passing to the limit
$y\to\infty$ in the contribution of exactly paired
tuples to the $m$th moment of $E_g^{(y)}(x)$.

On the other hand, expanding $\mathds{1}_{g,y}$
directly in additive characters modulo $Q_y$,
we obtain, after cancellation of the zero-frequency term,
\[
E_g^{(y)}(x)=
\sum_{\substack{\sigma\in Q_y^{-1}\mathbb Z/\mathbb Z\\
\sigma\ne0}}
E_x(\sigma)\widehat{\mathds{1}}_{g,y}(\sigma),
\]
where $\widehat{\mathds{1}}_{g,y}(\sigma)$ denotes
the Fourier coefficient of $\mathds{1}_{g,y}$ at
frequency $\sigma$. Writing $\sigma=\nu/r$ with $(\nu,r)=1$, we have
$r\mid Q_y$. Since exact pairing depends only on
the rational frequencies, the paired contribution
is preserved under this rearrangement. In contrast
to the original frequencies $\ell_i/k_i^2$, the two
opposite frequencies $\sigma$ and $-\sigma$ now
have the same denominator $r$. This allows us to
organize the paired contribution according to the
reduced denominators, with each pair contributing
\[
|E_x(\sigma)|^2
|\widehat{\mathds{1}}_{g,y}(\sigma)|^2.
\]
By the Chinese remainder theorem,
$\widehat{\mathds{1}}_{g,y}(\sigma)$ factors over
the primes $p\le y$. This factorization, together
with the independence of the coordinates $g_p$,
allows us to evaluate the paired contribution.

\subsubsection{Martingale structure of the paired contribution}

For $m=2k$, fix a pairing of the frequency indices.
To sum the contribution of each pair over its common
reduced denominator, put
\begin{equation}
\mathscr R:=
\{ab^2:a,b\text{ squarefree},\ (a,b)=1\}.
\label{eq:R-sf-decomposition}
\end{equation}
Note that the divisors of $Q_y$ exhaust $\mathscr R$
as $y\to\infty$. As shown in
\secref{sec:exact-frequency-pairings}, summing the
contribution of a single pair of opposite frequencies
over all nonzero frequencies gives, in the limit
$y\to\infty$, the non-negative quantity
\[
S(g,x):=
\sum_{\substack{r\in\mathscr R\\r>1}}
\sum_{\sigma\in\mathcal R(r)}
|E_x(\sigma)|^2|\Lambda(g,\sigma,r)|^2,
\]
where $\mathcal R(r):=\{\nu/r\bmod1:(\nu,r)=1\}$,
and $\operatorname{rad}(r):=\prod_{p\mid r}p$
denotes the radical of $r$.
Here $\Lambda(g,\sigma,r)$ is the convergent product of local
Fourier coefficients defined in
\secref{sec:finite-prime-fourier-expansion}, satisfying
$\Lambda(g,-\sigma,r)=\overline{\Lambda(g,\sigma,r)}$.
Thus, apart from negligible contributions from distinct pairs
sharing a denominator, a fixed pairing contributes essentially
$S(g,x)^k$. Summing over all pairings introduces the factor
$\mu_{2k}=(2k-1)!!$. Since the $2k$th moment is normalized by $x^{k/2}$,
the main task is therefore to understand $x^{-1/2}S(g,x)$. An elementary Fourier identity for  $E_x(\sigma)$  (used already in GMR) allows us to rewrite $S(g,x)$ in terms of non-negative coefficients $c_{g,\lambda}(r)$ as
\[
S(g,x)
=
2\sum_{\lambda\geq1}\lambda^{-2}
\sum_{\substack{r\in\mathscr R\\(r,\lambda)=1}}
c_{g,\lambda}(r)
f\!\left(\frac{r}{x\lambda}\right),
\]
for an explicit smooth (and deterministic) weight $f$, see \lemref{lem:one-pair-sum-decomposition}. By partial summation, it is therefore enough for the moment calculation to understand
\[
A_{g,\lambda}(Y)
:=
\sum_{\substack{r\leq Y\\(r,\lambda)=1}}
c_{g,\lambda}(r).
\]
Conveniently, the coefficients $c_{g,\lambda}(r)$ admit a
prime-by-prime factorization, although they are not
multiplicative in $r$.
This factorization, together with the independence of
the local coordinates $(g_p)_p$, makes it natural to
reveal these coordinates successively. With $\mathcal F_z$ as in \eqref{eq:sigmaAlgbra}, put
\[
A_{g,\lambda}^{(z)}(Y)
:=
\mathbb E\!\left[A_{g,\lambda}(Y)\mid\mathcal F_z\right].
\]
As $z$ increases, these conditional expectations form a martingale. For fixed $z$, independence of the remaining prime coordinates reduces $A_{g,\lambda}^{(z)}(Y)$ to finitely many multiplicative sums in arithmetic progressions, from which we obtain a limit for
$Y^{-1/2}A_{g,\lambda}^{(z)}(Y)$. Uniform moment estimates for the local coefficients, together with standard results in the theory of martingales, show that the contribution of the primes $p>z$ tends to zero uniformly in $Y$. Consequently, for each \(\lambda\ge1\) there exists a non-negative random variable \(D_g(\lambda)\) such that, as $Y \rightarrow \infty$
\[
Y^{-1/2}A_{g,\lambda}(Y)
\longrightarrow
D_g(\lambda)
\]
in \(L^q\) for every $q$.
A weighted summation argument then transfers these partial-sum limits back to $S(g,x)$ and gives
\[
x^{-1/2}S(g,x)\longrightarrow C(g)
\quad\text{in }L^q
\]
for every fixed $q$, where
\[
C(g)
:=
\frac1\pi
\sum_{\lambda\geq1}
\frac{D_g(\lambda)}{\lambda^{3/2}}.
\]
This constructs the random variance appearing in \thmref{thm:gaussian-mixture-limit}. Since there are $\mu_{2k}$ pairings, while repeated pair-denominators are negligible, the paired contribution to the $2k$th moment is
\[
\mu_{2k}x^{k/2}
\int_{V_{\square}}C(g)^k\,dg
+
o\!\left(x^{k/2}\right),
\]
whereas odd moments have no exactly paired contribution.

\subsubsection{Conclusion of the proof}

Thus far, our local-model analysis yields the asymptotics of all fixed ordinary moments. The comparison in \propref{prop:polynomial-local-moment-comparison} already includes bounded $\mathcal F_w$-measurable weights. To obtain \propref{prop:finite-coordinate-weighted-moments},
it therefore remains to adapt the local Fourier analysis
to the case where finitely many coordinates $g_p$ are fixed.
The general weighted statement then follows by linearity. The analysis is largely unchanged, and this gives
\begin{equation}
\lim_{x\rightarrow\infty}\mathbb{E}_H\left[\Phi\left(G_{H}\right)X_{x,H}^{m}\right]=\mu_{m}\int_{V_{\square}}\Phi\left(g\right)C\left(g\right)^{\frac{m}{2}}dg;
\label{eq:weighted-moment-identity}
\end{equation}
the few changes are explained in detail in \secref{sec:proof-weighted-moment-asymptotics}.

The condition $x=H^{o(1)}$ in \thmref{thm:gaussian-mixture-limit} ensures that, for every fixed $m$, the required relation $H\geq x^{C_dm^2}$ eventually holds.

We now briefly explain the probabilistic argument in \secref{sec:subsequential-conditional-moments}--\secref{sec:proof-main-theorem}, which identifies the limiting distribution of $X_{x,H}$ given \eqref{eq:weighted-moment-identity}. First, the local reduction vector $G_H$ converges to Haar-distributed
$G\in V_{\square}$. Suppose first $(G_H,X_{x,H})$ has a subsequential limit $(G,X)$. Since functions of finitely many prime coordinates generate $\sigma(G)$, it follows from \eqref{eq:weighted-moment-identity} that
\[
\mathbb E[X^{2k}\mid G]
=
\mu_{2k}C(G)^k,
\quad
\mathbb E[X^{2k+1}\mid G]
=
0.
\]
Thus, conditionally on $G$, the moments of $X$ are exactly those of
$N(0,C(G))$. Unlike the unconditional Gaussian mixture, this Gaussian law is determined by its moments, so
\[
X\mid G\sim N(0,C(G)),
\]
as explained in \propref{prop:conditional-gaussian-identification}. Moreover, since $V_{\square}$ is compact,
\propref{prop:finite-coordinate-weighted-moments}
with $m=2$ and $\Phi\equiv1$ implies tightness of
$(G_H,X_{x,H})$.
The uniqueness of the conditional law, together with
Prokhorov's theorem, then implies
\[
(G_H,X_{x,H})
\overset{d}{\longrightarrow}
\bigl(G,\sqrt{C(G)}\,Z\bigr),
\]
where $Z\sim N(0,1)$ is independent of $G$, completing the proof of \thmref{thm:gaussian-mixture-limit}.

\section{Reduction of the polynomial moments to the local model}\label{sec:polynomial-local-comparison}
We retain the notation introduced earlier. Throughout, we write
$[n_1,\ldots,n_r]$ and $(n_1,\ldots,n_r)$ for the least common
multiple and greatest common divisor, respectively, of the
positive integers $n_1,\ldots,n_r$, and set
$[n]:=\{1,\ldots,n\}$.

We begin by establishing truncation estimates for the
polynomial moments and their local-model counterparts
in \secref{sec:local-model-comparison}--\secref{sec:local-model-mobius-and-tail}.
In \secref{sec:proof-residue-class-reduction}, we use those estimation to derive \propref{prop:polynomial-local-moment-comparison}.

\subsection{The local model}
\label{sec:local-model-comparison}
Recall $V_\square^\sharp$ from
\secref{sec:local-model-subsubsection} and that its Haar mass in $V_\square$
is $\zeta(2d+2)^{-1}$. Throughout the paper, $dg$ on
$V_\square^\sharp$ denotes the normalized product probability
measure for which the coordinates $g_p$ are independent and uniform
on $V_{p^2}^\sharp$.

For $g=(g_p)_p\in V_\square^\sharp$, put
\[
\mathds{1}_g(n):=\prod_p\mathds{1}_{g_p(n)\not\equiv0\bmod p^2},
\quad
\mathfrak S(g):=\prod_p\left(1-\frac{\rho_{g_p}(p^2)}{p^2}\right),
\]
and recall the local discrepancy
\begin{equation}\label{eq:Egdef}
E_g(x):=\sum_{n\le x}\mathds{1}_g(n)-x\mathfrak S(g).
\end{equation}
We extend $E_g(x)$ by zero to $V_\square\setminus V_\square^\sharp$.
For squarefree $k$, define
\[
\rho_g(k^2):=\prod_{p\mid k}\rho_{g_p}(p^2),
\quad
\Delta_g(k;x):=
\sum_{\substack{n\le x\\g_p(n)\equiv0\bmod p^2\ \forall p\mid k}}1
-x\frac{\rho_g(k^2)}{k^2}.
\]
For a bounded function $\Phi$ on $V_\square$ and $z\ge1$, put
\begin{equation}\label{eq:weighted-local-lcm-sum}
\mathcal L_{m,z}^{\Phi}(x):=
\frac1{\zeta(2d+2)}
\sum_{\substack{k_1,\ldots,k_m\ge1\\K\le z}}
\prod_{i=1}^m\mu(k_i)
\int_{V_\square^\sharp}\Phi(g)\prod_{i=1}^m\Delta_g(k_i;x)\,dg.
\end{equation}
Here and below, divisor sums with M\"obius factors are understood to
be supported on squarefree integers. We write
$\mathcal L_{m,z}(x):=\mathcal L_{m,z}^{1}(x)$.

\begin{proposition}\label{prop:polynomial-local-moment-comparison}
Fix integers $m\ge1$ and $w\ge2$, and let $\Phi$ be a bounded
$\mathcal F_w$-measurable function on $V_\square$.
There exists $C_d>0$, depending only on $d$, such that, whenever
$H\ge x^{C_dm^2}$,
\begin{equation}\label{eq:direct-weighted-moment-comparison}
\mathbb E_H[\Phi(G_H)E_P(x)^m]
=
\frac1{\zeta(2d+2)}\int_{V_\square^\sharp}\Phi(g)E_g(x)^m\,dg
+o_{m,d,w}\!\left(\|\Phi\|_\infty x^{m/4}\right).
\end{equation}
In particular,
\[
\mathrm{MP}_m(x,H)
=\frac1{\zeta(2d+2)}\int_{V_\square^\sharp}E_g(x)^m\,dg
+o_{m,d}(x^{m/4}).
\]
\end{proposition}
Once it is proved, the remaining
moment analysis can be carried out entirely in the local model.

\subsection{Discarding bad polynomials}\label{sec:discarding-bad-polynomials}
We commence by discarding a subset of $\mathscr{P}_d(H)$
that is negligible but presents technical difficulties.

For $P(t)=\sum_{i=0}^{d}a_{i}t^{i}\in\mathbb{Z}[t]$, we write $\mathscr{C}(P):=\gcd(a_{0},...,a_{d})$ for the content of $P$ and $\mathrm{Disc}(P)$ for its discriminant. We adopt the convention $\mu^{2}(0)=0$ and put\begin{align*}
 & \mathscr{P}_{d}^{\sharp}\left(H\right):=\left\{ P\in\mathscr{P}_{d}\left(H\right):p^{2}\nmid\mathscr{C}\left(P\right)\ \text{for every prime }p\right\} ,\\
 & \mathscr{P}_{d}^{\circ}\left(H;x\right):=\left\{ P\in\mathscr{P}_{d}\left(H\right):P\left(n\right)\neq0\ \text{for every }1\leq n\leq x\right\} ,\\
 & \mathscr{P}_{d}^{\mathrm{sep}}\left(H\right):=\left\{ P\in\mathscr{P}_{d}\left(H\right):\mathrm{Disc}\left(P\right)\neq0\right\} ,\\
 & \mathcal{F}_{d}\left(H,x\right):=\mathscr{P}_{d}^{\sharp}\left(H\right)\cap\mathscr{P}_{d}^{\circ}\left(H;x\right)\cap\mathscr{P}_{d}^{\mathrm{sep}}\left(H\right),
\end{align*}where $\mathrm{sep}$ stands for separable. It is enough to restrict our count to $\mathcal{F}_{d}\left(H,x\right)$, as the following lemma shows. For $\deg P=1$ we use the convention $\mathrm{Disc}\left(P\right)=1$ so that $\mathscr{P}_{1}^{\mathrm{sep}}\left(H\right)=\mathscr{P}_{1}\left(H\right)$. 
\begin{lemma}\label{lem:exceptional-polynomials-negligible}
If $P\notin\mathscr{P}_{d}^{\sharp}\left(H\right)$, then $E_{P}\left(x\right)=0$ for every $x\geq1$. Moreover,\[
\#\left(\mathscr{P}_{d}\left(H\right)\setminus\mathscr{P}_{d}^{\circ}\left(H;x\right)\right)\ll_{d}H^{d}\log\left(2x\right)+xH^{d-1},
\]and\[
\#\left(\mathscr{P}_{d}\left(H\right)\setminus\mathscr{P}_{d}^{\mathrm{sep}}\left(H\right)\right)\ll_{d}H^{d}.
\]Consequently,\[
\frac{1}{\left(2H\right)^{d+1}}\sum_{P\notin\mathscr{P}_{d}^{\circ}\left(H;x\right)\cap\mathscr{P}_{d}^{\mathrm{sep}}\left(H\right)}\left|E_{P}\left(x\right)\right|^{m}\ll_{m,d}x^{m}\left(\frac{\log\left(2x\right)}{H}+\frac{x}{H^{2}}\right).
\]
\end{lemma}

\begin{proof}
If $p^{2}\mid\mathscr{C}\left(P\right)$ for some prime $p$, then every value of $P$ is divisible by $p^{2}$ and the $p$-factor of $\mathfrak{S}_{P}$ is zero. Thus both terms defining $E_{P}\left(x\right)$ vanish, and the first assertion follows. For the second assertion, fix $n\leq x$. We shall bound how many tuples $(a_0,\ldots,a_d)$ satisfy $P(n)=0$. Since $P(n)\equiv a_0 \bmod n$, we have a divisibility relation $n\mid a_0$. Further, given $a_0,a_2,\ldots,a_d$, there is at most one value of $a_1$ for which $P(n)=0$. Hence there are $O_{d}\left(H^{d}/n+H^{d-1}\right)$ such tuples. Summing over $n\le x$ gives the stated count. For the third assertion, observe that the discriminant is a nonzero polynomial in the $d+1$ coefficients, so the conclusion follows by Schwartz–Zippel. Finally, as $\left|E_{P}\left(x\right)\right|\leq2x$ since $0\leq\mathfrak{S}_{P}\leq1$, the last assertion follows from the earlier ones.
\end{proof}

\subsection{The polynomial large-lcm tail}
\label{sec:polynomial-mobius-and-tail}
For squarefree $k$, put
\[
\Delta_P(k;x):=
\sum_{\substack{n\le x\\P(n)\equiv0\bmod k^2}}1
-x\frac{\rho_P(k^2)}{k^2}.
\]
The first reduction is a physical-space identity.

\begin{lemma}\label{lem:polynomial-mobius-expansion}
For every $P\in\mathcal F_d(H,x)$,
\begin{equation}\label{eq:polynomial-mobius-expansion}
E_P(x)=\sum_{k\ge1}\mu(k)\Delta_P(k;x),
\end{equation}
and the series converges absolutely.
\end{lemma}

\begin{proof}
Since $P(n)\ne0$ for $n\le x$, M\"obius inversion gives
\[
\mu^2(P(n))=\sum_{k^2\mid P(n)}\mu(k).
\]
Also, since $P$ is separable, the Euler product $\mathfrak S_P$
converges absolutely and equals
$\sum_{k\ge1}\mu(k)\rho_P(k^2)/k^2$.
Consequently,
\[
\begin{aligned}
E_P(x)
&=\sum_{n\le x}\sum_{k^2\mid P(n)}\mu(k)
-x\sum_{k\ge1}\frac{\mu(k)\rho_P(k^2)}{k^2}\\
&=\sum_{k\ge1}\mu(k)
\left(\sum_{\substack{n\le x\\k^2\mid P(n)}}1
-x\frac{\rho_P(k^2)}{k^2}\right).
\end{aligned}
\]
The first double sum has finite support in $k$, and the second
series is absolutely convergent. This proves both assertions.
\end{proof}

Raising \eqref{eq:polynomial-mobius-expansion} to the $m$th
power produces products $\prod_i\Delta_P(k_i;x)$.
We will truncate this sum according to
$K=[k_1,\ldots,k_m]$. To control the discarded
terms, put
\[
A_P(k;n):=\mu^2(k)
\left(
\mathds{1}_{P(n)\equiv0\bmod k^2}
+\frac{\rho_P(k^2)}{k^2}
\right).
\]
By the definition of $\Delta_P(k;x)$, we have
\begin{equation}
\label{eq:delta_P_majorant}
|\mu(k)\Delta_P(k;x)|
\le \sum_{n\le x}A_P(k;n).    
\end{equation}
Consequently, the contribution from tuples with $K>z$ to the $m$th
is bounded in absolute value by
\[
\sum_{\mathbf n\in[x]^m}
\sum_{\substack{k_1,\ldots,k_m\ge1\\K>z}}
\sum_{P\in\mathcal F_d(H,x)}
\prod_{i=1}^m A_P(k_i;n_i).
\]
The following bound, uniform in $\mathbf n$, is the
polynomial analogue of Lemma 2.5 of \cite{key-14}.
For an integer $n\ge1$, we write $n\mid P$ to mean
that $n$ divides every coefficient of $P$.

\begin{lemma}\label{lem:large-lcm-tail-polynomial-average}
Let $\mathbf n=(n_1,\ldots,n_m)\in[x]^m$.
Uniformly in $\mathbf n$ and $z>0$, we have
for every $\varepsilon>0$

\[
\sum_{\substack{k_{1},...,k_{m}\ge1\\
K>z
}
}\sum_{P\in\mathcal{F}_{d}\left(H,x\right)}
\prod_{i=1}^{m}A_{P}\left(k_{i};n_{i}\right)
\ll_{m,d,\varepsilon}
\frac{H^{d+1+\varepsilon}x^{\varepsilon}}{z^{1-\varepsilon}}
+H^{d+\frac{2}{3}+\varepsilon}x^{\varepsilon}
+H^{d+\frac{1}{2}+\varepsilon}x^{\frac{d}{2}+\varepsilon}.
\]
\end{lemma}

\begin{proof}
Understand every $P$-sum in this proof to be over $\mathcal{F}_{d}\left(H,x\right)$ unless another family is displayed explicitly. We first record a tail estimate needed in the base case $m=1$. For every fixed $d\geq1$ and every $\varepsilon>0$, uniformly for $P\in\mathscr{P}_{d}^{\sharp}\left(H\right)\cap\mathscr{P}_{d}^{\mathrm{sep}}\left(H\right)$ and $z>0$, we now show
\begin{equation}\label{eq:rootdensity}
\sum_{k>z}\frac{\mu^{2}\left(k\right)\rho_{P}\left(k^{2}\right)}{k^{2}}\ll_{d,\varepsilon}\frac{H^{\varepsilon}}{z^{1-\varepsilon}}.
\end{equation}
Let
\[
B_{P}\left(Y\right):=\sum_{k\leq Y}\mu^{2}\left(k\right)\rho_{P}\left(k^{2}\right).
\]
Lemma 2.2 of \cite{key-1-1} gives, for every fixed $0<\eta<1$ and $d\geq 1$,
\[
B_{P}\left(Y\right)\ll_{d,\eta}\left(3d\right)^{2\omega\left(\mathrm{Disc}\left(P\right)\right)}Y^{1+\eta}\sum_{r^{2}\mid P}r
\]
where $\omega\left(n\right):=\sum_{p\mid n}1$. Since $P\in\mathscr{P}_{d}^{\sharp}\left(H\right)$, only $r=1$ contributes. Since also $\mathrm{Disc}\left(P\right)\neq0$, the standard estimate $\left|\mathrm{Disc}\left(P\right)\right|\ll_{d}H^{2d-2}$ holds, and the divisor bound $A^{\omega\left(N\right)}\ll_{A,\varepsilon}N^{\varepsilon}$ yields $B_{P}\left(Y\right)\ll_{d,\varepsilon}H^{\varepsilon}Y^{1+\varepsilon}$ for every $\varepsilon>0$. Partial summation therefore gives \eqref{eq:rootdensity}. Summing \eqref{eq:rootdensity} over $P$, we obtain for every $d\geq1$,
\[
\sum_{P}\sum_{k>z}\frac{\mu^{2}\left(k\right)\rho_{P}\left(k^{2}\right)}{k^{2}}\ll_{d,\varepsilon}\frac{H^{d+1+\varepsilon}}{z^{1-\varepsilon}}.
\]
We now prove the lemma by induction on $m$. For $m=1$, the above treats the contribution of $\rho_P(k^2)/k^2$ to the sum. For the indicator part,
\begin{equation}\label{eq:kbnd}
k^{2}\leq\left|P\left(n\right)\right|\ll_{d}Hx^{d},
\end{equation}
holds when $P(n)\equiv 0 \bmod k^2$, while, after fixing all the coefficients except the constant coefficient,
\begin{equation}\label{eq:Pfix}
\sum_{P}\mathds{1}_{P\left(n\right)\equiv0\bmod k^{2}}\ll_{d} H^d \left( \frac{H}{k^2}+1\right).
\end{equation}
Consequently, from \eqref{eq:kbnd} and \eqref{eq:Pfix},
\[
\sum_{k>z}\sum_{P}\mu^{2}\left(k\right)\mathds{1}_{P\left(n\right)\equiv0\bmod k^{2}}\ll_{d}\frac{H^{d+1}}{z}+H^{d+\frac{1}{2}}x^{\frac{d}{2}},
\]
which is bounded by the asserted right-hand side. We proceed to assume the lemma holds for $m-1$ and will demonstrate it for $m$. Put
\[
K_{j}:=\left[k_{1},\ldots,k_{j-1},k_{j+1},\ldots,k_{m}\right],\quad Z:=H^{\frac{1}{3}}.
\]
Applying \eqref{eq:rootdensity} with
$z=1/2$, together with \eqref{eq:kbnd} and the divisor bound,
we obtain the uniform estimate
\[
\sum_{k\geq1}A_P(k;n)
\leq
\sum_{k^2\mid P(n)}\mu^2(k)
+
\sum_{k\geq1}
\frac{\mu^2(k)\rho_P(k^2)}{k^2}
\ll_{d,\varepsilon}(Hx^d)^\varepsilon
\]
for all $P$ and $n\leq x$.
Now, if $K_j>Z$ for some $j$, we first sum over $k_j$
using this bound. We then discard the condition $K>z$
and apply the induction hypothesis to the remaining
$m-1$ variables with cutoff $Z$. Summing over the possible choices of $j$, this contribution is
\[
\begin{aligned}
&\ll_{m,d,\varepsilon}
(Hx^d)^\varepsilon
\left(
\frac{H^{d+1+\varepsilon}x^\varepsilon}{Z^{1-\varepsilon}}
+H^{d+\frac23+\varepsilon}x^\varepsilon
+H^{d+\frac12+\varepsilon}x^{\frac d2+\varepsilon}
\right)\\
&\ll_{m,d,\varepsilon}
H^{d+\frac23+\varepsilon}x^\varepsilon
+H^{d+\frac12+\varepsilon}x^{\frac d2+\varepsilon}.
\end{aligned}
\]
It remains to consider tuples for which $K_{j}\leq Z$ for every $j$. Comparing prime exponents implies
\[
K\leq\prod_{j=1}^{m}K_{j}^{\frac{1}{m-1}}\leq Z^{\frac{m}{m-1}}.
\]
We expand the $i$-product defining the $A_{P}$'s, obtaining a sum over subsets $I\subseteq\left[m\right]$:
\[
\sum_{I \subseteq [m]}\sum_{P}\prod_{i\in I}\mathds{1}_{P\left(n_{i}\right)\equiv0\bmod k_{i}^{2}}\prod_{j\notin I}\frac{\rho_{P}\left(k_{j}^{2}\right)}{k_{j}^{2}} = \sum_{I \subseteq [m]}\sum_{P}\prod_{i\in I}\mathds{1}_{P\left(n_{i}\right)\equiv0\bmod k_{i}^{2}}\prod_{j\notin I}
\frac{\sum_{y_{j}\bmod k_{j}^{2}}\mathds{1}_{P\left(y_{j}\right)\equiv0\bmod k_{j}^{2}}}{k_{j}^{2}}.
\]
For fixed residue variables $(y_j)_{j\not\in I}$, every prime $p\mid K$ occurs in at least one of the displayed congruences. Select one such congruence for each $p$. After the coefficients $a_{1},\ldots,a_{d}$ have been fixed, these selected conditions determine the constant coefficient $a_0$ modulo $p^{2}$ for every $p\mid K$, and hence modulo $K^{2}$ by the Chinese remainder theorem. Additional congruences can only reduce the count. Averaging over the residue variables therefore gives the bound
\[
\sum_{I \subseteq [m]}\sum_{P}\prod_{i\in I}\mathds{1}_{P\left(n_{i}\right)\equiv0\bmod k_{i}^{2}}\prod_{j\notin I}\frac{\rho_{P}\left(k_{j}^{2}\right)}{k_{j}^{2}}\ll_{m,d}\frac{H^{d+1}}{K^{2}}+H^{d}.
\]
For squarefree $K$, the number of squarefree $m$-tuples $(k_1,\ldots,k_m)$ with lcm $K$ is $\ll_{m,\varepsilon}K^{\varepsilon}$ by the divisor bound. Hence the remaining contribution, in which $K>z$ and $K\leq Z^{\frac{m}{m-1}}$, is
\[
\begin{aligned}
&\ll_{m,d,\varepsilon}H^{d+1}\sum_{K>z}\frac{K^{\varepsilon}}{K^{2}}+H^{d}\sum_{K\leq Z^{\frac{m}{m-1}}}K^{\varepsilon}\\
&\ll_{m,d,\varepsilon}\frac{H^{d+1+\varepsilon}}{z^{1-\varepsilon}}+H^{d+\varepsilon}Z^{\frac{m}{m-1}+\varepsilon}\\
&\ll_{m,d,\varepsilon}\frac{H^{d+1+\varepsilon}}{z^{1-\varepsilon}}+H^{d+\frac{2}{3}+\varepsilon},
\end{aligned}
\]
where we used $m\ge 2$, so that $\frac{m}{m-1}\leq2$. This is absorbed by the asserted bound and completes the induction.
\end{proof}

\subsection{The large-lcm tails in the local model}
\label{sec:local-model-mobius-and-tail}
We next justify the corresponding restriction to tuples
with $K\le z$ in the weighted moments of $E_g(x)$.
To this end, recall the finite-prime approximation
introduced in \secref{sec:reformulation_paired}: for $y\ge2$ and
$g=(g_p)_p\in V_\square^\sharp$, put
\begin{equation}\label{eq:E_g-truncated}
E_g^{(y)}(x):=
\sum_{n\le x}\mathds{1}_{g,y}(n)
-x\prod_{p\le y}
\left(1-\frac{\rho_{g_p}(p^2)}{p^2}\right),
\end{equation}
where
\[
\mathds{1}_{g,y}(n):=
\prod_{p\le y}
\mathds{1}_{g_p(n)\not\equiv0\bmod p^2}.
\]
\begin{lemma}\label{lem:local-model-mobius-expansion}
For every fixed $x$, the products defining $\mathds{1}_g(n)$ and
$\mathfrak S(g)$ converge almost surely. Moreover,
\[
E_g^{(y)}(x)
\xrightarrow[y\to\infty]{}
E_g(x)
\]
almost surely and in $L^m$ for every fixed $m\ge1$.
\end{lemma}

\begin{proof}
For each prime $p$, the coordinate $g_p$ is uniformly distributed on
$V_{p^2}^{\sharp}$. For fixed $n$, surjectivity of evaluation at $n$ on
$V_{p^2}$ gives
\[
\mathbb P\bigl(g_p(n)\equiv0\bmod p^2\bigr)
=
\frac{p^{2d}-1}{p^{2d+2}-1}
\ll_d p^{-2}.
\]
Since $\sum_p p^{-2}<\infty$, Borel--Cantelli implies that, almost
surely, only finitely many primes satisfy
$g_p(n)\equiv0\bmod p^2$. Hence $\mathds{1}_g(n)$ is well-defined
almost surely, simultaneously for all $n\le x$.
Similarly, Fubini--Tonelli gives
\[
\mathbb E
\sum_p\frac{\rho_{g_p}(p^2)}{p^2}
=
\sum_p\frac1{p^2}
\sum_{\nu\bmod p^2}
\mathbb P\bigl(g_p(\nu)\equiv0\bmod p^2\bigr)
\ll_d
\sum_p p^{-2}
<\infty.
\]
Thus $\sum_p \rho_{g_p}(p^2)/p^2<\infty$
almost surely, and therefore, the Euler product defining $\mathfrak S(g)$
converges almost surely.

By finite inclusion--exclusion,
\begin{equation} \label{E_gP_y}
E_g^{(y)}(x)
=
\sum_{\substack{k\geq1\\
p\mid k\Rightarrow p\le y}}
\mu(k)\Delta_g(k;x).
\end{equation}
We now show that $E_g^{(y)}(x)$ converges almost surely to $E_g(x)$.
By the almost-sure convergence established above, we may fix $g$ for
which $\mathds{1}_g(n)$ is defined for every $n\leq x$ and
$\mathfrak{S}(g)$ converges. For each fixed $n\leq x$,
\[
\mathds{1}_{g,y}(n)
=
\prod_{p\le y}
\mathds{1}_{g_p(n)\not\equiv0\bmod p^2}
\xrightarrow[y\to\infty]{\mathrm{a.s.}}
\prod_p
\mathds{1}_{g_p(n)\not\equiv0\bmod p^2}
=
\mathds{1}_g(n).
\]
Indeed, if $g_p(n)\equiv0\bmod p^2$ for some prime $p$, then once this
prime is $\le y$, the finite product is identically zero; if no
such prime exists, all the finite products are equal to one. Since
there are only finitely many $n\leq x$, it follows that
\[
\sum_{n\leq x}\mathds{1}_{g,y}(n)
\xrightarrow[y\to\infty]{\mathrm{a.s.}}
\sum_{n\leq x}\mathds{1}_g(n).
\]
Moreover, by the definition of the convergent Euler product,
\[
\prod_{p\le y}
\left(
1-\frac{\rho_{g_p}(p^2)}{p^2}
\right)
\xrightarrow[y\to\infty]{\mathrm{a.s.}}
\prod_p
\left(
1-\frac{\rho_{g_p}(p^2)}{p^2}
\right)
=
\mathfrak{S}(g).
\]
Consequently, as $y\rightarrow \infty$, $E_g^{(y)}(x)
\xrightarrow{\mathrm{a.s.}}
E_g(x).$
Finally, $\lvert E_g^{(y)}(x)\rvert\leq2x$ uniformly in $y$ and $g$.
Dominated convergence therefore gives
\[
E_g^{(y)}(x)
\xrightarrow[y\to\infty]{L^m}
E_g(x)
\]
for every fixed $m\geq1$. By \eqref{E_gP_y}, we conclude that
\[
E_g(x)
=
\lim_{y\to\infty}
\sum_{\substack{k\geq1\\
p\mid k\Rightarrow p\le y}}
\mu(k)\Delta_g(k;x),
\]
where the convergence holds both almost surely and in $L^m$.
\end{proof}

We now prove the local-model counterpart of \lemref{lem:large-lcm-tail-polynomial-average}. In the local setting the independence of the prime coordinates yields a cleaner bound, with no dependence on $H$.

\begin{lemma}\label{lem:large-lcm-tail-local-model}
Let $\mathbf{n}=\left(n_{1},\ldots,n_{m}\right)\in\mathbb{Z}^{m}$ and put\[
A_{g}\left(k;n\right):=\mu^{2}\left(k\right)\left(\mathds{1}_{\substack{g_{p}\left(n\right)\equiv0\bmod p^{2}\\
\forall p\mid k
}
}+\frac{\rho_{g}\left(k^{2}\right)}{k^{2}}\right).
\]For all $\varepsilon >0$, uniformly for $z\geq1$,\[
\sum_{K>z}\mathbb E_{g\in V_{\square}^{\sharp}}\prod_{i=1}^{m}A_{g}\left(k_{i};n_{i}\right)\ll_{m,d,\varepsilon}z^{-1+\varepsilon}.
\]
\end{lemma}

\begin{proof}
Express each factor $\rho_g(k^2)/k^2$ as the average of the root indicators over residue classes modulo $k^2$. For every $p\mid K$, at least one resulting local
condition has the form
\[
g_p(u)\equiv0\bmod p^2.
\]
Once the remaining coefficients of $g_p$ are fixed, such a condition
determines the constant coefficient modulo $p^2$. Hence the probability
of the local conditions at $p$ is $\ll_{m,d}p^{-2}$. Independence over
the primes therefore yields
\[
\mathbb E_{g\in V_{\square}^{\sharp}}
\prod_{i=1}^m A_g(k_i;n_i)
\le
\frac{C_{m,d}^{\omega(K)}}{K^2}.
\]
for some $C_{m,d}>0$. Since the number of squarefree $m$-tuples with
least common multiple $K$ is
$\ll_{m,\varepsilon}K^\varepsilon$, summing over $K>z$ proves the
claim.
\end{proof}

The restriction to primes $p\le y$ allows us to raise
the finite M\"obius expansion \eqref{E_gP_y} to the
$m$th power and integrate termwise. For $y\ge z$,
the terms with $K\le z$, after inserting a bounded
weight $\Phi(g)$ and integrating over
$V_\square^\sharp$ with the factor $\zeta(2d+2)^{-1}$,
give precisely $\mathcal L_{m,z}^{\Phi}(x)$.
Using
\[
|\mu(k)\Delta_g(k;x)|
\le \sum_{n\le x}A_g(k;n),
\]
\lemref{lem:large-lcm-tail-local-model} bounds the
contribution of the remaining tuples with $K>z$ by
$O_{m,d,\varepsilon}(\|\Phi\|_\infty x^mz^{-1+\varepsilon})$,
uniformly in $y$. Letting $y\to\infty$ with $x$ and
$z$ fixed, the $L^m$ convergence established in
\lemref{lem:local-model-mobius-expansion} yields
\begin{equation}\label{eq:weighted-local-lcm-truncation}
\frac1{\zeta(2d+2)}
\int_{V_\square^\sharp}\Phi(g)E_g(x)^m\,dg
=\mathcal L_{m,z}^{\Phi}(x)
+O_{m,d,\varepsilon}\!\left(
\|\Phi\|_\infty x^mz^{-1+\varepsilon}\right).
\end{equation}

\subsection{Proof of \propref{prop:polynomial-local-moment-comparison}}
\label{sec:proof-residue-class-reduction}
Having restricted both moment expansions to tuples with
$K\le z$, we use the equidistribution of polynomial
coefficients in residue classes to relate the polynomial
moments to their local-model counterparts.
For the square $q$ of a squarefree integer, put
\[
V_q^\sharp:=\{h\in V_q:h\not\equiv0\bmod p^2\text{ for every }p\mid q\},
\quad
\delta_d(q):=\sum_{\substack{a\ge1\\(a,q)=1}}
\frac{\mu(a)}{a^{2d+2}}
=\frac1{\zeta(2d+2)}\prod_{p\mid q}(1-p^{-2d-2})^{-1}.
\]
Thus
\begin{equation}\label{eq:sharp-residue-normalization}
\frac{\delta_d(q)}{q^{d+1}}
=\frac1{\zeta(2d+2)|V_q^\sharp|}
\end{equation}
is the normalization factor.
We use the following coefficient-equidistribution estimate.

\begin{lemma}\label{lem:polynomial-residue-equidistribution}
Let $q\le2H$ be the square of a squarefree integer. Uniformly for $g\in V_{q}$,\[
\#\left\{ P\in\mathscr{P}_{d}^{\sharp}\left(H\right):P\equiv g\bmod q\right\} =\mathds{1}_{g\in V_{q}^{\sharp}}\cdot\left(\frac{\left(2H\right)^{d+1}}{q^{d+1}}\delta_{d}\left(q\right)+O_{d}\left(\frac{H^{d}}{q^{d}}+H^{\frac{1}{2}}\right)\right).
\]
\end{lemma}

\begin{proof}
M\"obius inversion in the content gives\[
\mathds{1}_{P\in\mathscr{P}_{d}^{\sharp}\left(H\right)}=\sum_{a^{2}\mid\mathscr{C}\left(P\right)}\mu\left(a\right).
\]If $g\in V_{q}^{\sharp}$, only $\left(a,q\right)=1$ can occur, so writing $P=a^{2}Q\equiv g\bmod q$ is equivalent to $Q\equiv a^{-2}g\bmod q$. For a fixed coefficient residue class $q$ and $Y\geq1$, we have the uniform estimate\[
\#\left\{ Q\in\mathbb{Z}\left[t\right]:\deg Q=d,\ \left|Q\right|\leq Y,\ Q\equiv h\bmod q\right\} =\frac{\left(2Y\right)^{d+1}}{q^{d+1}}+O_{d}\left(\frac{Y^{d}}{q^{d}}+1\right).
\]The possible loss from the condition that the leading coefficient is nonzero is contained in the first error term. Summing the error with $Y=H/a^{2}$ over $a\leq \sqrt{\left\lvert\mathscr{C}\left(P\right)\right\rvert} \leq H^{\frac{1}{2}}$ gives\[
O_{d}\bigg(\frac{H^{d}}{q^{d}}\sum_{a\geq1}a^{-2d}+\sum_{a\leq H^{\frac{1}{2}}}1\bigg)=O_{d}\left(\frac{H^{d}}{q^{d}}+H^{\frac{1}{2}}\right).
\]Extending the main sum from $a\leq H^{\frac{1}{2}}$ to $a\ge 1$ costs $O_{d}\left(H^{\frac{1}{2}}/q^{d+1}\right)$ and is absorbed. This proves the formula for $g\in V_{q}^{\sharp}$. If $g\notin V_{q}^{\sharp}$, then every polynomial in that residue class has content divisible by $p^{2}$ for some $p\mid q$, and the count is 0.
\end{proof}
\begin{proof}[Proof of \propref{prop:polynomial-local-moment-comparison}]
Let $Z\ge1$ satisfy
$Q_wZ^2\le2H$. We may assume $\|\Phi\|_\infty\le1$, since the
general case follows by scaling. Define
\[
\mathrm{MP}_{m,Z}^{\Phi}(x,H):=
\frac1{(2H)^{d+1}}\sum_{P\in\mathcal F_d(H,x)}\Phi(G_H)
\sum_{\substack{k_1,\ldots,k_m\ge1\\K\le Z}}
\prod_{i=1}^m\mu(k_i)\Delta_P(k_i;x).
\]
Since $|\mathscr P_d(H)|=(2H)^{d+1}(1+O_d(H^{-1}))$ and
$|E_P(x)|\le2x$, \lemref{lem:exceptional-polynomials-negligible}
controls both the change of normalization and the removal of
exceptional polynomials. Applying
\lemref{lem:polynomial-mobius-expansion} and
\lemref{lem:large-lcm-tail-polynomial-average}, we obtain
\begin{align}
\mathbb E_H[\Phi(G_H)E_P(x)^m]
&=\mathrm{MP}_{m,Z}^{\Phi}(x,H)\notag\\
&\quad+O_{m,d,\varepsilon}\!\left(
\frac{x^{m+\varepsilon}H^\varepsilon}{Z^{1-\varepsilon}}
+H^{-1/3+\varepsilon}x^{m+\varepsilon}
+H^{-1/2+\varepsilon}x^{m+d/2+\varepsilon}
+x^m\left(\frac{\log(2x)}H+\frac{x}{H^2}\right)
\right).
\label{eq:weighted-polynomial-lcm-truncation}
\end{align}
Here absolute convergence justifies the expansion of the $m$th
power, and the tail bound remains valid with $\Phi$ inserted because
$|\Phi|\le1$.

For every polynomial $P$ and squarefree $k$, the elementary bound
$|\Delta_P(k;x)|\le2x$ holds. The number of squarefree $m$-tuples
with lcm $K$ is $\ll_{m,\varepsilon}K^\varepsilon$.
Thus, by \lemref{lem:exceptional-polynomials-negligible}, replacing
$\mathcal F_d(H,x)$ by $\mathscr P_d^\sharp(H)$ in
$\mathrm{MP}_{m,Z}^{\Phi}(x,H)$ costs
\begin{equation}\label{eq:weighted-truncated-exceptional-error}
O_{m,d,\varepsilon}\!\left(
x^m Z^{1+\varepsilon}
\left(\frac{\log(2x)}H+\frac{x}{H^2}\right)\right).
\end{equation}
Fix a squarefree tuple $\mathbf k$ with $K\le Z$, and put
$q:=[K^2,Q_w]$. The function
$\Phi(G_H)\prod_i\Delta_P(k_i;x)$ depends only on the coefficients
of $P$ modulo $q$. 
For $h\in V_q^\sharp$, define
\[
F_{\mathbf k}(h):=
\Phi(h\bmod Q_w)
\prod_{i=1}^m
\left(
\sum_{\substack{n\le x\\h(n)\equiv0\bmod k_i^2}}1
-\frac{x\rho_h(k_i^2)}{k_i^2}
\right).
\]
Since $Q_w,k_i^2\mid q$, this is well-defined
on coefficient residue classes modulo $q$.
If $P,g\equiv h\bmod q$ then
\[
\Phi(G_P)\prod_{i=1}^m\Delta_P(k_i;x)
=F_{\mathbf k}(h)
=\Phi(g)\prod_{i=1}^m\Delta_g(k_i;x).
\]
Thus both expressions depend only on the same
coefficient residue class modulo $q$.
To estimate the residue-class error, we use
the bound $|\Delta_P(k;x)|\le\rho_P(k^2)$. Indeed, recall the definition 
\[
\Delta_P(k;x)
=\sum_{\substack{\nu\bmod k^2\\P(\nu)\equiv0\bmod k^2}}
\left(\#\{n\le x:n\equiv\nu\bmod k^2\}-\frac{x}{k^2}\right).
\]
Each summand has absolute value at most $1$, so the triangle inequality gives the required bound. The same bound holds when replacing $P$ with $h$. Thus
\[
|F_{\mathbf k}(h)|\le\prod_{i=1}^m\rho_h(k_i^2).
\]
For $p\mid K$, put $t_p:=\#\{i:p\mid k_i\}$. By
\lemref{lem:root-count-moment-bounds},
\[
\mathbb E_{g\in V_{p^2}^\sharp}\rho_g(p^2)^{t_p}
\ll_{m,d}1+p^{t_p-2}\ll_{m,d}p^{t_p-1}.
\]
The Chinese remainder theorem therefore gives
\begin{equation}
\sum_{h\in V_q^\sharp}|F_{\mathbf k}(h)|
\le |V_q^\sharp|\prod_{p\mid K}
\mathbb E_{g\in V_{p^2}^\sharp}\rho_g(p^2)^{t_p}\ll_{m,d,\varepsilon}q^{d+1}K^\varepsilon
\prod_{p\mid K}p^{t_p-1}
=q^{d+1}K^{-1+\varepsilon}\prod_{i=1}^m k_i.
\label{eq:F_k_bound}  
\end{equation}
Primes dividing $q$ but not $K$ contribute exactly
$1$ to the normalized product, while the local implied
constants are absorbed using $C_{m,d}^{\omega(K)}
\ll_{m,d,\varepsilon}K^\varepsilon$.

Since $g\bmod q$ is uniformly distributed on $V_q^\sharp$,
\eqref{eq:sharp-residue-normalization} gives
\[
\frac{\delta_d(q)}{q^{d+1}}
\sum_{h\in V_q^\sharp}F_{\mathbf k}(h)
=
\frac1{\zeta(2d+2)}
\int_{V_\square^\sharp}\Phi(g)
\prod_{i=1}^m\Delta_g(k_i;x)\,dg.
\]
Moreover, by \eqref{eq:F_k_bound}, using
$\prod_i k_i\le K^m$ and $q\le Q_wK^2$, we have
\[
\left(\frac1{Hq^d}+\frac1{H^{d+1/2}}\right)
\sum_{h\in V_q^\sharp}|F_{\mathbf k}(h)|\ll_{m,d,\varepsilon}
\left(\frac qH+\frac{q^{d+1}}{H^{d+1/2}}\right)
K^{m-1+\varepsilon}\ll_{m,d,w,\varepsilon}
\frac{K^{m+1+\varepsilon}}H
+\frac{K^{m+2d+1+\varepsilon}}{H^{d+1/2}}.
\]
Since $q\le Q_wZ^2\le2H$, applying
\lemref{lem:polynomial-residue-equidistribution}
together with the preceding estimates gives
\begin{equation}\label{eq:direct-residue-comparison}
\frac1{(2H)^{d+1}}
\sum_{P\in\mathscr P_d^\sharp(H)}\Phi(G_H)
\prod_{i=1}^m\Delta_P(k_i;x)d=
\frac1{\zeta(2d+2)}
\int_{V_\square^\sharp}\Phi(g)
\prod_{i=1}^m\Delta_g(k_i;x)\,dg
+O_{m,d,w,\varepsilon}\!\left(
\frac{K^{m+1+\varepsilon}}H
+\frac{K^{m+2d+1+\varepsilon}}{H^{d+1/2}}
\right).
\end{equation}
Summing \eqref{eq:direct-residue-comparison} over $K\le Z$ with the
M\"obius factors, and using the divisor bound once more (after
adjusting $\varepsilon$), gives
\begin{align}
\mathrm{MP}_{m,Z}^{\Phi}(x,H)
&=\mathcal L_{m,Z}^{\Phi}(x)\notag+O_{m,d,w,\varepsilon}\!\left(
\frac{Z^{m+2+\varepsilon}}H
+\frac{Z^{m+2d+2+\varepsilon}}{H^{d+1/2}}
+x^mZ^{1+\varepsilon}
\left(\frac{\log(2x)}H+\frac{x}{H^2}\right)
\right).
\label{eq:direct-truncated-moment-comparison}
\end{align}
Combining this with \eqref{eq:weighted-polynomial-lcm-truncation}
and \eqref{eq:weighted-local-lcm-truncation}, we obtain

\begin{equation}\label{eq:PolynomialLocalReduction}
\begin{aligned}
&\mathbb E_H[\Phi(G_H)E_P(x)^m]
-\frac1{\zeta(2d+2)}
\int_{V_\square^\sharp}\Phi(g)E_g(x)^m\,dg\\
&\quad\ll_{m,d,w,\varepsilon}
x^{m+\varepsilon}\left(
\frac{H^\varepsilon}{Z^{1-\varepsilon}}
+H^{-1/3+\varepsilon}
+x^{d/2}H^{-1/2+\varepsilon}\right)
+\frac{Z^{m+2+\varepsilon}}H
+\frac{Z^{m+2d+2+\varepsilon}}{H^{d+1/2}}
+x^mZ^{1+\varepsilon}
\left(\frac{\log(2x)}H+\frac{x}{H^2}\right).
\end{aligned}
\end{equation}
For $\Phi\equiv1$, we may take $q=K^2$, so
\eqref{eq:PolynomialLocalReduction} holds whenever
$Z^2\le2H$, with implied constants depending only on
$m,d,\varepsilon$. The same bound holds with
$\mathbb E_H$ replaced by
$(2H)^{-d-1}\sum_{P\in\mathscr P_d(H)}$,
since the difference in normalization is
$O_{m,d}(x^m/H)$ and is already absorbed.

To choose the parameters, put
\[
\beta:=\frac{3m+1}{4},\quad
\eta:=\frac1{8(m+2d+4)},\quad Z:=x^\beta H^\eta,
\]
and choose $0<\varepsilon<\min\{\eta/4,\,(16(1+\beta))^{-1},\,1/12\}.$
With these choices, the error in
\eqref{eq:PolynomialLocalReduction} is
$o_{m,d,w}(x^{m/4})$ whenever $H\ge x^{C_dm^2}$,
for a sufficiently large $C_d>0$ depending only
on $d$. Since $w$ is fixed, this also ensures
$Q_wZ^2\le2H$ for sufficiently large $x$.
Restoring $\|\Phi\|_\infty$ completes the proof.
\end{proof}

\section{Non-paired frequency analysis}\label{sec:nonpaired-frequency-analysis}
By \propref{prop:polynomial-local-moment-comparison},
it suffices to study the moments of $E_g(x)$.
In this section, we derive a harmonic representation
of the lcm-truncated local moments and identify the
frequency tuples that give rise to the main contribution.

\subsection{Harmonic expansion of the local moments}
\label{sec:local-harmonic-expansion}
Recall the definitions for squarefree $k$
\[
\mathcal R_g(k^2):=
\{\nu\bmod k^2:g_p(\nu)\equiv0\bmod p^2\text{ for every }p\mid k\},
\quad
R_g(\alpha,k^2):=\sum_{\nu\in\mathcal R_g(k^2)}e(\alpha\nu/k^2).
\]
The Chinese remainder theorem gives
$|\mathcal R_g(k^2)|=\rho_g(k^2)$.
For a squarefree tuple $\mathbf k=(k_1,\ldots,k_m)$, put
$K:=[k_1,\ldots,k_m]$ and $q:=K^2$. Recall the arithmetic factor defined in \eqref{eq:ArithmeticFactor}

\[
\mathfrak A_{\mathbf k}(\boldsymbol\ell)
=\frac1{\zeta(2d+2)}
\int_{V_\square^\sharp}\prod_{i=1}^m R_g(-\ell_i,k_i^2)\,dg=\frac{\delta_d(q)}{q^{d+1}}
\sum_{h\in V_q^\sharp}
\prod_{i=1}^m
\sum_{\substack{\nu_i\bmod k_i^2\\h(\nu_i)\equiv0\bmod k_i^2}}
e\left(-\frac{\ell_i\nu_i}{k_i^2}\right).
\]
The second equality follows from
\eqref{eq:sharp-residue-normalization}, since the integrand depends
only on the reduction modulo $q$.

\begin{proposition}\label{prop:local-model-moment-representation}
For every fixed $m,d\ge1$ and every $z\ge1$,
$\mathcal L_{m,z}(x)=\mathcal M_{m,z}(x)$. Moreover, for every
$\varepsilon>0$,
\[
\frac1{\zeta(2d+2)}\int_{V_\square^\sharp}E_g(x)^m\,dg
=\mathcal M_{m,z}(x)
+O_{m,d,\varepsilon}(x^m z^{-1+\varepsilon}).
\]
\end{proposition}

\begin{proof}
Orthogonality of additive characters yields
\begin{equation}\label{eq:Delta-local-fourier}
\begin{aligned}
\Delta_g(k;x)
&=\sum_{\nu\in\mathcal R_g(k^2)}\sum_{n\le x}
\left(\mathds{1}_{n\equiv\nu\bmod k^2}
-\frac1{k^2}\right)=\frac1{k^2}\sum_{\ell=1}^{k^2-1}
\left(\sum_{n\le x}e(\ell n/k^2)\right)
\left(\sum_{\nu\in\mathcal R_g(k^2)}
e(-\ell\nu/k^2)\right)\\
&=\frac1{k^2}\sum_{\ell=1}^{k^2-1}
E_x\left(\frac{\ell}{k^2}\right)R_g(-\ell,k^2).
\end{aligned}
\end{equation}
Applying it to each factor in
\eqref{eq:weighted-local-lcm-sum} with $\Phi\equiv1$, and then
using \eqref{eq:ArithmeticFactor}, gives
$\mathcal L_{m,z}(x)=\mathcal M_{m,z}(x)$. All sums here are finite.
The second assertion is now
\eqref{eq:weighted-local-lcm-truncation} with $\Phi\equiv1$.
\end{proof}

\subsection{Exactly paired frequencies}
\label{sec:exact-pairing-definition}
For even $m=2k$, call a frequency tuple
$\left(
\ell_{1}/k_{1}^{2},\ldots,
\ell_{m}/k_{m}^{2}
\right)$
appearing in $\mathcal{M}_{m,z}(x)$  \emph{exactly paired} if the
indices $\{1,2,\ldots,2k\}$ admit a perfect matching $\mathcal{P}$ such that
\[
\frac{\ell_{i}}{k_{i}^{2}}
+
\frac{\ell_{j}}{k_{j}^{2}}
\in\mathbb{Z}
\quad
\text{for every }\{i,j\}\in\mathcal{P}.
\]
For odd $m$, no exactly paired tuple exists. We write
\[
\mathcal{M}_{m,z}(x)
=
\mathcal{M}_{m,z}^{\mathrm{pair}}(x)
+
\mathcal{M}_{m,z}^{\mathrm{nonpair}}(x),
\]
where $\mathcal{M}_{m,z}^{\mathrm{pair}}(x)$ is the restriction to
exactly paired frequency tuples and
$\mathcal{M}_{m,z}^{\mathrm{nonpair}}(x)$ is the restriction to all
remaining frequency tuples. The main result of this section is the following proposition.
\begin{proposition}\label{prop:nonpaired-frequency-power-saving}
Fix $m\geq3$ and $\beta>0$. There exists $\delta_{m,\beta}>0$ such that, uniformly for $z\leq x^{\beta}$,\begin{equation}
\left\lvert \mathcal{M}_{m,z}^{\mathrm{nonpair}}\left(x\right)\right\rvert \ll_{m,d,\beta}x^{\frac{m}{4}-\delta_{m,\beta}}.\label{eq:GoodBad-8-6}
\end{equation}
\end{proposition}

The remainder of the section is devoted to proving
\propref{prop:nonpaired-frequency-power-saving}.
As explained in \secref{sec:main-objects-proof-structure},
our approach builds on the method of GMR
\cite[Proposition~2.1]{key-14}.
The new feature relative to GMR’s argument is the arithmetic weight $\mathfrak A_{\mathbf k}(\boldsymbol\ell)$ appearing in
\eqref{eq:MathcalMP_m,z}.

\begin{remark}
Our `exactly paired tuples' definition differs from the paired/non-paired classification used in \cite{key-14}. In particular, the classification in \cite{key-14} applies to frequency tuples \textit{after} certain steps, which include transforming the denominators to be squares; we do not apply these transformations for this definition.
\end{remark}

\subsection{Local estimates for root exponential sums}\label{sec:root-exponential-sum-estimates}

For a prime $p$, a polynomial $g\in V_{p^{2}}$, and a frequency 
$\xi\in\mathbb{Q}/\mathbb{Z}$, define
\[
T_{g}\left(\xi;p^{2}\right)
:=
\sum_{\substack{u\bmod p^{2}\\
g\left(u\right)\equiv0\bmod p^{2}
}}
e\left(\xi u\right).
\]
We call these sums root (exponential) sums. The inner sums in
\eqref{eq:ArithmeticFactor} are products of such sums, by the Chinese remainder theorem. The purpose of this
subsection is to obtain bounds on those sums, which are needed to control $\mathfrak A_{\mathbf{k}}(\boldsymbol{\ell})$. When $\xi=\lambda/p^2$, the behavior of the root sum depends on whether $p\nmid \lambda$ (a case we call `primitive') or not. At primitive frequencies we have
pointwise cancellation, while
for frequencies divisible by $p$ we must use averaging over $g$ to get cancellation. The
following lemma gives both bounds.
\begin{lemma}\label{lem:root-exponential-sum-moment-bounds}
Let $p$ be prime and let $n,d\geq1$ be fixed. Uniformly for
$\lambda\bmod p^{2}$,
\[
\mathbb{E}_{g\in V_{p^{2}}^{\sharp}}
\left|
T_{g}\left(\frac{\lambda}{p^{2}};p^{2}\right)
\right|^{n}
\ll_{d,n}
\begin{cases}
1, & p\nmid\lambda,\\
1+p^{n-2}, & p\mid\lambda.
\end{cases}
\]
In particular, if $(a,p)=1$, then
\[
\mathbb{E}_{g\in V_{p^{2}}}
\left|
T_{g}\left(\frac{a}{p^{2}};p^{2}\right)
\right|^{n}
\ll_{d,n}1,
\quad
\mathbb{E}_{g\in V_{p^{2}}}
\left|
T_{g}\left(\frac{a}{p};p^{2}\right)
\right|^{n}
\ll_{d,n}1+p^{n-2}.
\]
\end{lemma}

\begin{proof}
Suppose first that $p\nmid\lambda$. We begin with the proof of the stronger pointwise
bound
\[
\left|
T_{g}\left(\frac{\lambda}{p^{2}};p^{2}\right)
\right|
\leq d
\quad
\left(g\in V_{p^{2}}\right).
\]
Every residue class modulo $p^{2}$ can be written uniquely in the form
$u=\beta+kp$, with $\beta,k\in\mathbb{Z}/p\mathbb{Z}$. Hence
\begin{align*}
T_{g}\left(\frac{\lambda}{p^{2}};p^{2}\right)
&=
\sum_{\substack{\beta\in\mathbb{Z}/p\mathbb{Z}\\
g\left(\beta\right)\equiv0\bmod p
}}
e\left(\frac{\lambda\beta}{p^{2}}\right)
\sum_{\substack{k\in\mathbb{Z}/p\mathbb{Z}\\
g\left(\beta+kp\right)\equiv0\bmod p^{2}
}}
e\left(\frac{\lambda k}{p}\right)
\\
&=
\sum_{\substack{\beta\in\mathbb{Z}/p\mathbb{Z}\\
g\left(\beta\right)\equiv0\bmod p
}}
e\left(\frac{\lambda\beta}{p^{2}}\right)
\sum_{\substack{k\in\mathbb{Z}/p\mathbb{Z}\\
g\left(\beta\right)+kp\,g'\left(\beta\right)\equiv0\bmod p^{2}
}}
e\left(\frac{\lambda k}{p}\right),
\end{align*}
where the second equality follows from
\[
g\left(\beta+kp\right)
\equiv
g\left(\beta\right)+kp\,g'\left(\beta\right)
\bmod{p^{2}}.
\] Fix a root $\beta$ of $g$ modulo $p$. If
$g'\left(\beta\right)\equiv0\bmod p$, then either $\beta$ has no lift to
a root modulo $p^{2}$, or every residue $\beta+kp$, $k\bmod p$, is such a
lift. In the latter case their total contribution is
\[
e\left(\frac{\lambda\beta}{p^{2}}\right)
\sum_{k\in\mathbb{Z}/p\mathbb{Z}}
e\left(\frac{\lambda k}{p}\right)
=0,
\]
since $p\nmid\lambda$. Thus singular roots modulo $p$ contribute
nothing. If instead $g'\left(\beta\right)\not\equiv0\bmod p$, Hensel's lemma
shows that $\beta$ has a unique lift to a root modulo $p^{2}$. Hence each
nonsingular root contributes a single term of absolute value $1$. If
$g\bmod p$ is the zero polynomial, there are no nonsingular roots.
Otherwise, $g\bmod p$ is a nonzero polynomial of degree at most $d$, so
it has at most $d$ roots modulo $p$. This proves the pointwise bound, and
therefore
\[
\mathbb{E}_{g\in V_{p^{2}}^{\sharp}}
\left|
T_{g}\left(\frac{\lambda}{p^{2}};p^{2}\right)
\right|^{n}
\ll_{d,n}1.
\]
Suppose now that $p\mid\lambda$. For $g\in V_{p^{2}}^{\sharp}$, define
\[
Z_{g}
:=
\#\left\{
y\bmod p:
g\left(y\right)\equiv0\bmod p^{2},
\quad
g'\left(y\right)\equiv0\bmod p
\right\}.
\]
A nonsingular root modulo $p$ has at most one lift modulo $p^{2}$,
whereas each residue counted by $Z_{g}$ gives a complete fibre
$y+pt$, $t\bmod p$. Therefore, uniformly for $p\mid\lambda$,
\[
\left|
T_{g}\left(\frac{\lambda}{p^{2}};p^{2}\right)
\right|
\leq d+pZ_{g}.
\]
Moreover, $Z_{g}\leq d$. Indeed, if $g\bmod p$ is nonzero, every residue
counted by $Z_g$ is a root of the nonzero polynomial $g\bmod p$, of
which there are at most $d$. If all coefficients of $g$ are divisible by $p$, write
$g=ph$ modulo $p^{2}$. Since $g\in V_{p^{2}}^{\sharp}$,
the polynomial $h\bmod p$ is nonzero and $\deg h\leq d$.
Moreover, $g(y)\equiv0\bmod p^{2}$ is equivalent to
$h(y)\equiv0\bmod p$. Hence there are at most $d$ such
residues $y\bmod p$. To estimate the mean of $Z_g$, split $V_{p^{2}}^{\sharp}$ according to
whether $g\bmod p$ is nonzero or all coefficients of $g$ are divisible
by $p$. In the first class, for each fixed $y\bmod p$, the conditions
\[
g\left(y\right)\equiv0\bmod p^{2},
\quad
g'\left(y\right)\equiv0\bmod p
\]
impose three independent linear conditions modulo $p$, and hence occur
with probability $O_{d}\left(p^{-3}\right)$. Summing over $y\bmod p$
gives a contribution of $O_{d}\left(p^{-2}\right)$. The class in which all
coefficients are divisible by $p$ has relative density
\[
\frac{p^{d+1}-1}{p^{2d+2}-1}
=
O\left(p^{-d-1}\right)
=
O\left(p^{-2}\right),
\]
since $d\geq1$. Consequently,
\[
\mathbb{E}_{g\in V_{p^{2}}^{\sharp}}Z_{g}
\ll_{d}p^{-2},
\quad
\mathbb{E}_{g\in V_{p^{2}}^{\sharp}}Z_{g}^{n}
\leq
d^{n-1}\mathbb{E}_{g\in V_{p^{2}}^{\sharp}}Z_g
\ll_{d,n}p^{-2}.
\]
It follows that
\[
\mathbb{E}_{g\in V_{p^{2}}^{\sharp}}
\left|
T_{g}\left(\frac{\lambda}{p^{2}};p^{2}\right)
\right|^{n}
\ll_{d,n}
1+p^{n}\mathbb{E}_{g\in V_{p^{2}}^{\sharp}}Z_{g}^{n}
\ll_{d,n}
1+p^{n-2}.
\] Finally, for the two nonzero frequencies: $a/p^{2}$ and $a/p$ with $(a,p)=1$, the zero polynomial
contributes zero by additive orthogonality. Passing from
$V_{p^{2}}^{\sharp}$ to $V_{p^{2}}$ therefore changes the average only
by the bounded factor $1-p^{-2d-2}$. This proves the two final estimates.
\end{proof}

The following lemma bounds the moments of $\rho_g(p^2)$ over $V_{p^2}^{\sharp}$.

\begin{lemma}\label{lem:root-count-moment-bounds}
Let $p$ be a prime and $\ell\in\mathbb N$. Then
\[
\mathbb E_{g\in V_{p^2}^{\sharp}}
\rho_g(p^2)^\ell
\ll_{\ell,d}1+p^{\ell-2}.
\]
\end{lemma}

\begin{proof}
Since $T_g(0;p^2)=\rho_g(p^2)$, the result follows
from \lemref{lem:root-exponential-sum-moment-bounds}
with $\lambda=0$ and $n=\ell$.
\end{proof}

\subsection{The arithmetic factor}\label{sec:arithmetic-factor}
We derive a product formula, over primes $p$ dividing $K$, for the arithmetic factor $\mathfrak{A}_{\mathbf{k}}(\boldsymbol{\ell})$. For a prime $p\mid K$, the local factor in the formula will depend on
which of the corresponding numerators $\ell_i$, with $p\mid k_i$, are
also divisible by $p$. The local factor will be related to the root sums studied earlier. To state the formula we require some notation. For $p\mid K$ put
$ I_p = \{ i \in [m]: p \mid k_i\}.$
For $i\in [m]$ put
$N_i:=k_i^{2}$, so that
$q=[N_1,\ldots,N_m]$. For each prime $p\mid K$ and $i\in I_p$, write
$\lambda_{i,p}\bmod p^{2}$ for the $p$-component of
$\ell_i/N_i$, namely
\begin{equation}
\lambda_{i,p}
\equiv
\ell_i\left(\frac{N_i}{p^{2}}\right)^{-1}
\bmod p^{2},
\label{eq:Pcomponent}
\end{equation}
so that
\[ \frac{\ell_i}{N_i} = \sum_{p\mid N_i}  \frac{\lambda_{i,p}}{p^2} \bmod 1.\]
Put $\boldsymbol{\lambda}_p
:=
\left(\lambda_{i,p}\right)_{i\in I_p}$.
Define the corresponding local arithmetic factor by
\[
A_p\left(\boldsymbol{\lambda}_p\right)
:=
\frac{1}{\left|V_{p^{2}}^{\sharp}\right|}
\sum_{g\in V_{p^{2}}^{\sharp}}
\prod_{i\in I_p}
T_g\left(-\frac{\lambda_{i,p}}{p^{2}};p^{2}\right)
\] and put 
\begin{equation}
\mathcal{H}\left(\mathbf{k}\right)
:=
\left\{
\boldsymbol{\ell}\in\mathbb{Z}^m:
0< \ell_i<N_i,\ 
\sum_{i\in[m]}\frac{\ell_i}{N_i}\in\mathbb{Z}
\right\}.
\label{eq:zeroSum}
\end{equation}
The following lemma first shows $\mathfrak{A}_{\mathbf{k}}\left(\boldsymbol{\ell}\right)=0$ when $\boldsymbol{\ell} \not\in \mathcal{H}\left(\mathbf{k}\right)$, and gives the product formula when $\boldsymbol{\ell} \in \mathcal{H}\left(\mathbf{k}\right)$.
\begin{lemma}\label{lem:arithmetic-factor-factorization}
Let $\boldsymbol{\ell}=(\ell_1,\ldots,\ell_m)\in\mathbb{Z}^m$
satisfy $0<\ell_i<N_i$ for every $i\in[m]$.
Then the arithmetic factor
$\mathfrak{A}_{\mathbf{k}}(\boldsymbol{\ell})$
vanishes unless
$\boldsymbol{\ell}\in\mathcal{H}(\mathbf{k})$.
For $\boldsymbol{\ell}\in\mathcal{H}(\mathbf{k})$,
\[
\mathfrak{A}_{\mathbf{k}}(\boldsymbol{\ell})
=
\frac{1}{\zeta(2d+2)}
\prod_{p\mid K}A_p(\boldsymbol{\lambda}_p).
\]
\end{lemma}

\begin{proof}
We first prove the vanishing statement. Fix
$h\bmod q$ and make the change of variables
\[
g(X)\mapsto g(X-h),\quad \nu_i\mapsto \nu_i+h.
\]
This is a bijection of the summation variables defining $\mathfrak{A}_{\mathbf{k}}(\boldsymbol{\ell})$, while the exponential
factor is multiplied by
$e\left(-h\sum_{i=1}^{m}\ell_i/N_i\right).$
Hence, for every $h\bmod q$,
\[
\mathfrak{A}_{\mathbf{k}}(\boldsymbol{\ell})
=
e\left(-h\sum_{i=1}^{m}\frac{\ell_i}{N_i}\right)
\mathfrak{A}_{\mathbf{k}}(\boldsymbol{\ell}).
\]
Therefore $\mathfrak{A}_{\mathbf{k}}\left(\boldsymbol{\ell}\right)=0$ when $
\boldsymbol{\ell} \not\in \mathcal{H}\left(\mathbf{k}\right)$, as needed. We turn to the second assertion. The multiplier in \eqref{eq:Pcomponent} is a unit modulo
$p^{2}$, so $p\mid\lambda_{i,p}$
if and only if
$p\mid\ell_i$. The requirement  $\boldsymbol{\ell} \in \mathcal{H}\left(\mathbf{k}\right)$ is equivalent to
\[\sum_{i\in I_p}\lambda_{i,p}\equiv0\bmod p^{2}
\quad\forall p\mid K.\]
Using the local root sum
$T_g\left(\xi;p^{2}\right)$, and since
\[
\frac{\delta_d(q)}{q^{d+1}}
=
\frac{1}{\zeta\left(2d+2\right)}
\frac{1}{\left|V_q^{\sharp}\right|},
\]
the Chinese remainder theorem gives the prime-by-prime decomposition of $\mathfrak{A}_{\mathbf{k}}\left(\boldsymbol{\ell}\right)$.
\end{proof}

With this factorization established, the next lemma bounds each local factor and identifies the only possible source of growth. Put $u_{p}\left(\boldsymbol{\ell}\right)
:=
\#\left\{
i\in I_{p}:p\mid\ell_{i}
\right\}$.
\begin{lemma}\label{lem:local-arithmetic-factor-bound}
For every $p\mid K$ and every
$\boldsymbol{\ell}\in\mathcal{H}\left(\mathbf{k}\right)$,
\begin{equation}\label{eq:Apbnd}
\left\lvert
A_{p}\left(\boldsymbol{\lambda}_{p}\right)
\right\rvert
\ll_{m,d}
p^{\max\left\{0,u_{p}\left(\boldsymbol{\ell}\right)-2\right\}}.
\end{equation}
Consequently, for every $\varepsilon>0$,
\[
\left\lvert
\mathfrak{A}_{\mathbf{k}}\left(\boldsymbol{\ell}\right)
\right\rvert
\ll_{m,d,\varepsilon}
q^{\varepsilon}
\prod_{p\mid K}
p^{\max\left\{0,u_{p}\left(\boldsymbol{\ell}\right)-2\right\}}.
\]
In particular, a positive
power of $p \mid K$ can occur in the bound only when
$u_{p}\left(\boldsymbol{\ell}\right)\geq3$.
\end{lemma}

\begin{proof}
By definition \eqref{eq:Pcomponent}, $u_{p}\left(\boldsymbol{\ell}\right)$ is the number of $i\in [m]$ for which 
$p\mid\lambda_{i,p}$. \lemref{lem:root-exponential-sum-moment-bounds}
gives
\[
\mathbb{E}_{g\in V_{p^{2}}^{\sharp}}
\left\lvert
T_{g}\left(-\frac{\lambda_{i,p}}{p^{2}};p^{2}\right)
\right\rvert^{u_{p}}
\ll_{m,d}
1+p^{u_{p}-2}.
\]
If $u_{p}=0$, all factors are bounded pointwise. If $u_{p}\geq1$,
Hölder's inequality applied to the $u_{p}$ divisible-frequency factors
gives
\[
\left\lvert
A_{p}\left(\boldsymbol{\lambda}_{p}\right)
\right\rvert
\ll_{m,d}
\prod_{\substack{i\in I_{p}\\
p\mid\lambda_{i,p}}}
\left(
\mathbb{E}_{g\in V_{p^{2}}^{\sharp}}
\left\lvert
T_{g}\left(-\frac{\lambda_{i,p}}{p^{2}};p^{2}\right)
\right\rvert^{u_{p}}
\right)^{1/u_{p}}\ll_{m,d}
1+p^{u_{p}-2}.\]
This gives \eqref{eq:Apbnd}. Using \lemref{lem:arithmetic-factor-factorization} and plugging \eqref{eq:Apbnd} gives
\[
\left\lvert
\mathfrak{A}_{\mathbf{k}}\left(\boldsymbol{\ell}\right)
\right\rvert
\leq
C_{m,d}^{\omega(K)}
\prod_{p\mid K}
p^{\max\left\{0,u_{p}\left(\boldsymbol{\ell}\right)-2\right\}}.
\]
for some constant $C_{m,d}>0$. For every $\varepsilon>0$, the divisor bound $C_{m,d}^{\omega(K)}
\ll_{m,d,\varepsilon}
K^{\varepsilon}
$
together with $q=K^{2}$ gives the second bound needed.
\end{proof}
\subsection{Good-bad decomposition of frequencies}\label{sec:good-bad-frequency-decomposition}
We partition the frequencies $(\ell_i/k^2_i)_{1\le i \le m}$ into two disjoint sets of \emph{good} and \emph{bad} frequencies we now define. Fix a small parameter $\eta>0$ to be chosen later. For each $\mathbf{k}\in\mathbb{N}^m$, define the good set

\begin{equation}
\mathcal{G}_{\eta}\left(\mathbf{k}\right):=\left\{ \boldsymbol{\ell}\in\mathcal{H}\left(\mathbf{k}\right):\left\lvert \mathfrak{A}_{\mathbf{k}}\left(\boldsymbol{\ell}\right)\right\rvert \leq q^{\eta}\right\} ,\label{eq:goodSet}
\end{equation}and the bad set

\begin{equation}
\mathcal{B}_{\eta}\left(\mathbf{k}\right):=\left\{ \boldsymbol{\ell}\in\mathcal{H}\left(\mathbf{k}\right):\left\lvert \mathfrak{A}_{\mathbf{k}}\left(\boldsymbol{\ell}\right)\right\rvert >q^{\eta}\right\} \label{eq:badSet}
\end{equation}and decompose

\begin{equation}
\mathcal{M}_{m,z}^{\mathrm{nonpair}}\left(x\right):=\mathcal{M}_{m,z}^{\mathrm{good}}\left(x\right)+\mathcal{M}_{m,z}^{\mathrm{bad}}\left(x\right)\label{eq:GoodBad}
\end{equation} with the additional restriction
$\boldsymbol{\ell}\in\mathcal{G}_{\eta}\left(\mathbf{k}\right)$ or
$\boldsymbol{\ell}\in\mathcal{B}_{\eta}\left(\mathbf{k}\right)$,
respectively.

We now analyze the arithmetic factor
\(\mathfrak{A}_{\mathbf{k}}(\boldsymbol{\ell})\) appearing in \eqref{eq:MathcalMP_m,z}. In the GMR setting (\cite[Proposition 2.1]{key-14}) no such additional factor occurs, so their estimates can be applied directly after reducing the frequencies. In our setting, the same argument still gives the required power saving whenever $\left|\mathfrak A_{\mathbf{k}}(\boldsymbol{\ell})\right|\leq q^{\eta}$
for a sufficiently small \(\eta>0\). Since $q\leq z^2$ and $z\leq x^\beta$, this introduces
at most a factor $x^{2\beta\eta}$. As we shall see,
this loss can be absorbed into the power savings
obtained in GMR's estimates
(\cite[Section~5]{key-14}). We therefore split the frequencies according to whether $\left\lvert \mathfrak{A}_{\mathbf{k}}\left(\boldsymbol{\ell}\right)\right\rvert \leq q^{\eta}$ holds or not.  
The bad frequencies require a separate estimate. 

We first treat the good contribution and then estimate the bad contribution separately in \secref{sec:bad-frequency-weighted-estimates}. Throughout, we use
\[
\left|E_x(\alpha)\right|\leq F_x(\alpha),
\quad
F_x(\alpha):=\min\left(x,\left\Vert\alpha\right\Vert^{-1}\right),
\]
where \(\left\Vert\alpha\right\Vert\) denotes the distance from \(\alpha\) to the nearest integer. This notation goes back to Montgomery and Vaughan \cite{key-16}. 

\subsection{Good frequencies: reduction to the GMR argument}
\label{sec:good-frequency-GMR-reduction}

The argument in this subsection reduces the analysis of $\mathcal{M}_{m,z}^{\mathrm{good}}\left(x\right)$ to existing bounds derived in \cite[\S\S4--5]{key-14}. We begin by rewriting the good-frequency contribution in the notation used by GMR. For each frequency $\ell_i/N_i$, reduce
\[
\frac{\ell_i}{N_i}
=
\frac{a_i}{r_i},
\quad
(a_i,r_i)=1.
\]
Write $r_i=c_i v_i^2,$ $k_i=c_i v_i b_i$
where $c_i,v_i,b_i$ are squarefree and pairwise coprime, and set
\[
n_i:=c_i v_i,
\quad
\widetilde r_i:=n_i^2=c_i r_i,
\quad
u_i:=a_i c_i.
\]
Then
\[
N_i=k_i^2=\widetilde r_i b_i^2,
\quad
\frac{\ell_i}{N_i}
=
\frac{u_i}{\widetilde r_i},
\quad
\ell_i=u_i b_i^2.
\]
We refer to $u_i/\widetilde r_i$ as the \emph{square-block} representation of $\ell_i/N_i$; this representation is the one used in \cite{key-14}, although they do not use this name. Note that $u_i/\widetilde{r_i}$ is not necessarily reduced. We refer to the denominator $\widetilde{r_i}$ as the \emph{block-square} denominator.
For squarefree $n$ and $c\mid n$, let
\[
\mathcal R_c(n^2)
:=
\left\{
\frac{u}{n^2}:
0<u<n^2,\;
(u,n^2)=c
\right\};
\]
we call this set the class of fractions of \emph{type} $c$.
With these notations, the frequency $u_i/\widetilde r_i$ belongs to
$\mathcal R_{c_i}(\widetilde r_i)$. Moreover,
\[
\mathcal R_B(n^2)
:=
\bigsqcup_{c\mid n}\mathcal R_c(n^2)
=
\left\{
\frac{u}{n^2}:
0<u<n^2,\;
(u,n^2)\ \text{is squarefree}
\right\},
\]
that is, the fractions of type $c$, as $c$ varies over divisors of $n$, form a partition of the frequency class $\mathcal R_B(n^2)$, which appears in 
\cite[p.~9]{key-14}, for $B:=\{p^2:p\ \mathrm{prime}\}$.
The preceding change of variables also gives
\[
\frac{\mu(k_i)}{N_i}
=
\frac{\mu(n_i)\mu(b_i)}{\widetilde r_i b_i^2},
\]
since $k_i=n_i b_i$, with $n_i$ and $b_i$ squarefree and coprime.
Consequently, writing $\sigma_i:=u_i/\widetilde r_i$
and $\widetilde r_i:=n_i^2$, the good contribution
takes the form
\begin{equation}
\mathcal M_{m,z}^{\mathrm{good}}(x)
=
\sum_{\substack{
n_i\ge 2,\ c_i\mid n_i\\
b_i\ge 1,\ (b_i,n_i)=1,\ (1\le i\le m)\\
K=[n_1b_1,\ldots,n_mb_m]\le z
}}
\frac{\prod_{i=1}^{m}\mu(n_i)}{\widetilde r_1\cdots\widetilde r_m}
\prod_{i=1}^m\frac{\mu(b_i)}{b_i^2}
\sum_{\substack{
\sigma_i\in\mathcal R_{c_i}(\widetilde r_i),\ 1\leq i\leq m\\
\sum_{i=1}^m\sigma_i\in\mathbb Z\\
(\sigma_1,\ldots,\sigma_m)\ \mathrm{not\ exactly\ paired}\\
|\mathfrak A_{\mathbf k}(\boldsymbol\ell)|\leq q^\eta
}}
\mathfrak A_{\mathbf k}(\boldsymbol\ell)
\prod_{i=1}^m E_x(\sigma_i),
\label{eq:M_m,zGoodGMR}
\end{equation}
where, in the inner sum, $k_i=n_i b_i$ and
$\ell_i=u_i b_i^2$. For a tuple of square-block denominators $\widetilde{\mathbf r}
=
(\widetilde r_1,\ldots,\widetilde r_m),$
write
\[
S_x(\widetilde{\mathbf r})
:=
\sum_{\substack{
\sigma_i\in
\mathcal R_B(\widetilde r_i),\ 1\leq i\leq m\\
\sum_{i=1}^m\sigma_i\in\mathbb Z
}}
\prod_{i=1}^m F_x(\sigma_i).
\]
We denote by
\(S_x^{\mathrm{exact}}(\widetilde{\mathbf r})\)
the contribution from exactly paired tuples
\((\sigma_1,\ldots,\sigma_m)\), and by
\(S_x^{\mathrm{nonexact}}(\widetilde{\mathbf r})\)
the contribution from all remaining tuples. Thus
\[
S_x(\widetilde{\mathbf r})
=
S_x^{\mathrm{exact}}(\widetilde{\mathbf r})
+
S_x^{\mathrm{nonexact}}(\widetilde{\mathbf r}).
\]
\begin{lemma}\label{lem:good-frequency-square-block-reduction}
For $m\ge3$ and $z\ge1$, the good-frequency contribution satisfies
\begin{equation}
\left|
\mathcal M_{m,z}^{\mathrm{good}}(x)
\right|
\ll_m
z^{2\eta}
\sum_{\substack{
\widetilde{\mathbf r}=(n_1^2,\ldots,n_m^2)\\
n_i\ge2\ \mathrm{squarefree}\ (1\le i\le m)
}}
\frac{
S_x^{\mathrm{nonexact}}(\widetilde{\mathbf r})
}{
\widetilde r_1\cdots\widetilde r_m
}.
\label{eq:good-frequency-GMR-reduction}
\end{equation}
\end{lemma}

\begin{proof}
On the good set,
\[
\left|
\mathfrak A_{\mathbf k}(\boldsymbol\ell)
\right|
\leq q^\eta
\leq z^{2\eta},
\]
since $q=K^2$ and $K\leq z$. Moreover, uniformly for fixed $n_i$,
\begin{equation}
\sum_{\substack{b_i\geq1\\(b_i,n_i)=1}}
\frac{\mu^2(b_i)}{b_i^2}
\ll 1.
\label{eq:squarefree-completion-summability}
\end{equation}
Hence, after taking absolute values in \eqref{eq:M_m,zGoodGMR} and summing over the \(b_i\), their total contribution is bounded. Also, since the classes
$\mathcal R_c(n^2)$ form a disjoint partition of
$\mathcal R_B(n^2)$, summing over the types $c_i$
replaces the type-restricted frequency sums by the complete
square-block classes from $\mathcal{R}_B$, without any multiplicity. The frequencies
themselves have not changed, so both the zero-sum condition and the
exclusion of exact pairings are preserved.
Using the majorant $F_x$ introduced in \secref{sec:good-bad-frequency-decomposition}, we obtain the bound in the lemma, initially over the reduced-denominator tuples arising from the preceding parametrization. Since the summands are nonnegative, the sum may be enlarged to the complete square-block range of GMR: indeed, $\widetilde r_i=n_i^2=\prod_{p\mid n_i}p^2$ is a product of distinct blocks from $B$.
\end{proof}
The following lemma supplies the required power saving for this complete sum.

\begin{lemma}
\label{lem:good-frequency-power-saving}
Fix $m\geq3$ and $\beta>0$. For $\eta>0$ sufficiently small, there
exists $\delta_{m,\beta}>0$ such that, uniformly for $z\leq x^\beta$,
\[
\left|
\mathcal M_{m,z}^{\mathrm{good}}(x)
\right|
\ll_{m,\beta}
x^{\frac m4-\delta_{m,\beta}}.
\]
\end{lemma}

\begin{proof}
By \lemref{lem:good-frequency-square-block-reduction}, it remains to bound
\[
\sum_{\substack{
\widetilde{\mathbf r}=(n_1^2,\ldots,n_m^2)\\
n_i\ge2\ \mathrm{squarefree}\ (1\le i\le m)
}}
\frac{
S_x^{\mathrm{nonexact}}(\widetilde{\mathbf r})
}{
\widetilde r_1\cdots\widetilde r_m
}.
\]
We split the square-block denominator tuples
$\widetilde{\mathbf r}$ according to the classification
used in \cite[\S\S4--5]{key-14} for the square-block family $B$.
Suppose first that $m=2k$ and that $\widetilde{\mathbf r}$
is paired and non-repeated in the sense of GMR. After relabeling, its entries occur in $k$ pairs with distinct common denominators \(R_1,\ldots,R_k\). Write the two frequencies in the \(j\)-th pair as $u_j/R_j$,
$u'_j/R_j,$
and let \(h_j\in[1,R_j]\) be
\[
h_j\equiv u_j+u'_j\bmod{R_j}.
\]
Then
\[
h_j=R_j
\quad\Longleftrightarrow\quad
\frac{u_j}{R_j}+\frac{u'_j}{R_j}\in\mathbb Z,
\]
so the case \(h_j=R_j\) for every \(j\) is precisely the exactly paired frequency contribution. Since this contribution is excluded from
\(S_x^{\mathrm{nonexact}}(\widetilde{\mathbf r})\), at least one \(j\) satisfies
\(0<h_j<R_j\).
Moreover, the zero-sum condition in \eqref{eq:M_m,zGoodGMR} gives
\[
\frac{h_1}{R_1}+\cdots+\frac{h_k}{R_k}\in\mathbb Z.
\]
Hence there cannot be exactly one index with \(0<h_j<R_j\). Thus at least two such blocks occur. To apply the argument of
\cite[Proposition 4.3]{key-14} to our positive weights, put
\[
J_x^+(b,R):=
\sum_{\substack{\sigma,\sigma'\in\mathcal R_B(R)\\
\sigma+\sigma'=b/R\;(\bmod 1)}}
F_x(\sigma)F_x(\sigma').
\]
By \cite[Lemma 2.2]{key-14},
\[
\sum_{a=1}^{R-1}F_x(a/R)\ll R\log(2R),
\quad
\sum_{a=1}^{R-1}F_x(a/R)^2\ll Rx.
\]
Since $J_x^+(\cdot,R)$ is the convolution of $F_x$
restricted to $\mathcal R_B(R)$ with itself,
\cite[Lemma 2.2]{key-14} and Young's inequality give
\[
\sum_{0<b<R}J_x^+(b,R)\ll R^2(\log(2R))^2,
\quad
\sum_{b\bmod R}J_x^+(b,R)^2
\ll R^3x(\log(2R))^2.
\]
Let $Z\ge1$ be a parameter to be chosen below, and split
the sum according to whether $[R_1,\ldots,R_j]\le Z$ or
$[R_1,\ldots,R_j]>Z$. Applying \cite[Lemma~4.1]{key-14},
using the first estimate for small least common multiples
and the second for large ones, as in the proof of
\cite[Proposition~4.3]{key-14}, gives, for $j\ge2$,
\[
\begin{aligned}
W_j^+(x)
&:=
\sum_{\substack{R_1,\ldots,R_j>1\\R_i\in[B]}}
\frac{1}{R_1^2\cdots R_j^2}
\sum_{\substack{0<b_i<R_i\\
\sum_i b_i/R_i\in\mathbb Z}}
\prod_{i=1}^j J_x^+(b_i,R_i)\\
&\ll_j x^{j/2}Z^{-1/2+o(1)}
       +Z^{1/2+o(1)}.
\end{aligned}
\]
Choosing $Z=x^{j/2}$ to balance the two terms, we obtain $W_j^+(x)\ll_j x^{j/4+o(1)}$
where $Z=x^{j/2}$. Here we also use
\cite[Lemma 2.4]{key-14} and the fact that
the semigroup generated by $B$ has index $1/2$.

For an exactly paired block, \cite[Lemma 2.2]{key-14}
gives
\[
D_x:=
\sum_{\substack{R\in[B]\\R>1}}
\frac{J_x^+(0,R)}{R^2}
\ll
\sum_{\substack{R\in[B]\\R\le x}}1
+x\sum_{\substack{R\in[B]\\R>x}}\frac1R
\ll x^{1/2}.
\]
Therefore, the contribution of paired and non-repeated
denominator tuples with exactly $j\ge2$ nonzero
pair sums is at most
\[
D_x^{k-j}W_j^+(x)
\ll_m x^{m/4-j/4+o(1)}
\ll_m x^{m/4-1/2+o(1)}.
\]
If instead \(\widetilde{\mathbf r}\) is repeated or non-paired in the sense of GMR, the estimates of \cite[\S5]{key-14} give a power saving for the entire corresponding frequency sum, and hence also for its restriction to nonexact tuples. When $m$ is odd, the paired and non-repeated case cannot occur.
These cases exhaust all square-block denominator tuples
$\widetilde{\mathbf r}$. Consequently, for some \(\Delta_m>0\),
\[
\sum_{\widetilde{\mathbf r}}
\frac{
S_x^{\mathrm{nonexact}}(\widetilde{\mathbf r})
}{
\widetilde r_1\cdots\widetilde r_m
}
\ll_m
x^{\frac m4-\Delta_m}.
\label{eq:GMR-nonexact-square-block-bound}
\]
Substituting this into \eqref{eq:good-frequency-GMR-reduction} gives
\[
\left|
\mathcal M_{m,z}^{\mathrm{good}}(x)
\right|
\ll_m
z^{2\eta}x^{\frac m4-\Delta_m}.
\]
Since \(z\le x^\beta\),
\[
\left|
\mathcal M_{m,z}^{\mathrm{good}}(x)
\right|
\ll_m
x^{\frac m4-\Delta_m+2\beta\eta}.
\]
Choosing \(\eta>0\) sufficiently small so that $\delta_{m,\beta}:=\Delta_m-2\beta\eta>0$
proves the result.
\end{proof}

\subsection{Bad frequencies}\label{sec:bad-frequency-weighted-estimates}
We now turn to the bad set, where the arithmetic factor can be too large for reusing the GMR estimate.
Instead, we prove two complementary lemmas dealing with large and small $q=[k_1,\ldots,k_m]^2$, respectively.
For $m>2$, fix a small parameter
\begin{equation}
0<\vartheta<\frac{1}{m-2}\label{eq:thetaMoment}
\end{equation}
to be chosen later. Throughout this subsection, $m,d,\vartheta$ and $\eta$
are fixed independently of $x,z$ and $K$. For a prime $p$, let $\nu_p(n):=\max\{k\ge0:p^k \mid n\}$. 

\begin{lemma}
\label{lem:bad-frequency-weighted-sum-bound}
With $\vartheta$ as in \eqref{eq:thetaMoment}, for every
$\varepsilon>0$, uniformly in
$\mathbf{k}\in\mathbb N^m$, we have
\[
\sum_{\boldsymbol{\ell}\in\mathcal{B}_{\eta}(\mathbf{k})}
\prod_{i=1}^{m}
F_x\left(\frac{\ell_i}{N_i}\right)
\left|
\mathfrak{A}_{\mathbf{k}}(\boldsymbol{\ell})
\right|
\ll_{m,d,\vartheta}
q^{-\eta\vartheta+\varepsilon}
\frac{N_1\cdots N_m}{q}
x^{\frac{m}{2}}.
\]
\end{lemma}

\begin{proof}
Let $r_1,\ldots,r_m\geq2$ and put
\[
R:=[r_1,\ldots,r_m].
\]
For every prime $p\mid R$, the maximum of
$v_p(r_1),\ldots,v_p(r_m)$ is necessarily attained at least twice;
otherwise, after multiplying
$\sum_i a_i/r_i\in\mathbb Z$ by $R$ and reducing modulo $p$, one obtains
a contradiction to $(a_i,r_i)=1$ for the unique index attaining the
maximum. Hence, we may use the Fundamental Lemma of
Montgomery and Vaughan \cite[Chapter~8, \S4]{key-15} that gives
\begin{equation}
\sum_{\substack{
a_i\bmod r_i,\ (a_i,r_i)=1\\
\sum_{i=1}^m a_i/r_i\in\mathbb Z
}}
\prod_{i=1}^{m}
F_x\left(\frac{a_i}{r_i}\right)
\ll_m
\frac{r_1\cdots r_m}{R}
x^{\frac m2}.
\label{eq:reduced-frequency-zero-sum}
\end{equation}
We now apply this estimate to the frequencies occurring in
$\mathcal H(\mathbf k)$. For a fixed $\boldsymbol\ell\in\mathcal H(\mathbf k)$,
reduce each frequency by setting
\[
h_i:=(\ell_i,N_i),
\quad
r_i:=\frac{N_i}{h_i},
\quad
a_i:=\frac{\ell_i}{h_i}.
\]
Since $0<\ell_i<N_i$, we have $r_i\ge2$ and
$1\le a_i<r_i$, with $(a_i,r_i)=1$.
Moreover, $\ell_i/N_i=a_i/r_i$, and the zero-sum
condition becomes $\sum_{i=1}^m a_i/r_i\in\mathbb Z$ so the bound \eqref{eq:reduced-frequency-zero-sum} applies.

It remains to compare the right-hand side in
\eqref{eq:reduced-frequency-zero-sum} with the quantity
$N_1\cdots N_m/q$, while at the same time incorporating the bound for
the arithmetic factor $\left|
\mathfrak{A}_{\mathbf{k}}(\boldsymbol{\ell})
\right|$. Recall that $I_p=\{i:p\mid k_i\}$.
For $p\mid K$ and $i\in I_p$ define
\[
c_{i,p}:=v_p(h_i)\in\{0,1,2\},
\quad
e_{i,p}:=v_p(r_i)=2-c_{i,p}
\]
 and put
\[
E_p:=\max_{i\in I_p}e_{i,p},
\quad
C_p:=\sum_{i\in I_p}c_{i,p}.
\]
By the definition of $h_i$, $p\mid\ell_i$ if and only if
$c_{i,p}\geq1,$
and therefore $u_p(\boldsymbol\ell)
=
\#\{i\in I_p:c_{i,p}\geq1\}.$
For brevity, write this quantity as $u_p$.
We then have
\begin{equation}
\frac{r_1\cdots r_m/R}{N_1\cdots N_m/q}
=
\prod_{p\mid K}
p^{\,2-E_p-C_p}.
\label{eq:reduced-original-lcm-ratio}
\end{equation}
On the other hand, raising the arithmetic-factor
\eqref{eq:ArithmeticFactor} to the power $1+\vartheta$ gives by \lemref{lem:local-arithmetic-factor-bound}
\[
\left|
\mathfrak A_{\mathbf k}(\boldsymbol\ell)
\right|^{1+\vartheta}
\ll_{m,d,\vartheta}
q^{\varepsilon'}
\prod_{p\mid K}
p^{(1+\vartheta)\max(0,u_p-2)}.
\]
Thus, for a fixed local pattern \((c_{i,p})_{i\in I_p}\), the additional contribution of \(p\), relative to the base factor $(N_1\cdots N_m\,x^{m/2})/q,$
is given by the exponent
\[
\Theta_p
:=
(1+\vartheta)\max(0,u_p-2)
+
2-E_p-C_p.
\]
Since
\[
E_p
=
\max_{i\in I_p}(2-c_{i,p})
=
2-\min_{i\in I_p}c_{i,p},
\]
we may equivalently write
\begin{equation}
\Theta_p
=
(1+\vartheta)\max(0,u_p-2)
+
\min_{i\in I_p}c_{i,p}
-
C_p.
\label{eq:local-pattern-exponent}
\end{equation}
We claim that $\Theta_p\leq0$ for every local pattern.
Since each positive $c_{i,p}$ is at least $1$, removing
one summand from $C_p$ gives
\[
C_p-\min_{i\in I_p}c_{i,p}
\geq \max(0,u_p-1).
\]
Consequently,
\[
\Theta_p
\leq (1+\vartheta)\max(0,u_p-2)
-\max(0,u_p-1).
\]
If $u_p\leq2$, the right-hand side is nonpositive.
Otherwise, since $u_p\leq m$, we have
\[
\Theta_p
\leq (1+\vartheta)(u_p-2)-(u_p-1)
=-1+\vartheta(u_p-2)
\leq -1+\vartheta(m-2)<0,
\]
by \eqref{eq:thetaMoment}.
Thus no local pattern contributes a positive power of $p$. For each
$p\mid K$ there are at most $3^{|I_p|}\leq3^m$
possible local patterns. Summing over them therefore contributes at most
$C_{m,\vartheta}^{\omega(K)}
\ll_{m,\vartheta,\varepsilon''}q^{\varepsilon''}$
for every $\varepsilon''>0$, uniformly in $K$.
Combining this with
\eqref{eq:reduced-frequency-zero-sum},
\eqref{eq:reduced-original-lcm-ratio}, and the arithmetic-factor bound
gives, for $\varepsilon>0$
\begin{equation}
\sum_{\boldsymbol\ell\in\mathcal H(\mathbf k)}
\prod_{i=1}^{m}
F_x\left(\frac{\ell_i}{N_i}\right)
\left|
\mathfrak A_{\mathbf k}(\boldsymbol\ell)
\right|^{1+\vartheta}
\ll_{m,d,\vartheta}
q^{\varepsilon}
\frac{N_1\cdots N_m}{q}
x^{\frac m2}.
\label{eq:arithmetic-factor-moment}
\end{equation}
We now restrict to the bad set and apply Rankin's trick. By definition, \(|\mathfrak A_{\mathbf k}(\boldsymbol\ell)|>q^\eta\) for \(\boldsymbol\ell\in\mathcal B_\eta(\mathbf k)\), and hence
\[
|\mathfrak A_{\mathbf k}(\boldsymbol\ell)|
\le q^{-\eta\vartheta}
|\mathfrak A_{\mathbf k}(\boldsymbol\ell)|^{1+\vartheta}.
\]
Applying \eqref{eq:arithmetic-factor-moment} therefore yields
\[
\sum_{\boldsymbol{\ell}\in\mathcal{B}_{\eta}(\mathbf{k})}
\prod_{i=1}^{m}
F_x\left(\frac{\ell_i}{N_i}\right)
\left|
\mathfrak{A}_{\mathbf{k}}(\boldsymbol{\ell})
\right|
\ll_{m,d,\vartheta}
q^{-\eta\vartheta+\varepsilon}
\frac{N_1\cdots N_m}{q}
x^{\frac m2},
\]
which proves the lemma.
\end{proof}

\begin{lemma}
\label{lem:arithmetic-factor-higher-moment-bound}
With $\vartheta$ as in \eqref{eq:thetaMoment}, for every
$\varepsilon>0$, we have uniformly in $\mathbf{k}\in\mathbb N^m$
\[
\sum_{\boldsymbol{\ell}\in\mathcal H(\mathbf{k})}
\left(
\prod_{i=1}^m
\frac{F_x(\ell_i/N_i)}{N_i}
\right)
\left|
\mathfrak A_{\mathbf{k}}(\boldsymbol{\ell})
\right|^{1+\vartheta}
\ll_{m,d,\vartheta,\varepsilon}
K^\varepsilon.
\]
\end{lemma}

\begin{proof}
We use the normalized $L^1$-estimate for $F_x$, which we now recall. For $N\geq2$ and
$a\mid N$, applying
\cite[Lemma 2.2]{key-14} gives
\begin{equation}
\frac{1}{N}
\sum_{\substack{0<\ell<N\\a\mid\ell}}
F_x(\ell/N)
=
\frac{1}{N}
\sum_{0<b<N/a}
F_x\bigl(b/(N/a)\bigr)
\ll
\frac{\log(2N)}{a}.
\label{eq:normalized-frequency-divisibility}
\end{equation}
Fix $\mathbf{k}$ and consider a tuple
$\boldsymbol{\ell}\in\mathcal H(\mathbf{k})$.
We partition the tuples according to the exact local
divisibility sets
\[
U_p:=\{i\in I_p:p\mid\ell_i\},
\quad
u_p:=|U_p|.
\]
Thus the sets \((U_p)_{p\mid K}\) determine, for each \(p\mid K\) and \(i\in I_p\), whether \(p\mid\ell_i\). For each \(i\), define
\[
a_i:=\prod_{\substack{p\mid K\\ i\in U_p}}p.
\]
Then every tuple with this divisibility pattern satisfies \(a_i\mid\ell_i\). Moreover,
\begin{equation}
\prod_{i=1}^m a_i^{-1}
=
\prod_{p\mid K}p^{-u_p}.
\label{eq:divisibility-pattern-cost}
\end{equation}
For this divisibility pattern, \lemref{lem:local-arithmetic-factor-bound} gives
\begin{equation}
\left|
\mathfrak A_{\mathbf{k}}(\boldsymbol{\ell})
\right|^{1+\vartheta}
\leq
C_{m,d,\vartheta}^{\omega(K)}
\prod_{p\mid K}
p^{(1+\vartheta)\max(0,u_p-2)}.
\label{eq:arithmetic-factor-pattern-bound}
\end{equation}
The right-hand side depends only on the chosen sets $U_p$, and not on
the individual values of the $\ell_i$.
We now bound
\[\sum_{\substack{\boldsymbol{\ell}\ \mathrm{in\ the\ fixed\ pattern}\\
\boldsymbol{\ell}\in\mathcal H(\mathbf{k})}}
\prod_{i=1}^m\frac{F_x(\ell_i/N_i)}{N_i}.\]
Dropping the
zero-sum condition defining $\mathcal H(\mathbf{k})$, as well as the
conditions $p\nmid\ell_i$ corresponding to $i\notin U_p$, can only
increase the sum. Hence \eqref{eq:normalized-frequency-divisibility},
applied independently for $i=1,\ldots,m$, gives
\begin{align}
\sum_{\substack{\boldsymbol{\ell}\ \mathrm{in\ the\ fixed\ pattern}\\
\boldsymbol{\ell}\in\mathcal H(\mathbf{k})}}
\prod_{i=1}^m\frac{F_x(\ell_i/N_i)}{N_i}\leq
\prod_{i=1}^m
\bigg(
\frac{1}{N_i}
\sum_{\substack{0<\ell_i<N_i\\a_i\mid\ell_i}}
F_x(\ell_i/N_i)
\bigg)\ll_m
\left(\prod_{i=1}^m\log(2N_i)\right)
\prod_{p\mid K}p^{-u_p},
\label{eq:frequency-pattern-bound}
\end{align}
where the last inequality uses
\eqref{eq:divisibility-pattern-cost}.
Combining \eqref{eq:arithmetic-factor-pattern-bound} and
\eqref{eq:frequency-pattern-bound}, the contribution associated with a
fixed collection $(U_p)_{p\mid K}$ is bounded by
\[
C_{m,d,\vartheta}^{\omega(K)}
\prod_{i=1}^m\log(2N_i)
\prod_{p\mid K}
p^{(1+\vartheta)\max(0,u_p-2)-u_p}.
\]
We may now sum independently over the possible subsets
$U_p\subseteq I_p$. Thus the resulting local factor at $p$ is
\begin{equation}
\Sigma_p
:=
\sum_{U\subseteq I_p}
p^{(1+\vartheta)\max(0,|U|-2)-|U|}.
\label{eq:local-divisibility-sum}
\end{equation}
To estimate $\Sigma_p$, let $u=|U|$. The exponent in
\eqref{eq:local-divisibility-sum} is
\[
(1+\vartheta)\max(0,u-2)-u
=
\begin{cases}
0, & u=0,\\
-1, & u=1,\\
-2, & u=2,\\
-2+\vartheta(u-2), & u\geq3.
\end{cases}
\]
By \eqref{eq:thetaMoment}, $\kappa:=1-\vartheta(m-2)>0.$
Since $u\leq m$, for $u\geq3$ we have
\[
-2+\vartheta(u-2)
\leq
-1-\kappa.
\]
There are at most $\binom{m}{u}$ subsets of size $u$, and therefore
\begin{equation}
\Sigma_p
=
1+O_m(p^{-1})+O_{m,\vartheta}(p^{-1-\kappa}).
\label{eq:local-divisibility-sum-bound}
\end{equation}
Summing over all local divisibility patterns and using
\eqref{eq:local-divisibility-sum-bound}, we obtain
\[
\sum_{\boldsymbol{\ell}\in\mathcal H(\mathbf{k})}
\left(
\prod_{i=1}^m
\frac{F_x(\ell_i/N_i)}{N_i}
\right)
\left|
\mathfrak A_{\mathbf{k}}(\boldsymbol{\ell})
\right|^{1+\vartheta}
\ll_{m,d,\vartheta}
C_{m,d,\vartheta}^{\omega(K)}
\left(\prod_{i=1}^m\log(2N_i)\right)
\prod_{p\mid K}\Sigma_p.
\]
Since $N_i=k_i^2\leq K^2$, the logarithmic factor is $K^{o(1)}$.
Furthermore, $C_{m,d,\vartheta}^{\omega(K)}=K^{o(1)}$
and $\prod_{p\mid K}\Sigma_p=K^{o(1)}$ using \eqref{eq:local-divisibility-sum-bound} and the divisor bound.
Consequently the right-hand side is $K^{o(1)}$, and hence, for every
$\varepsilon>0$,
\[
\sum_{\boldsymbol{\ell}\in\mathcal H(\mathbf{k})}
\left(
\prod_{i=1}^m
\frac{F_x(\ell_i/N_i)}{N_i}
\right)
\left|
\mathfrak A_{\mathbf{k}}(\boldsymbol{\ell})
\right|^{1+\vartheta}
\ll_{m,d,\vartheta,\varepsilon}
K^\varepsilon,
\]
as claimed.
\end{proof}

We now combine \lemref{lem:bad-frequency-weighted-sum-bound} and \lemref{lem:arithmetic-factor-higher-moment-bound} by splitting the bad contribution according to the size of $q$.

\begin{lemma}
\label{lem:bad-frequency-power-saving}
There exists
$\delta_m^{\mathrm{bad}}>0$ such that, uniformly in $z>0$,
\[
\left|
\mathcal M_{m,z}^{\mathrm{bad}}(x)
\right|
\ll_{m,d}
x^{\frac m4-\delta_m^{\mathrm{bad}}}.
\]
\end{lemma}

\begin{proof}
Put $a:=\eta\vartheta,$
so that $0<a<1/2$. We split $\mathcal M_{m,z}^{\mathrm{bad}}$ according to $K>Y$ or $K\leq Y$ where $Y>0$ is a parameter to be optimized later.
For the range $K\leq Y$, Rankin's trick on the bad set gives
\[
|\mathfrak A_{\mathbf k}(\boldsymbol\ell)|
\leq
q^{-a}
|\mathfrak A_{\mathbf k}(\boldsymbol\ell)|^{1+\vartheta}.
\]
Applying \lemref{lem:arithmetic-factor-higher-moment-bound} shows that the
contribution of a fixed tuple $\mathbf k$ is
\[
\ll_{m,d,\vartheta,\varepsilon}
K^{-2a+\varepsilon}.
\]
By the divisor bound $\#\{\mathbf k:[k_1,\ldots,k_m]=K\}
\ll_{m,\varepsilon} K^\varepsilon.$
Summing over $K\leq Y$ gives
\begin{equation}
\left|
\mathcal M_{m,z}^{\mathrm{bad}}(x;K\leq Y)
\right|
\ll_{m,d,\vartheta,\varepsilon}
Y^{1-2a+\varepsilon}.
\label{eq:bad-small-lcm}
\end{equation}
For $K>Y$, we use \lemref{lem:bad-frequency-weighted-sum-bound}. After multiplying its
estimate by the factor $(N_1\cdots N_m)^{-1}$ occurring in
$\mathcal M_{m,z}^{\mathrm{bad}}(x)$, the contribution of a fixed
$\mathbf k$ is
\[
\ll_{m,d,\vartheta,\varepsilon}
x^{\frac m2}K^{-2-2a+\varepsilon}.
\]
Using the same estimate for the number of tuples with lcm $K$ and
summing the convergent tail, we obtain
\begin{equation}
\left|
\mathcal M_{m,z}^{\mathrm{bad}}(x;K>Y)
\right|
\ll_{m,d,\vartheta,\varepsilon}
x^{\frac m2}Y^{-1-2a+\varepsilon}.
\label{eq:bad-large-lcm}
\end{equation}
Choose $Y=x^{m/4}.$
If $z\leq Y$, only \eqref{eq:bad-small-lcm} is needed; if $z>Y$, we
combine \eqref{eq:bad-small-lcm} and \eqref{eq:bad-large-lcm}. In either
case,
\[
\left|
\mathcal M_{m,z}^{\mathrm{bad}}(x)
\right|
\ll_{m,d,\eta,\vartheta,\varepsilon}
x^{\frac m4-\frac{m\eta\vartheta}{2}+\varepsilon}.
\]
Taking $\varepsilon<m\eta\vartheta/4$ gives
\[
\left|
\mathcal M_{m,z}^{\mathrm{bad}}(x)
\right|
\ll_{m,d}
x^{\frac m4-\delta_m^{\mathrm{bad}}},
\quad
\delta_m^{\mathrm{bad}}
:=
\frac{m\eta\vartheta}{4}>0,
\]
as required.
\end{proof}

\subsection{Proof of \propref{prop:nonpaired-frequency-power-saving}}
\label{sec:completion-nonpaired-estimate}

Fix $\beta>0$ and suppose that $z\leq x^\beta$.
Choose $\vartheta$ as in \eqref{eq:thetaMoment},
and then fix $\eta>0$ sufficiently small in terms of
$m,\beta,\vartheta$ so that both
\lemref{lem:good-frequency-power-saving} and
\lemref{lem:bad-frequency-power-saving} apply.
The resulting power-saving exponents may depend on
these fixed parameters, but are independent of $x$
and $z$. By \lemref{lem:good-frequency-power-saving},
there exists $\delta_{m,\beta}^{\mathrm{good}}>0$ such that
\[
\left|
\mathcal M_{m,z}^{\mathrm{good}}(x)
\right|
\ll_m
x^{\frac m4-\delta_{m,\beta}^{\mathrm{good}}}.
\]
By \lemref{lem:bad-frequency-power-saving}, there exists
$\delta_m^{\mathrm{bad}}>0$ such that
\[
\left|
\mathcal M_{m,z}^{\mathrm{bad}}(x)
\right|
\ll_{m,d}
x^{\frac m4-\delta_m^{\mathrm{bad}}}.
\]
Since
\[
\mathcal M_{m,z}^{\mathrm{nonpair}}(x)
=
\mathcal M_{m,z}^{\mathrm{good}}(x)
+
\mathcal M_{m,z}^{\mathrm{bad}}(x),
\]
it follows, with $\delta_{m,\beta}
:=
\min\{
\delta_{m,\beta}^{\mathrm{good}},
\delta_m^{\mathrm{bad}}
\}>0,$
that $\left|
\mathcal M_{m,z}^{\mathrm{nonpair}}(x)
\right|
\ll_{m,d}
x^{\frac m4-\delta_{m,\beta}}.$
This proves \propref{prop:nonpaired-frequency-power-saving}.\qed

\section{Reduced Fourier expansion and the paired contribution}
\label{sec:reduced-local-moments}
The purpose of this section is to rewrite the harmonic
expression $\mathcal M_{m,z}(x)$ from
\propref{prop:local-model-moment-representation}
in a form better suited for extracting the main
contribution coming from exactly paired frequencies.
For this purpose, we consider the Fourier expansion
of the local discrepancy $E_g(x)$ obtained by
restricting the products in \eqref{eq:local-discrepancy}
to finitely many primes and writing the resulting
frequencies in lowest terms.

\subsection{Reduced finite-prime Fourier expansion}
\label{sec:finite-prime-fourier-expansion}

We now derive a second representation of the
$m$th moment of the local discrepancy $E_g(x)$.
The representation in
\propref{prop:local-model-moment-representation} is obtained by
truncating according to the
least common multiple $K$. For the analysis of the main term, it is more
convenient instead to keep only finitely many prime coordinates at
first, write the resulting frequencies in lowest terms, and
then let the prime cutoff tend to infinity.

We start with additional notations. For a positive integer $r$, put
\[
\mathcal C(r)
:=
\left\{
\frac ar:a\in\mathbb Z
\right\}/\mathbb Z,
\quad
\mathcal R(r)
:=
\left\{
\frac ar\in\mathcal C(r):(a,r)=1
\right\}.
\]
Recall the set $\mathscr R$ defined in
\eqref{eq:R-sf-decomposition}, consisting of the possible reduced denominators of fractions $\ell/k^2$
with $k$ squarefree.
For $g_p\in V_{p^2}^{\sharp}$, put
\[
U_p(g_p)
:=
\left\{
u\in\mathbb Z/p^2\mathbb Z:
g_p(u) \not\equiv 0\bmod p^2
\right\}.
\]
We decompose each frequency $\sigma\in\mathbb Q/\mathbb Z$ as
\[ \sigma = \sum_{p} \sigma_p\]
where $\sigma_p$ is a fraction with denominator a power of $p$; we call $\sigma_p$ the $p$-primary component of $\sigma$. If $p\mid r$ and $\sigma\in\mathcal R(r)$,
then $p^2\sigma_p=0$ in $\mathbb Q/\mathbb Z$, so
\[
u\longmapsto e(-\sigma_pu),
\quad u\in\mathbb Z/p^2\mathbb Z,
\]
is a well-defined additive character. For $r\in\mathscr R$ and
$\sigma\in\mathcal R(r)$, define
\[
\Lambda_p(g_p,\sigma,r)
:=
\begin{cases}
\displaystyle
1-\frac{\rho_{g_p}(p^2)}{p^2},
& p\nmid r,\\[2mm]
\displaystyle
\int_{\mathbb Z/p^2\mathbb Z\setminus U_p(g_p)}
e(-\sigma_pu)\,du,
& p\mid r,
\end{cases}
\]
where $du$ is the normalized counting measure on
$\mathbb Z/p^2\mathbb Z$. For $y\ge2$, put
\[
\Lambda^{(y)}(g,\sigma,r)
:=
\prod_{p\le y}\Lambda_p(g_p,\sigma,r),
\]
whenever every prime divisor of $r$ is at most $y$, and also define
\[
\Lambda(g,\sigma,r)
:=
\prod_p\Lambda_p(g_p,\sigma,r).
\]
For a fixed pair $(r,\sigma)$, the latter product converges almost surely and
in $L^q$ for every $q$. Indeed, only finitely many factors with
$p\mid r$ are present, while the remaining product is a subproduct of the Euler product defining $\mathfrak S(g)$, which is convergent by \lemref{lem:local-model-mobius-expansion}.
If $p\mid r$ then a standard calculation shows that
\[
\Lambda_p(g_p,\sigma,r)
=
p^{-2}T_{g_p}(-\sigma_p;p^2).
\]
For $y\ge2$, put $Q_y:=\prod_{p\le y}p^2$.
Note that we can write \eqref{eq:E_g-truncated} as
\[
E_g^{(y)}(x)
=
\sum_{n\le x}
\prod_{p\le y}
\mathds{1}_{U_p(g_p)}(n)
-
x
\prod_{p\le y}
\left(
1-\frac{\rho_{g_p}(p^2)}{p^2}
\right).
\]
The next lemma gives a reduced Fourier
expansion of $E_g^{(y)}(x)$.

\begin{lemma}\label{lem:finite-prime-local-model-convergence}
For every fixed $m\ge1$ and $x$, we have
\begin{equation}
E_g^{(y)}(x)
=
\sum_{\substack{r\in\mathscr R,\ r>1\\
p\mid r\Rightarrow p\le y}}
\mu(\operatorname{rad}(r))
\sum_{\sigma\in\mathcal R(r)}
E_x(\sigma)\Lambda^{(y)}(g,\sigma,r).
\label{eq:FinitePrimeFourier}
\end{equation}
\end{lemma}

\begin{proof}
For the Fourier expansion, regard
\[
n\longmapsto
\prod_{p\le y}\mathds{1}_{U_p(g_p)}(n)
\]
as a function on $\mathbb Z/Q_y\mathbb Z$. Its Fourier expansion is
\[
\prod_{p\le y}\mathds{1}_{U_p(g_p)}(n)
=
\sum_{\sigma\in\mathcal C(Q_y)}
\widehat{\mathds{1}}_{g,y}(\sigma)e(\sigma n),
\]
where
\[
\widehat{\mathds{1}}_{g,y}(\sigma)
:=
\frac1{Q_y}
\sum_{u\bmod Q_y}
\prod_{p\le y}
\mathds{1}_{U_p(g_p)}(u)e(-\sigma u).
\]
At $\sigma=0$, the Chinese remainder theorem gives
\[
\widehat{\mathds{1}}_{g,y}(0)
=
\prod_{p\le y}
\left(
1-\frac{\rho_{g_p}(p^2)}{p^2}
\right),
\]
so the zero-frequency term is exactly canceled by the centering term
in $E_g^{(y)}(x)$.
Every nonzero $\sigma\in\mathcal C(Q_y)$ has a unique reduced
representation
\[
\sigma=\frac ar,
\quad
(a,r)=1,\quad r\mid Q_y,\quad r>1.
\]
The possible reduced denominators are precisely the
$r\in\mathscr R$ supported on primes $p\le y$. Under the Chinese
remainder decomposition,
\[
\widehat{\mathds{1}}_{g,y}(\sigma)
=
\prod_{p\le y}
\left(
\frac1{p^2}
\sum_{\substack{u\bmod p^2\\u\in U_p(g_p)}}
e(-\sigma_pu)
\right).
\]
If $p\nmid r$, then $\sigma_p=0$, and the local factor is
\[
1-\frac{\rho_{g_p}(p^2)}{p^2}
=
\Lambda_p(g_p,\sigma,r).
\]
If $p\mid r$, then the $p$-component $\sigma_p$ is
nontrivial modulo $p^2$, since $r$ is the reduced
denominator of $\sigma$. Hence, by orthogonality
of additive characters, $\sum_{u\bmod p^2}e(\sigma_p u)=0$ and hence
\[
\frac1{p^2}
\sum_{\substack{u\bmod p^2\\u\in U_p(g_p)}}
e(-\sigma_pu)
=
-\Lambda_p(g_p,\sigma,r).
\]
It then follows that
\[
\widehat{\mathds{1}}_{g,y}(\sigma)
=
(-1)^{\omega(r)}
\Lambda^{(y)}(g,\sigma,r)
=
\mu(\operatorname{rad}(r))
\Lambda^{(y)}(g,\sigma,r),
\]
which proves \eqref{eq:FinitePrimeFourier}.
\end{proof}

For $y\ge2$, set
\[
\mathscr R_y
:=
\left\{
r\in\mathscr R:
r>1,\ p\mid r\Rightarrow p\le y
\right\}.
\]
We now raise \eqref{eq:FinitePrimeFourier} to the $m$th power. After
averaging over $g$, translation invariance forces the reduced
frequencies to satisfy the same zero-sum relation as in the
unreduced local harmonic representation.

\begin{proposition}\label{prop:finite-prime-local-moment-expansion}
For $y\ge2$, define
\[
\mathcal C_m^{(y)}(g,x)
:=
\sum_{r_1,\ldots,r_m\in\mathscr R_y}
\prod_{i=1}^m\mu(\operatorname{rad}(r_i))
\sum_{\substack{
\sigma_i\in\mathcal R(r_i)\ (1\le i\le m)\\
\sigma_1+\cdots+\sigma_m\in\mathbb Z
}}
\prod_{i=1}^m
E_x(\sigma_i)
\Lambda^{(y)}(g,\sigma_i,r_i).
\]
Then
\begin{equation}
\int_{V_\square^\sharp}E_g(x)^m\,dg
=
\lim_{y\to\infty}
\int_{V_\square^\sharp}
\mathcal C_m^{(y)}(g,x)\,dg.
\label{eq:LocalMomentLimit}
\end{equation}
\end{proposition}

\begin{proof}
For fixed $y$, \eqref{eq:FinitePrimeFourier} is a finite sum. Raising
it to the $m$th power and integrating gives
\[
\int_{V_\square^\sharp}E_g^{(y)}(x)^m\,dg
=
\sum_{r_1,\ldots,r_m\in\mathscr R_y}
\prod_{i=1}^m\mu(\operatorname{rad}(r_i))
\sum_{\substack{\sigma_i\in\mathcal R(r_i)\\1\le i\le m}}
\prod_{i=1}^mE_x(\sigma_i)\,
I_y(\mathbf r,\boldsymbol\sigma),
\]
where
\[
I_y(\mathbf r,\boldsymbol\sigma)
:=
\int_{V_\square^\sharp}
\prod_{i=1}^m
\Lambda^{(y)}(g,\sigma_i,r_i)\,dg.
\]
For $h\bmod Q_y$, define
\[
(\tau_hg)_p(t):=g_p(t+h),
\quad p\le y.
\]
Translation preserves the product measure and gives
\[
\Lambda^{(y)}(\tau_hg,\sigma,r)
=
e(h\sigma)\Lambda^{(y)}(g,\sigma,r).
\]
Consequently,
\[
I_y(\mathbf r,\boldsymbol\sigma)
=
e\left(
h\sum_{i=1}^m\sigma_i
\right)
I_y(\mathbf r,\boldsymbol\sigma)
\]
for every $h\bmod Q_y$. Thus when $\sigma_1+\cdots+\sigma_m\notin\mathbb Z$, we have $I_y(\mathbf r,\boldsymbol\sigma)= 0$.
It follows that
\[
\int_{V_\square^\sharp}E_g^{(y)}(x)^m\,dg
=
\int_{V_\square^\sharp}
\mathcal C_m^{(y)}(g,x)\,dg.
\]
Finally,
\lemref{lem:local-model-mobius-expansion} gives
$E_g^{(y)}(x)\to E_g(x)$ in $L^m$. Since
$|E_g^{(y)}(x)|,|E_g(x)|\le2x$, we have
\[
\int_{V_\square^\sharp}
\left|
E_g^{(y)}(x)^m-E_g(x)^m
\right|\,dg
\longrightarrow0.
\]
Taking $y\to\infty$ proves \eqref{eq:LocalMomentLimit}.
\end{proof}

\subsection{Reduction to exactly paired frequencies}
\label{sec:exactly-paired-contribution}
Proposition
\ref{prop:finite-prime-local-moment-expansion} provides a representation of
the \emph{full} local-model moment. We now isolate its exactly paired
part and show that all remaining frequencies contribute negligibly. 

For even $m$, let $\mathcal C_{m,\mathrm{pair}}^{(y)}(g,x)$
denote the restriction of $\mathcal C_m^{(y)}(g,x)$ to those reduced
frequency tuples $(\sigma_1,\ldots,\sigma_m)$ for which the frequencies are exactly paired. For odd $m$, put
\[
\mathcal C_{m,\mathrm{pair}}^{(y)}(g,x):=0.
\]
This is the same notion of exact pairing as in
\secref{sec:nonpaired-frequency-analysis}: replacing a rational
frequency by its reduced representative does not change its value in
$\mathbb Q/\mathbb Z$, and hence does not change whether it can be
paired with its negative.

Before comparing this object with the truncated frequency expression,
we record that the large-$K$ estimate of
\lemref{lem:large-lcm-tail-local-model} remains valid after restricting
to exactly paired Fourier frequencies.

\begin{lemma}\label{lem:paired-large-lcm-tail-local-model}
Fix $m\ge1$ and $\varepsilon>0$.
For $y\ge2$ and $Z\ge1$, let
$\mathcal T_{m,y}^{\mathrm{pair}}(x;Z)$ denote the
contribution of exactly paired frequency tuples with
$K=[k_1,\ldots,k_m]>Z$ to the finite-prime Fourier
expansion of
\[
\frac{1}{\zeta(2d+2)}
\int_{V_\square^\sharp}E_g^{(y)}(x)^m\,dg.
\]
Then, uniformly for $y\ge2$ and $Z\ge1$.
\[
\left|\mathcal T_{m,y}^{\mathrm{pair}}(x;Z)\right|
\ll_{m,d,\varepsilon}
x^m Z^{-1+\varepsilon}.
\]
\end{lemma}

\begin{proof}
For odd $m$ there are no exactly paired tuples. Suppose that $m=2k$
and fix a perfect matching
\[
\mathcal P=\{\{i_1,j_1\},\ldots,\{i_k,j_k\}\}.
\]
Let $\mathcal T_{\mathcal P}(Z,y)$ denote the contribution of the paired
frequency tuples according to $\mathcal P$ and arising from
$(k_1,\ldots,k_m)$ with $K>Z$.
Since all frequencies occurring in $E_g^{(y)}(x)$
have denominators dividing $Q_y$, orthogonality
of additive characters gives
\[
\mathds{1}_{\sigma_{i_s}+\sigma_{j_s}=0}
=
\frac1{Q_y}
\sum_{h_s\bmod Q_y}
e\!\left(h_s(\sigma_{i_s}+\sigma_{j_s})\right).
\]
Using $\Lambda^{(y)}(\tau_hg,\sigma,r)
=
e(h\sigma)\Lambda^{(y)}(g,\sigma,r),$
and applying this identity independently to the $k$ pairs, We obtain
\[
\mathcal T_{\mathcal P}(Z,y)
=
\frac{1}{Q_y^k}
\sum_{h_1,\ldots,h_k\bmod Q_y}
\frac{1}{\zeta(2d+2)}
\sum_{\substack{
k_1,\ldots,k_m\ge1\\
K=[k_1,\ldots,k_m]>Z\\
p\mid K \Rightarrow p\le y
}}
\prod_{i=1}^m\mu(k_i)
\int_{V_\square^\sharp}
\prod_{s=1}^k
\Delta_{\tau_{h_s}g}(k_{i_s};x)
\Delta_{\tau_{h_s}g}(k_{j_s};x)\,dg.
\] For every $h$,
\[
\left|\Delta_{\tau_hg}(k;x)\right|
\le
\sum_{n\le x}A_g(k;n+h),
\]
because
$(\tau_hg)_p(n)=g_p(n+h)$ and translation does not change
$\rho_g(k^2)$. Hence
\[
|\mathcal T_{\mathcal P}(Z,y)|
\le
\frac{1}{Q_y^k}
\sum_{\mathbf h\bmod Q_y}
\sum_{n_1,\ldots,n_m\le x}
\sum_{K>Z}
\mathbb E_g
\prod_{s=1}^k
A_g(k_{i_s};n_{i_s}+h_s)
A_g(k_{j_s};n_{j_s}+h_s).
\]
We have enlarged the $k_i$-sum by dropping the restriction
$p\mid k_i\Rightarrow p\le y$. The summands are now nonnegative, and
\lemref{lem:large-lcm-tail-local-model} applies uniformly to the
shifted residue variables. Thus the inner $K>Z$ sum is
$O_{m,d,\varepsilon}(Z^{-1+\varepsilon})$, uniformly in the shifted vector $\mathbf h$
and $y$. Summing over $n_1,\ldots,n_m\le x$ and averaging over
$\mathbf h$ gives
\[
\mathcal T_{\mathcal P}(Z,y)
\ll_{m,d,\varepsilon}
x^mZ^{-1+\varepsilon}.
\]
Finally, summing over the finitely many perfect matchings proves the
claim.
\end{proof}

We are now ready to extract the exactly paired contribution in the local model.

\begin{proposition}\label{prop:local-model-exact-pairing-reduction}
Fix $m\ge1$. There exists $\delta_m>0$ such that
\begin{equation}
\frac{1}{\zeta(2d+2)}
\int_{V_\square^\sharp}E_g(x)^m\,dg
=
\frac{1}{\zeta(2d+2)}
\lim_{y\to\infty}
\int_{V_\square^\sharp}
\mathcal C_{m,\mathrm{pair}}^{(y)}(g,x)\,dg
+
O_{m,d}\left(
x^{m/4-\delta_m}
\right).
\label{eq:ReducedPairedLocalMoment}
\end{equation}
Consequently, whenever $H\ge x^{C_dm^2}$,
\[
\mathrm{MP}_m(x,H)
=
\frac{1}{\zeta(2d+2)}
\lim_{y\to\infty}
\int_{V_\square^\sharp}
\mathcal C_{m,\mathrm{pair}}^{(y)}(g,x)\,dg
+
o_{m,d}\left(x^{m/4}\right).
\]
\end{proposition}

\begin{proof}
Put $z:=x^{\frac{3m+1}{4}}.$
We first use the lcm-truncated local harmonic representation to show
that frequency tuples which are not exactly paired contribute negligibly on the scale $x^{m/4}$.
By \propref{prop:local-model-moment-representation},
\[
\frac{1}{\zeta(2d+2)}
\int_{V_\square^\sharp}E_g(x)^m\,dg
=
\mathcal M_{m,z}(x)
+
O_{m,d,\varepsilon}
\left(
x^mz^{-1+\varepsilon}
\right).
\]
If $m\ge3$, then
\propref{prop:nonpaired-frequency-power-saving}, applied with
$\beta=(3m+1)/4$, gives
\[
\mathcal M_{m,z}(x)
=
\mathcal M_{m,z}^{\mathrm{pair}}(x)
+
O_{m,d}
\left(
x^{m/4-\delta_m^{(1)}}
\right)
\]
for some $\delta_m^{(1)}>0$. For $m=1$, there is no nonzero frequency
satisfying the zero-sum relation, while for $m=2$ every surviving
zero-sum pair consists of opposite frequencies. Thus the same
conclusion holds in these two cases, with
$\mathcal M_{m,z}^{\mathrm{pair}}(x)=0$ when $m=1$.
Since
\[
x^mz^{-1+\varepsilon}
=
x^{m/4-1/4+\frac{3m+1}{4}\varepsilon},
\]
we obtain, for a suitable $\delta_m^{(2)}>0$,
\begin{equation}
\frac{1}{\zeta(2d+2)}
\int_{V_\square^\sharp}E_g(x)^m\,dg
=
\mathcal M_{m,z}^{\mathrm{pair}}(x)
+
O_{m,d}
\left(
x^{m/4-\delta_m^{(2)}}
\right).
\label{eq:local-via-original-paired}
\end{equation}
We next show that $\mathcal M_{m,z}^{\mathrm{pair}}(x)$
also arises as the contribution of exactly paired frequencies
with $K\le z$ in the finite-prime Fourier expansion. Before reducing frequencies, the finite-prime M\"obius
expansion of $E_g^{(y)}(x)$ is
\[
E_g^{(y)}(x)
=
\sum_{\substack{k\ge1\\p\mid k\Rightarrow p\le y}}
\mu(k)\Delta_g(k;x).
\]
Since $x$, and hence $z$, is fixed when taking $y\to\infty$,
we may restrict to $y\ge z$. In this range, every tuple
with $K\le z$ occurs in the finite-prime expansion. For
this range, the identity $\mathcal L_{m,z}(x)=\mathcal M_{m,z}(x)$
from
\propref{prop:local-model-moment-representation} identifies the
exactly paired contribution with
$\mathcal M_{m,z}^{\mathrm{pair}}(x)$. Reducing the individual rational
frequencies does not alter this contribution, because exact pairing
depends only on their values in $\mathbb Q/\mathbb Z$.

The remaining exactly paired terms have $K>z$. By
\lemref{lem:paired-large-lcm-tail-local-model}, their total contribution
is
\[
O_{m,d,\varepsilon}
\left(
x^mz^{-1+\varepsilon}
\right),
\]
uniformly in $y\ge z$. Moreover, the same lemma shows that the paired finite-prime
contributions form a Cauchy sequence as \(y\to\infty\): if \(y_2>y_1\),
then, before reducing the frequencies, every term in their difference
contains at least one modulus \(k_i\) divisible by a prime \(p>y_1\),
and hence has \(K>y_1\). Therefore the limit
\[
\lim_{y\to\infty}
\int_{V_\square^\sharp}
\mathcal C_{m,\mathrm{pair}}^{(y)}(g,x)\,dg
\]
exists, and
\begin{equation}
\frac{1}{\zeta(2d+2)}
\lim_{y\to\infty}
\int_{V_\square^\sharp}
\mathcal C_{m,\mathrm{pair}}^{(y)}(g,x)\,dg
=
\mathcal M_{m,z}^{\mathrm{pair}}(x)
+
O_{m,d,\varepsilon}
\left(
x^mz^{-1+\varepsilon}
\right).
\label{eq:reduced-paired-via-original-paired}
\end{equation}
Combining
\eqref{eq:local-via-original-paired} and
\eqref{eq:reduced-paired-via-original-paired}, and taking
$\varepsilon>0$ sufficiently small, proves
\eqref{eq:ReducedPairedLocalMoment}. The final assertion follows from
\propref{prop:polynomial-local-moment-comparison} with $\Phi\equiv1$.
\end{proof}

\section{The paired main term and the limiting variance}
\label{sec:paired-main-term-limiting-variance}

By \propref{prop:local-model-exact-pairing-reduction}, the non-paired
frequencies have already been removed with a negligible error.
The purpose of this section is therefore to evaluate the remaining
exactly paired contribution. The outcome is the following fixed-moment
asymptotic. Let $C\colon V_\square\to[0,\infty)$ be the function constructed in
\secref{sec:limiting-variance-construction}.

\begin{proposition}\label{prop:fixed-moment-asymptotics}
Fix $m\ge1$. There exists a constant $C_d>0$ such that, whenever
$H\ge x^{C_dm^2}$,
\[
\mathrm{MP}_m(x,H)
=
\begin{cases}
\displaystyle
\mu_m x^{m/4}
\int_{V_\square} C(g)^{m/2}\,dg
+
o\left(x^{m/4}\right),
& m\ \mathrm{even},\\[2mm]
\displaystyle
o\left(x^{m/4}\right),
& m\ \mathrm{odd}.
\end{cases}
\]
\end{proposition}

The proof proceeds in three steps, which we briefly indicate. We first show that \emph{each} exact pair
contributes the \emph{same} non-negative one-pair sum $S(g,x)$. We then
determine the limiting behavior of $x^{-1/2}S(g,x)$ and thereby
construct $C(g)$; this is by far the most technical and interesting step. Finally, we show that the contribution of tuples of exact pairs where two pairs have the same denominators, is negligible.

The next subsection elaborates on these steps, and gives a complete proof of \propref{prop:fixed-moment-asymptotics} up to two results which occupy the rest of the section.

\subsection{Exact frequency pairings and the one-pair contribution}
\label{sec:exact-frequency-pairings}

By \propref{prop:local-model-exact-pairing-reduction}, it remains to
analyze the limiting exactly paired part
\[
\lim_{y\to\infty}
\int_{V_\square^\sharp}
\mathcal C_{m,\mathrm{pair}}^{(y)}(g,x)\,dg
\]
of the reduced finite-prime expansion. For odd $m$ there are no exactly paired tuples, so in this subsection
we assume that $m=2k$. Fix a perfect matching $\mathcal P
=
\bigl\{
\{i_1,j_1\},\ldots,\{i_k,j_k\}
\bigr\}$
of $\{1,\ldots,2k\}$. A tuple paired according to $\mathcal P$
satisfies
\[
\sigma_{j_s}=-\sigma_{i_s}
\quad (1\le s\le k).
\]
Since the frequencies are written in lowest terms, opposite
frequencies have the same reduced denominator. After relabeling the
indices, we may therefore write $r_{k+j}=r_j$ and $\sigma_{k+j}=-\sigma_j$ for all $1\le j\le k$.

We first restrict to those terms of
$\mathcal C_{2k,\mathrm{pair}}^{(y)}(g,x)$ for which the pair-denominators $r_1,\ldots,r_k$ are distinct. Substituting these relations into the finite-prime expansion
of \propref{prop:finite-prime-local-moment-expansion},
the contribution corresponding to the fixed matching $\mathcal P$ is
\[
\int_{V_\square^\sharp}
\sum_{\substack{
r_1,\ldots,r_k\in\mathscr R_y\\
r_1,\ldots,r_k\ \mathrm{distinct}
}}
\prod_{j=1}^k
\left[
\sum_{\sigma\in\mathcal R(r_j)}
E_x(\sigma)E_x(-\sigma)
\Lambda^{(y)}(g,\sigma,r_j)
\Lambda^{(y)}(g,-\sigma,r_j)
\right]
\,dg.
\]
Thus, for a fixed matching, the exactly paired contribution separates
into one factor for each of the $k$ pairs.

For every reduced frequency $\sigma$,
\[
E_x(-\sigma)=\overline{E_x(\sigma)},
\quad
\Lambda^{(y)}(g,-\sigma,r)
=
\overline{\Lambda^{(y)}(g,\sigma,r)}.
\]
Hence each one-pair factor is non-negative:
\[
E_x(\sigma)E_x(-\sigma)
\Lambda^{(y)}(g,\sigma,r)
\Lambda^{(y)}(g,-\sigma,r)
=
|E_x(\sigma)|^2
|\Lambda^{(y)}(g,\sigma,r)|^2.
\]
For $r=ab^2\in\mathscr R$, we therefore define the limiting
contribution associated with one exact pair of reduced denominator
$r$ by
\begin{equation}
A_r(g,x)
:=
\sum_{\sigma\in\mathcal R(r)}
|E_x(\sigma)|^2
|\Lambda(g,\sigma,r)|^2.
\label{eq:Ar-one-pair}
\end{equation}
Here, for each fixed $r$ and $\sigma$,
\[
\Lambda^{(y)}(g,\sigma,r)
\longrightarrow
\Lambda(g,\sigma,r)
\quad (y\to\infty),
\]
with $\Lambda(g,\sigma,r)$ as defined in
\secref{sec:finite-prime-fourier-expansion}.
Summing over all possible reduced denominators gives the total
contribution of one exact pair:
\begin{equation}
S(g,x)
:=
\sum_{\substack{r\in\mathscr R\\ r>1}}
A_r(g,x).
\label{eq:S-one-pair-sum}
\end{equation}
We justify the passage from the finite-prime
one-pair contributions to the preceding sum.
Fix $g\in V_\square^\sharp$ and consider
\[
T:=\prod_p\mathbb Z/p^2\mathbb Z
\]
with its normalized Haar measure. Define
\[
f_y(u):=\prod_{p\le y}
\mathds{1}_{\{g_p(u_p)\not\equiv0\bmod p^2\}},
\quad
f(u):=\lim_{y\to\infty}f_y(u).
\]
Since $0\le f_y\le1$ and $f_y\to f$
pointwise, dominated convergence gives
$\|f_y-f\|_{L^2(T)}\to0$.
Identifying the dual group $\widehat T$
with the rational frequencies having
reduced denominators in
$\mathscr R\cup\{1\}$, Parseval gives
\[
\sum_{\sigma\in\widehat T}
|\widehat f_y(\sigma)-\widehat f(\sigma)|^2
\longrightarrow0.
\]
For $\sigma\in\mathcal R(r)$, $r>1$,
the nonzero Fourier coefficients are
\[
\widehat f_y(\sigma)
=\mu(\operatorname{rad}(r))
\Lambda^{(y)}(g,\sigma,r)
\]
when $r\in\mathscr R_y$, and zero
otherwise, while
\[
\widehat f(\sigma)
=\mu(\operatorname{rad}(r))
\Lambda(g,\sigma,r).
\]
Since $|E_x(\sigma)|\le x$, Parseval
and Cauchy--Schwarz imply
\[
\begin{aligned}
&\sum_{\sigma\ne0}|E_x(\sigma)|^2
\left|
|\widehat f_y(\sigma)|^2
-|\widehat f(\sigma)|^2
\right|\le
2x^2\|f_y-f\|_{L^2(T)}
\longrightarrow0.
\end{aligned}
\]
Thus, defining $A_r^{(y)}(g,x)$
as the finite-prime analogue of
$A_r(g,x)$, and setting it to zero
for $r\notin\mathscr R_y$, we have
\[
\sum_{\substack{r\in\mathscr R\\r>1}}
|A_r^{(y)}(g,x)-A_r(g,x)|
\longrightarrow0.
\]
Moreover, Parseval gives
\[
\sum_{r>1}A_r^{(y)}(g,x)\le x^2,
\quad
S(g,x)\le x^2.
\]
Hence the sums over distinct
pair-denominators converge for every
fixed $x$, and dominated convergence
justifies the corresponding passage
under the integral over $g$.
Consequently, after letting $y\to\infty$, the part associated with the
fixed matching $\mathcal P$ and with distinct pair-denominators is
\begin{equation}
\int_{V_\square^\sharp}
\sum_{\substack{
r_1,\ldots,r_k\in\mathscr R\\
r_1,\ldots,r_k>1\ \mathrm{distinct}
}}
\prod_{j=1}^k A_{r_j}(g,x)\,dg.
\label{eq:fixed-pairing-distinct-denominators}
\end{equation}
On the other hand,
\[
S(g,x)^k
=
\sum_{1<r_1,\ldots,r_k\in\mathscr R}
\prod_{j=1}^k A_{r_j}(g,x).
\]
Thus the only difference between
\eqref{eq:fixed-pairing-distinct-denominators} and
\[
\int_{V_\square^\sharp}S(g,x)^k\,dg
\]
is the contribution of tuples for which at least two pair
denominators coincide. In the short \S\ref{sec:completion-paired-contribution} we show that this contribution is negligible, through the following proposition.
\begin{proposition}\label{prop:repeated-denominators-negligible}
For every fixed $k\ge2$,
\[
\int_{V_{\square}^{\sharp}}\Bigg(S\left(g,x\right)^{k}-\sum_{\substack{r_{1},...,r_{k}\in\mathscr{R}\\
r_{1},...,r_{k}>1\ \mathrm{distinct}
}
}\prod_{i=1}^{k}A_{r_{i}}\left(g,x\right)\Bigg)dg=o\left(x^{\frac{k}{2}}\right).\]
\end{proposition}
From \propref{prop:repeated-denominators-negligible} it follows that, for each fixed perfect matching,
\begin{equation}
\eqref{eq:fixed-pairing-distinct-denominators}
=
\int_{V_\square^\sharp}S(g,x)^k\,dg
+
o\left(x^{k/2}\right).
\label{eq:fixed-pairing-via-S}
\end{equation}
The remaining task is therefore to determine the asymptotic behavior
of the single quantity $S(g,x)$. In
\secref{sec:limiting-variance-construction}, which is the main novel probabilistic part of the paper, we construct a
non-negative function
\[
C\colon V_\square^\sharp\longrightarrow[0,\infty)
\]
and prove, for every fixed $k\ge1$,
\begin{equation}
x^{-1/2}S(g,x)
\longrightarrow
C(g)
\quad\text{in }L^k.
\label{eq:S-to-C}
\end{equation}
In particular,
\begin{equation}
\int_{V_\square^\sharp}S(g,x)^k\,dg
=
x^{k/2}
\int_{V_\square^\sharp}C(g)^k\,dg
+
o\left(x^{k/2}\right).
\label{eq:S-moment-asymptotic}
\end{equation}
There are $\mu_{2k}=(2k-1)!!$
perfect matchings $\mathcal{P}$ of the $2k$ frequency indices. Whenever the pair
denominators are distinct, the frequency tuple determines its matching
uniquely. Therefore
\eqref{eq:fixed-pairing-via-S} and
\eqref{eq:S-moment-asymptotic}, together with \propref{prop:repeated-denominators-negligible}, will give the complete exactly paired
contribution
\[
\mu_{2k}x^{k/2}
\int_{V_\square^\sharp}C(g)^k\,dg
+
o\left(x^{k/2}\right).
\]
Finally, extend $C$ to $V_\square$ by setting
\begin{equation}\label{eq:Cextend}
C(g):=0
\quad
g\in V_\square\setminus V_\square^\sharp.
\end{equation}
Since $dg$ on $V_\square^\sharp$ is the normalized product probability
measure and $V_\square^\sharp$ has Haar mass
$\zeta(2d+2)^{-1}$ in $V_\square$,
\begin{equation}
\frac{1}{\zeta(2d+2)}
\int_{V_\square^\sharp}C(g)^k\,dg
=
\int_{V_\square}C(g)^k\,dg.
\label{eq:C-sharp-to-Haar}
\end{equation}
Thus \propref{prop:fixed-moment-asymptotics} will follow from
\propref{prop:local-model-exact-pairing-reduction} once
\eqref{eq:S-to-C} and \propref{prop:repeated-denominators-negligible} have been
established. This is the content of the rest of the section.

\subsection{Asymptotics of the exact-pair contribution and construction of the limiting variance}
\label{sec:limiting-variance-construction}

It remains to determine the asymptotic behavior of $S(g,x)$, which is the contribution of opposite reduced-frequency pairs over all reduced denominators. 
Let
\[
V(t):=\frac{\sin(\pi t)}{\pi t},
\quad
f(t):=\frac{\left|V(1/t)\right|^2}{t^2}.
\]
For $\lambda\geq1$ and $r=ab^2\in\mathscr R$ with $(\lambda,r)=1$,
define
\[
c_{g,\lambda}(r)
:=
r^2
\left|\Lambda\left(g,\frac{\lambda}{r},r\right)\right|^2,
\]
and put
\begin{equation}
A_{g,\lambda}(Y)
:=
\sum_{\substack{r\leq Y\\
r\in\mathscr R\\
(\lambda,r)=1}}
c_{g,\lambda}(r).
\label{eq:A_g(Y)}    
\end{equation}

The first step is to express $S(g,x)$ in terms of these coefficients.
The following identity is the analogue of the representation used in
the proof of \cite[Lemma 3.4]{key-14}; it allows us to reduce the study of $S(g,x)$ to that of partial sums of $c_{g,\lambda}$.

\begin{lemma}\label{lem:one-pair-sum-decomposition}
With the above notation,
\[
S(g,x)
=
2\sum_{\lambda\geq1}\lambda^{-2}
\sum_{\substack{r\in\mathscr R\\
(\lambda,r)=1}}
c_{g,\lambda}(r)
f\left(\frac{r}{x\lambda}\right).
\]
\end{lemma}
\begin{proof}
For integer $x$, the Fourier identity used in the proof of
\cite[Lemma 3.4]{key-14} gives
\[
|E_x(\theta)|^2
=
x^2\sum_{\ell\in\mathbb Z}
\left|V\left(x(\theta+\ell)\right)\right|^2.
\]
If $\theta=n/r$, then as $\ell$ varies over $\mathbb Z$,
$\lambda=n+\ell r$ runs exactly over the integers satisfying
$\lambda\equiv n\bmod r$. Moreover,
$(n,r)=1$ is equivalent to $(\lambda,r)=1$. Therefore
\[
\sum_{\sigma\in\mathcal R(r)}
|E_x(\sigma)|^2|\Lambda(g,\sigma,r)|^2
=
x^2
\sum_{\substack{\lambda\in\mathbb Z\setminus\{0\}\\
(\lambda,r)=1}}
\left|V\left(\frac{x\lambda}{r}\right)\right|^2
\left|\Lambda\left(g,\frac{\lambda}{r},r\right)\right|^2.
\]
The terms for $\lambda$ and $-\lambda$ are equal. Also,
\[
f\left(\frac{r}{x\lambda}\right)
=
\left(\frac{x\lambda}{r}\right)^2
\left|V\left(\frac{x\lambda}{r}\right)\right|^2.
\]
Substituting this into the definition of $S(g,x)$ and using the definition of $c_{g,\lambda}(r)$
gives
\[
S(g,x)
=
2\sum_{\lambda\geq1}\lambda^{-2}
\sum_{\substack{r\in\mathscr R\\
(\lambda,r)=1}}
c_{g,\lambda}(r)
f\left(\frac{r}{x\lambda}\right),
\]
as required.
\end{proof}

By \lemref{lem:one-pair-sum-decomposition}, the asymptotic behavior of $S(g,x)$ is determined by the coefficients $c_{g,\lambda}(r)$ and, in particular, by the growth of their partial sums $A_{g,\lambda}(Y)$. We therefore begin by describing the prime-by-prime structure of these coefficients.
Let $r=ab^2\in\mathscr R$. For a prime $p$, we have $\nu_p(r)\in\{0,1,2\},$
according as $p\nmid r$, $p\mid a$, or $p\mid b$. These three cases yield the following local factors
in the factorization of $r^2|\Lambda(g,\lambda/r,r)|^2$. For $e\in\{0,1,2\}$ and a coordinate $g_p \in V_{p^2}^{\sharp}$ define
\begin{equation}\label{eq:Xpdef}
X_{p,e}(g_p;\xi)
:=
\begin{cases}
\displaystyle
\left(1-\frac{\rho_{g_p}(p^2)}{p^2}\right)^2,
& e=0,\\[2mm]
\displaystyle
p^{-2}
\left|T_{g_p}\left(\frac{\xi}{p};p^2\right)\right|^2,
& e=1,\quad
\xi\in(\mathbb Z/p\mathbb Z)^\times,\\[2mm]
\displaystyle
\left|T_{g_p}\left(\frac{\xi}{p^2};p^2\right)\right|^2,
& e=2,\quad
\xi\in(\mathbb Z/p^2\mathbb Z)^\times.
\end{cases}
\end{equation}
If $p\mid r$, let $\xi_p(r,\lambda)$ denote the numerator of the
$p$-primary component of $\lambda/r$, so that this component is
$\xi_p(r,\lambda)/p^{\nu_p(r)}$ in $\mathbb Q/\mathbb Z$. Since
$(\lambda,r)=1$, this numerator is a unit modulo $p^{\nu_p(r)}$. The
coefficient $c_{g,\lambda}(r)$ therefore factors as
\begin{equation}
c_{g,\lambda}(r)
=
\prod_p
X_{p,\nu_p(r)}
\left(g_p;\xi_p(r,\lambda)\right),
\label{eq:c_g,lamda(r)}
\end{equation}
where no frequency argument is needed when $p\nmid r$. Note that this is a prime-by-prime factorization,
rather than multiplicativity in $r$, since
$\xi_p(r,\lambda)$ also depends on $r/p^{v_p(r)}$, the largest divisor
of $r$ coprime to $p$.

The next lemma gives the moment estimates for these three types
of local factors. We then use this to bound the moments of
$A_{g,\lambda}(Y)$ uniformly in $\lambda$, and later to show that the
effect of the prime coordinates $p>z$ becomes negligible as
$z\to\infty$. In \lemref{lem:local-coefficient-moment-bounds}, all expectations are with respect to the
uniform probability measure on $V_{p^2}^{\sharp}$, obtained by
restricting the uniform measure on $V_{p^2}$ to
$V_{p^2}^{\sharp}=V_{p^2}\setminus\{0\}$ and re-normalizing.

\begin{lemma}\label{lem:local-coefficient-moment-bounds}
Let $p$ be prime. For every fixed
$q\geq1$,
\[
\mathbb E X_{p,e}^q(g_p,\xi)
\ll_{d,q}
\begin{cases}
p^{-2}, & e=1,\\
1, & e=2.
\end{cases}
\]
uniformly for $\xi\in (\mathbb{Z}/p^e\mathbb{Z})^{\times}$.
Moreover, if $s\geq1$, $t\geq0$,
\[
\mathbb E
\prod_{j=1}^s X_{p,1}(g_p;\xi_j)
\prod_{\ell=1}^t X_{p,2}(g_p;\eta_\ell)
\ll_{d,s,t}p^{-2}
\]
holds uniformly for $\xi_j \in (\mathbb{Z}/p\mathbb{Z})^{\times}$ and $\eta_\ell \in (\mathbb{Z}/p^2\mathbb{Z})^{\times}$. Finally,
\[
\mathbb E X_{p,e}
=
\begin{cases}
1+O_d(p^{-2}),
& e=0,\\[1mm]
\displaystyle
\frac{2p^{-2}(1-p^{-1})}{1-p^{-2d-2}},
& e=1,\\[3mm]
\displaystyle
\frac{1-p^{-1}}{1-p^{-2d-2}},
& e=2.
\end{cases}
\]
\end{lemma}

\begin{proof}
We first prove the moment bounds. For $e=2$, the frequency is primitive
modulo $p^2$, so \lemref{lem:root-exponential-sum-moment-bounds}, with exponent $2q$, gives
\[
\mathbb E X_{p,2}^q
=
\mathbb E
\left|
T_{g_p}\left(\frac{\xi}{p^2};p^2\right)
\right|^{2q}
\ll_{d,q}1.
\]
For $e=1$, the second case of \lemref{lem:root-exponential-sum-moment-bounds} gives
\[
\mathbb E X_{p,1}^q
=
p^{-2q}
\mathbb E
\left|
T_{g_p}\left(\frac{\xi}{p};p^2\right)
\right|^{2q}
\ll_{d,q}
p^{-2q}\left(1+p^{2q-2}\right)
\ll_{d,q}p^{-2}.
\]
Recall also from the primitive-frequency part of
\lemref{lem:root-exponential-sum-moment-bounds} that
\[
\left|
T_{g_p}\left(\frac{\eta}{p^2};p^2\right)
\right|
\leq d.
\]
Thus the factors $X_{p,2}$ are pointwise bounded, and Hölder's
inequality applied to the $s$ factors of type $e=1$ gives
\[
\begin{aligned}
\mathbb E
\prod_{j=1}^s X_{p,1}(g_p;\xi_j)
\prod_{\ell=1}^t X_{p,2}(g_p;\eta_\ell)
&\ll_{d,t}
\prod_{j=1}^s
\left(
\mathbb E X_{p,1}(g_p;\xi_j)^s
\right)^{1/s}\\
&\ll_{d,s,t}p^{-2}.
\end{aligned}
\]
It remains to compute the first moments. We first average over the full
space $V_{p^2}$. Expanding the squared absolute values in the definitions
of $X_{p,1}$ and $X_{p,2}$ reduces the calculation to
counting pairs of roots. Writing
$g(t)=a_0+a_1t+\cdots+a_dt^d$,
for $u,v\bmod p^2$, counting the possible values of
$a_0$ and $a_1$ gives
\[
\mathbb P_{g\in V_{p^2}}
\bigl(
g(u)\equiv g(v)\equiv0\bmod p^2
\bigr)
=
\frac{\gcd(p^2,u-v)}{p^4}.
\]
Hence, for $\xi\in(\mathbb Z/p^2\mathbb Z)^\times$,
\begin{align*}
\mathbb E_{g\in V_{p^2}}
\left|
T_g\left(\frac{\xi}{p^2};p^2\right)
\right|^2
&=
p^{-2}
\sum_{h\bmod p^2}
\gcd(p^2,h)e\left(\frac{\xi h}{p^2}\right)\\
&=
p^{-2}
\sum_{d\mid p^2}\varphi(d)
\sum_{\substack{h\bmod p^2\\d\mid h}}
e\left(\frac{\xi h}{p^2}\right)\\
&=
\frac{\varphi(p^2)}{p^2}
=
1-\frac1p.
\end{align*}
The same calculation with the character $e(\xi h/p)$ gives, for
$\xi\in(\mathbb Z/p\mathbb Z)^\times$,
\[
\mathbb E_{g\in V_{p^2}}
\left|
T_g\left(\frac{\xi}{p};p^2\right)
\right|^2
=
2\left(1-\frac1p\right).
\]
The zero polynomial contributes zero to both nonzero-frequency root
sums. Since
\[
\frac{|V_{p^2}^{\sharp}|}{|V_{p^2}|}
=
1-p^{-2d-2},
\]
conditioning on $V_{p^2}^{\sharp}$ therefore gives
\[
\mathbb E X_{p,2}
=
\frac{1-p^{-1}}{1-p^{-2d-2}},
\quad
\mathbb E X_{p,1}
=
\frac{2p^{-2}(1-p^{-1})}{1-p^{-2d-2}}.
\]
Finally,
\[
X_{p,0}
=
\left(
1-\frac{\rho_{g_p}(p^2)}{p^2}
\right)^2
=
1-\frac{2\rho_{g_p}(p^2)}{p^2}
+\frac{\rho_{g_p}(p^2)^2}{p^4}.
\]
By \lemref{lem:root-count-moment-bounds},
\[
\mathbb E\rho_{g_p}(p^2)^j\ll_d1
\quad (j=1,2).
\]
Consequently, averaging over $V_{p^2}^{\sharp}$ gives
\[
\mathbb E X_{p,0}
=
1-\frac{2}{p^2}\mathbb E\rho_{g_p}(p^2)
+\frac{1}{p^4}\mathbb E\rho_{g_p}(p^2)^2
=1+O_d(p^{-2}).
\]
This proves the lemma.
\end{proof}

We shall also need centered versions of these estimates. Put $\eta_{p,e}:=\mathbb E X_{p,e}.$
For every fixed integer $q\geq2$, we have
\begin{equation}
\mathbb E
|X_{p,e}-\eta_{p,e}|^q
\ll_{d,q}
\begin{cases}
p^{-q-2}, & e=0,\\
p^{-2}, & e=1,\\
1, & e=2.
\end{cases}
\label{eq:CenteredExpe}
\end{equation}
The cases $e=1,2$ follow immediately from
\lemref{lem:local-coefficient-moment-bounds}. For $e=0$, \lemref{lem:root-count-moment-bounds} gives
\[
\mathbb E\rho_{g_p}(p^2)^j
\ll_{d,j}1+p^{j-2},
\]
and
\[
|X_{p,0}-1|
\ll
\frac{\rho_{g_p}(p^2)}{p^2}
+
\frac{\rho_{g_p}(p^2)^2}{p^4}.
\]
Consequently,
\[
\mathbb E|X_{p,0}-1|^q
\ll_{d,q}
p^{-2q}\mathbb E\rho_{g_p}(p^2)^q
+
p^{-4q}\mathbb E\rho_{g_p}(p^2)^{2q}
\ll_{d,q}p^{-q-2}.
\]
Since $\eta_{p,0}=1+O_d(p^{-2})$, the same estimate holds for
$X_{p,0}-\eta_{p,0}$.

The preceding local estimates now combine to give the key uniform moment bound for the coefficient partial sums.

\begin{lemma}\label{lem:partial-sum-moment-bound}
Fix an integer $q\geq2$. Uniformly for $Y\geq1$ and $\lambda\geq1$,
\[
\mathbb E A_{g,\lambda}(Y)^q
\ll_{d,q}Y^{q/2}.
\]
\end{lemma}

\begin{proof}
Expanding the $q$th power and using non-negativity in \eqref{eq:A_g(Y)} gives
\[
\mathbb E A_{g,\lambda}(Y)^q
\leq
\sum_{\substack{r_1,\ldots,r_q\leq Y \\
r_i=a_ib_i^2 \\ (r_i, \lambda)=1,(1\le i\le q)}}
\mathbb E\prod_{i=1}^q c_{g,\lambda}(r_i).
\]
For a prime $p$, write, $\mathbf e_p
:=
\bigl(\nu_p(r_1),\ldots,\nu_p(r_q)\bigr)
\in\{0,1,2\}^q$
and put
\[
W_p(\mathbf e,\boldsymbol\xi)
:=
\mathbb E_{g_p\in V_{p^2}^\sharp}
\prod_{i=1}^q
X_{p,e_i}(g_p;\xi_i).
\]
We normalize by
\[
\widetilde W_p(\mathbf e,\boldsymbol\xi)
:=
\frac{W_p(\mathbf e,\boldsymbol\xi)}
     {W_p(\mathbf0)},
\quad
W_p(\mathbf0)
=
\mathbb E X_{p,0}^q
=
1+O_{d,q}(p^{-2}).
\]
The product of the factors $W_p(\mathbf0)$ converges absolutely.

The estimates of \lemref{lem:local-coefficient-moment-bounds}, together
with $0\leq X_{p,0}\leq1$, imply
\begin{equation}
\widetilde W_p(\mathbf e,\boldsymbol\xi)
\ll_{d,q}
\begin{cases}
p^{-2},
& \text{if some }e_i=1,\\
1,
& \text{if at least two }e_i=2
  \text{ and no }e_i=1.
\end{cases}
\label{eq:W_pEstimates}
\end{equation}
The only remaining case is that exactly one $e_i$ equals $2$ and all
the others are $0$.  Here
\[
1-X_{p,0}^{q-1}
\ll_q
\frac{\rho_{g_p}(p^2)}{p^2},
\]
while the unique $X_{p,2}$-factor has mean
$1-p^{-1}+O_d(p^{-2})$ by \lemref{lem:local-coefficient-moment-bounds}.  Hence
\lemref{lem:local-coefficient-moment-bounds} and
\lemref{lem:root-count-moment-bounds} give
\begin{equation}
\widetilde W_p(\mathbf e,\boldsymbol\xi)
=
1-\frac1p+O_{d,q}(p^{-2}).
\label{eq:W_pEstimates2}
\end{equation}
Recall that $r_i=a_i b_i^2$, with $a_i,b_i$ squarefree and $(a_i,b_i)=1$. We now fix the tuple
$(a_1,\ldots,a_q)$ and sum first over the variables $b_i$. Denote $A:=a_1\cdots a_q.$
For $p\nmid A$, define the local Dirichlet series
\[
F_p(\mathbf s)
:=
\sum_{I\subseteq[q]}
\widetilde W_p(\mathbf e(I))
p^{-\sum_{i\in I}s_i},
\]
where $i\in I$ means $p\mid b_i$.  By
\eqref{eq:W_pEstimates}--\eqref{eq:W_pEstimates2},
\[
F_p(\mathbf s)
=
1+
\sum_{i=1}^q
\left(1-\frac1p+O_{d,q}(p^{-2})\right)p^{-s_i}
+
O_{d,q}\left(
\sum_{|I|\geq2}
p^{-\sum_{i\in I}s_i}
\right).
\]
Factor off the local zeta factors by writing
\[
H_p(\mathbf s)
:=
F_p(\mathbf s)
\prod_{i=1}^q(1-p^{-s_i}).
\]
At $s_1=\cdots=s_q=1$, the sum of the absolute values of the
nonconstant terms of $H_p$ is $O_{d,q}(p^{-2})$: the linear terms have
coefficients
\[
-\frac1p+O_{d,q}(p^{-2})
\]
and are multiplied by $p^{-1}$, while every term of degree at least two
already contributes $O_{d,q}(p^{-2})$.

If $p\mid A$, then \(p\mid a_i\) for at least one
\(i\), i.e. \(e_i=1\).  By the first case of
\eqref{eq:W_pEstimates}, the normalized local contribution at such a
prime is therefore $O_{d,q}(p^{-2}),$
uniformly in the remaining divisibility pattern of the \(b_j\)'s.

We now sum over the square variables $b_1,\ldots,b_q$.  
For fixed $\mathbf a=(a_1,\ldots,a_q)$, define
$\mathcal W_{\mathbf a}(\mathbf b)=0$ unless every
$b_i$ is squarefree and $(a_i,b_i)=1$. Otherwise, set
\[
\mathcal W_{\mathbf a}(\mathbf b)
:=
\prod_p
\sup_{\boldsymbol\xi}
\left|
\widetilde W_p
\bigl(
\nu_p(a_1b_1^2),\ldots,\nu_p(a_qb_q^2);
\boldsymbol\xi
\bigr)
\right|,
\]
where $\mathbf b=(b_1,\ldots,b_q)$.
Thus $\mathcal W_{\mathbf a}(\mathbf b)$ is a
prime-by-prime factorized majorant for the normalized
contribution of $(a_1b_1^2,\ldots,a_qb_q^2)$
to the expanded moment.
Its multivariable Dirichlet series admits the factorization
\[
\sum_{b_1,\ldots,b_q\geq1}
\frac{\mathcal W_{\mathbf a}(\mathbf b)}
     {b_1^{s_1}\cdots b_q^{s_q}}
=
\prod_{i=1}^q\zeta(s_i)
\sum_{d_1,\ldots,d_q\geq1}
\frac{h_{\mathbf a}(d_1,\ldots,d_q)}
     {d_1^{s_1}\cdots d_q^{s_q}},
\]
which defines the coefficients $h_{\mathbf a}$.
The local estimates above imply that, after the zeta factors are
removed, the local $\ell^1$-norm at $s_1=\cdots=s_q=1$ is
$1+O_{d,q}(p^{-2})$ for $p\nmid A$, while for $p\mid A$ it is
$O_{d,q}(p^{-2})$, since then some $e_i=1$.  Hence
\[
\sum_{d_1,\ldots,d_q\geq1}
\frac{|h_{\mathbf a}(d_1,\ldots,d_q)|}
     {d_1\cdots d_q}
\ll_{d,q}
C^{\omega(A)}
\prod_{p\mid A}p^{-2}.
\]
By convolution, for any $B_1,\ldots,B_q\geq1$,
\[
\sum_{\substack{b_i\leq B_i\\1\leq i\leq q}}
\mathcal W_{\mathbf a}(\mathbf b)
\ll_{d,q}
B_1\cdots B_q\,
C^{\omega(A)}
\prod_{p\mid A}p^{-2}.
\]
Since $a_i b_i^2\leq Y$, taking $B_i=\left(Y/a_i\right)^{1/2}$
shows that the contribution of the fixed tuple
$\mathbf a=(a_1,\ldots,a_q)$ is
\[
\ll_{d,q}
Y^{q/2}
\left(\prod_{i=1}^q a_i^{-1/2}\right)
C^{\omega(A)}
\prod_{p\mid A}p^{-2}.
\]
It remains to sum the above bounds over the squarefree variables $a_i$.  If a prime $p$ divides
exactly those $a_i$ with $i\in J$, where
$\emptyset\neq J\subseteq[q]$, its contribution is $\ll_{d,q}p^{-2-|J|/2}$.
Thus the local Euler factor is
\[
1+
O_{d,q}\left(
p^{-2}
\sum_{\emptyset\neq J\subseteq[q]}
p^{-|J|/2}
\right)
=
1+O_{d,q}(p^{-5/2}),
\]
so the Euler product converges absolutely.  Therefore, uniformly in $\lambda$
\[
\mathbb E A_{g,\lambda}(Y)^q
\ll_{d,q}Y^{q/2}.
\]
\end{proof}

We shall also need the corresponding coefficients after removing one local prime from \eqref{eq:c_g,lamda(r)}.  For a prime $p$, define
\[
c_{g,\lambda}^{[p]}(r)
:=
\prod_{\substack{
\ell\ \mathrm{prime}\\
\ell\neq p
}}
X_{\ell,\nu_\ell(r)}
\left(g_\ell;\xi_\ell(r,\lambda)\right),
\quad r=ab^2\in\mathscr R,
\]
and put
\[
A_{g,\lambda}^{[p]}(Y)
:=
\sum_{\substack{r\leq Y\\
r\in\mathscr R\\
(r,\lambda)=1}}
c_{g,\lambda}^{[p]}(r).
\]
We also define
\[
A_{g,\lambda}^{[p],0}(Y)
:=
\sum_{\substack{r\leq Y\\
r\in\mathscr R\\
(r,\lambda)=1\\
p\nmid r}}
c_{g,\lambda}^{[p]}(r)
\]
and for $e=1,2$, let
\[
A_{g,\lambda}^{[p],1}(Y)
:=
\sum_{\substack{r=ab^2\leq Y\\
(r,\lambda)=1\\
p\mid a}}
c_{g,\lambda}^{[p]}(r),
\]
and
\[
A_{g,\lambda}^{[p],2}(Y)
:=
\sum_{\substack{r=ab^2\leq Y\\
(r,\lambda)=1\\
p\mid b}}
c_{g,\lambda}^{[p]}(r).
\]
Thus, for $e=0,1,2$, $A_{g,\lambda}^{[p],e}$
is the part of the prime-omitted sum
$A_{g,\lambda}^{[p]}$ with $\nu_p(r)=e$.
Since all coefficients $c_{g,\lambda}^{[p]}(r)$
are nonnegative, $0\leq A_{g,\lambda}^{[p],0}(Y)
\leq A_{g,\lambda}^{[p]}(Y).$
Hence the $L^q$ bound for $A_{g,\lambda}^{[p],0}(Y)$
follows from
\lemref{lem:prime-omitted-partial-sum-moment-bound}.
The same argument also gives the following prime-omitted form, which we
record separately for later use.

\begin{lemma}\label{lem:prime-omitted-partial-sum-moment-bound}
Fix an integer $q\geq2$ and a prime $p$. Uniformly for $Y\geq1$ and
$\lambda\geq1$,
\[
\mathbb E\left(A_{g,\lambda}^{[p]}(Y)\right)^q
\ll_{d,q}Y^{q/2}.
\]
Moreover,
\[
\mathbb E\left(A_{g,\lambda}^{[p],1}(Y)\right)^q
\ll_{d,q}
p^{-q/2}Y^{q/2},
\]
and
\[
\mathbb E\left(A_{g,\lambda}^{[p],2}(Y)\right)^q
\ll_{d,q}
p^{-q}Y^{q/2}.
\]
\end{lemma}

\begin{proof}
Omitting the $p$-local factor leaves the proof of
\lemref{lem:partial-sum-moment-bound} unchanged, except that the Euler
factor at $p$ is removed. Hence
\[
\mathbb E\left(A_{g,\lambda}^{[p]}(Y)\right)^q
\ll_{d,q}Y^{q/2}.
\]
For $A_{g,\lambda}^{[p],1}(Y)$, every term has $p\mid a$. Thus, in the
expanded $q$th moment, $p\mid a_i$ for every $i$. Writing
$a_i=pa_i'$, we have
\[
B_i=\left(\frac{Y}{a_i}\right)^{1/2}
=
p^{-1/2}\left(\frac{Y}{a_i'}\right)^{1/2}.
\]
Since the $p$-local factor has been omitted, the remaining weighted sum
is bounded by the same multiplicative majorant as in the proof of
\lemref{lem:partial-sum-moment-bound}. Therefore the $q$ ranges contribute
a factor $p^{-q/2}$, and
\[
\mathbb E\left(A_{g,\lambda}^{[p],1}(Y)\right)^q
\ll_{d,q}
p^{-q/2}Y^{q/2}.
\]
Similarly, for $A_{g,\lambda}^{[p],2}(Y)$ every term has $p\mid b$.
Writing $b_i=pb_i'$ shortens each of the $q$ ranges by a factor $p$,
while the remaining weighted sum is controlled by the same
prime-omitted majorant. Hence
\[
\mathbb E\left(A_{g,\lambda}^{[p],2}(Y)\right)^q
\ll_{d,q}
p^{-q}Y^{q/2}.
\]
\end{proof}

We are now ready to begin the identification of the main term of $Y^{-1/2} A_{g,\lambda}(Y)$.
We first do so after retaining only finitely many prime coordinates. Let $\mathcal F_z$ be as in \eqref{eq:sigmaAlgbra} and put
\[
A_{g,\lambda}^{(z)}(Y)
:=
\mathbb E\!\left[
A_{g,\lambda}(Y)\mid\mathcal F_z
\right].
\]
As $z$ increases, $A_{g,\lambda}^{(z)}(Y)$ forms a \textbf{martingale} with
respect to the filtration $(\mathcal F_z)_{z\ge2}$. The next lemma identifies the limit of
$Y^{-1/2}A_{g,\lambda}^{(z)}(Y)$ for a fixed $z$.

\begin{lemma}\label{lem:finite-prime-partial-sum-limit}
Fix $z\geq2$ and $\lambda\geq1$. There exists a non-negative
$\mathcal F_z$-measurable random variable $D_g^{(z)}(\lambda)$ such that,
for every $q\geq1$,
\[
Y^{-1/2}A_{g,\lambda}^{(z)}(Y)
\longrightarrow
D_g^{(z)}(\lambda)
\]
in $L^q$ as $Y\to\infty$.
\end{lemma}

\begin{proof}
Recall the notation $\eta_{p,e}=\mathbb E X_{p,e}$. By independence of the
prime coordinates, conditioning on $\mathcal F_z$ replaces each local
factor at $p>z$ by its expectation. Set
\[
\kappa_z:=\prod_{p>z}\eta_{p,0},
\quad
\alpha_p:=\frac{\eta_{p,1}}{\eta_{p,0}},
\quad
\beta_p:=\frac{\eta_{p,2}}{\eta_{p,0}},
\]
and, for squarefree $a,b$,
\[
\alpha(a):=\prod_{\substack{p\mid a\\p>z}}\alpha_p,
\quad
\beta(b):=\prod_{\substack{p\mid b\\p>z}}\beta_p.
\]
Then
\[
A_{g,\lambda}^{(z)}(Y)
=
\kappa_z
\sum_{\substack{ab^2\leq Y\\
a,b\ \mathrm{squarefree}\\
(a,b)=(ab,\lambda)=1}}
\prod_{p\leq z}
X_{p,\nu_p(ab^2)}
\!\left(g_p;\xi_p(ab^2,\lambda)\right)
\alpha(a)\beta(b).
\]
By \lemref{lem:local-coefficient-moment-bounds},
$\kappa_z$ converges absolutely and
\[
\alpha_p\ll_d p^{-2},
\quad
\beta_p=1-\frac1p+O_d(p^{-2}).
\]
For fixed $z$, the product over $p\leq z$ depends on $r=ab^2$ through
only finitely many congruence classes. Indeed, if
$r=p^em$, with $e=\nu_p(r)\in\{1,2\}$ and $(m,p)=1$, then the
$p$-primary component of $\lambda/r$ is
\[
\frac{\lambda\overline m}{p^e}\in\mathbb Q/\mathbb Z,
\]
where $\overline m$ is the inverse of $m\bmod p^e$. Hence the
$p$-local factor is determined by $r\bmod p^{2e}$. Thus all factors
with $p\leq z$ are determined by
\[
r\bmod M_z,
\quad
M_z:=\prod_{p\leq z}p^4.
\]
Consequently, grouping the terms according to the residue
class $c\equiv ab^2\bmod M_z$, we can express
$A_{g,\lambda}^{(z)}(Y)$ as a finite sum of the form
\[
\sum_{c\bmod M_z}
\gamma_{g,z,\lambda}(c)
\sum_{\substack{
ab^2\leq Y\\
a,b\ \mathrm{squarefree}\\
(a,b)=(ab,\lambda)=1\\
ab^2\equiv c\bmod M_z
}}
\alpha(a)\beta(b),
\]
where $\gamma_{g,z,\lambda}(c)$ is the non-negative
$\mathcal F_z$-measurable coefficient determined by
the corresponding product of local factors.
Fix such a residue class $c$ and a squarefree $a$. Then
$ab^2\equiv c\bmod{M_z}$ restricts $b$ to finitely many residue classes
modulo $M_z$. Define
\[
w_a(n):=
\mu^2(n)\mathds{1}_{(n,a\lambda)=1}
\prod_{\substack{p\mid n\\p>z}}\beta_p.
\]
The function $w_a$ is multiplicative. Writing $w_a=1*h_a$, for $p>z$ with $p\nmid a\lambda$
we have
\[
h_a(p)=\beta_p-1,\quad
h_a(p^2)=-\beta_p,\quad
h_a(p^j)=0\quad(j\ge3).
\]
For $p\mid a\lambda$, we instead have
$h_a(p)=-1$ and $h_a(p^2)=0$, while for
$p\le z$ with $p\nmid a\lambda$, we have
$h_a(p)=0$ and $h_a(p^2)=-1$. Hence
\begin{equation}
\sum_{n\geq1}\frac{|h_a(n)|}{n}
\ll_{d,z,\lambda}C^{\omega(a)}.   
\label{eq:abs-convergence-h_a}
\end{equation}
For a fixed residue class \(u\bmod M_z\), convolution gives
\[
\sum_{\substack{b\leq B\\ b\equiv u\!\!\!\bmod{M_z}}}w_a(b)
=
\sum_{d\leq B}h_a(d)
\#\left\{m\leq\frac Bd:
dm\equiv u\bmod{M_z}\right\}.
\]
Writing \(g=(d,M_z)\), the inner count is \(0\) unless \(g\mid u\), and otherwise equals
\[
\frac{g}{M_z}\frac Bd+O_z(1).
\]
Since \eqref{eq:abs-convergence-h_a} also implies
\[
\sum_{d\leq B}|h_a(d)|=o_a(B),
\]
we obtain
\[
\sum_{\substack{
1\le b\le B\\
ab^2\equiv c\bmod{M_z}
}}
w_a(b)
=
\delta_c(a,\lambda)B+o_a(B)
\]
for some \(\delta_c(a,\lambda)\geq0\). To justify summing the asymptotic over $a$, we also need
a uniform bound. Since the inner count is $\ll_z B/d$,
we have
\[
\sum_{\substack{
1\le b\le B\\
ab^2\equiv c\bmod{M_z}
}}
w_a(b)
\ll_{d,z,\lambda}C^{\omega(a)}B.
\]
uniformly for \(B\geq1\).
Taking $B=(Y/a)^{1/2}$ gives
\[
Y^{-1/2}\alpha(a)
\sum_{\substack{b\leq(Y/a)^{1/2}\\
ab^2\equiv c\!\!\!\bmod{M_z}}}
w_a(b)
\longrightarrow
\frac{\alpha(a)}{\sqrt a}\,\delta_c(a,\lambda),
\]
with uniform majorant
\[
\ll_{d,z,\lambda}
\frac{|\alpha(a)|C^{\omega(a)}}{\sqrt a}.
\]
This is summable over squarefree $a$, since
\[
\sum_{\substack{a\geq1\\a\ \mathrm{squarefree}}}
\frac{|\alpha(a)|C^{\omega(a)}}{\sqrt a}
\leq
\prod_{p\le z}\left(1+\frac{C}{p^{1/2}}\right)\prod_{p>z}
\left(1+\frac{C|\alpha_p|}{p^{1/2}}\right)
<\infty.
\]
Dominated convergence therefore applies to the $a$-sum. Summing over
the finitely many classes $c\bmod M_z$ yields a non-negative,
$\mathcal F_z$-measurable limit
\[
Y^{-1/2}A_{g,\lambda}^{(z)}(Y)
\longrightarrow
D_g^{(z)}(\lambda)
\]
pointwise in $g$. Finally, for fixed $z$ both sides depend on $g$ only through the finite
coordinate space
\[
\prod_{p\leq z}V_{p^2}^{\sharp}.
\]
Hence pointwise convergence is uniform, and therefore also holds in
$L^q$ for every $q\geq1$.
\end{proof}

We shall state the following standard form of Doob's martingale convergence theorem.
\begin{proposition}[Doob's martingale convergence theorem]
\label{prop:doob-martingale-convergence}
Let $(\mathcal G_n)_{n\geq1}$ be an increasing sequence of
$\sigma$-algebras, and put
\[
\mathcal G_\infty
:=
\sigma\bigg(\bigcup_{n\geq1}\mathcal G_n\bigg).
\]
For any random variable $X\in L^p$ with $p>1$, we have
\[
\mathbb E[X\mid\mathcal G_n]
\longrightarrow
\mathbb E[X\mid\mathcal G_\infty]
\]
almost surely and in $L^p$. In particular, if $X$ is
$\mathcal G_\infty$-measurable, then
\[
\mathbb E[X\mid\mathcal G_n]
\longrightarrow X
\]
almost surely and in $L^p$.
\end{proposition}

\begin{proof}
The almost-sure convergence is Doob's martingale convergence
theorem \cite[Theorem 35.6]{key-18}; the $L^p$ convergence is its
standard $L^p$ consequence, using conditional Jensen's inequality.
\end{proof}

We next show that the contribution of the prime coordinates $p>z$ is
uniformly negligible after normalization.

\begin{lemma}\label{lem:prime-tail-uniform-approximation}
For every fixed integer $q\geq1$, and uniformly in $\lambda$
\[
\lim_{z\to\infty}
\sup_{Y\geq1}
\left\|
Y^{-1/2}
\left(
A_{g,\lambda}(Y)-A_{g,\lambda}^{(z)}(Y)
\right)
\right\|_{L^q}
=
0
\]
\end{lemma}

\begin{proof}
Put
\[
U_{Y,\lambda}:=Y^{-1/2}A_{g,\lambda}(Y),
\quad
U_{Y,\lambda}^{(z)}
:=\mathbb E[U_{Y,\lambda}\mid\mathcal F_z].
\]
For each prime $p$, let
\[
\mathcal F_{p^-}:=\sigma(g_\ell:\ell<p),
\quad
\mathcal F_p:=\sigma(g_\ell:\ell\le p),
\]
and define the martingale difference
\[
\Delta_p(Y):=
\mathbb E[U_{Y,\lambda}\mid\mathcal F_p]
-
\mathbb E[U_{Y,\lambda}\mid\mathcal F_{p^-}].
\]
By \lemref{lem:partial-sum-moment-bound},
$U_{Y,\lambda}\in L^2$. Since the prime-coordinate
$\sigma$-algebras generate the full product
$\sigma$-algebra,
\propref{prop:doob-martingale-convergence}
gives
\[
U_{Y,\lambda}-U_{Y,\lambda}^{(z)}
=\sum_{p>z}\Delta_p(Y)
\]
with convergence in $L^2$.
Write
\[
\Delta_p
=\Delta_{p,0}+\Delta_{p,1}+\Delta_{p,2},
\]
where $\Delta_{p,e}$ contains the terms with
$\nu_p(r)=e$ and the centered local factor
$X_{p,e}-\eta_{p,e}$.
By \eqref{eq:CenteredExpe} with exponent $2$,
\[
\mathbb E|X_{p,e}-\eta_{p,e}|^2
\ll_d
\begin{cases}
p^{-4},&e=0,\\
p^{-2},&e=1,\\
1,&e=2.
\end{cases}
\]
Using independence of $g_p$, the triangle
inequality in $L^2(g_p)$, and conditional
Jensen's inequality, we obtain
\[
\|\Delta_{p,e}(Y)\|_2^2
\ll_d
\frac{v_{p,e}}{Y}
\mathbb E\left(
A_{g,\lambda}^{[p],e}(Y)
\right)^2,
\]
where $v_{p,0}=p^{-4},
v_{p,1}=p^{-2},
$ and $v_{p,2}=1.$
Thus
\lemref{lem:prime-omitted-partial-sum-moment-bound}
gives
\[
\begin{aligned}
\|\Delta_{p,0}(Y)\|_2^2&\ll_d p^{-4},\\
\|\Delta_{p,1}(Y)\|_2^2&\ll_d p^{-3},\\
\|\Delta_{p,2}(Y)\|_2^2&\ll_d p^{-2}.
\end{aligned}
\]
Consequently, $\|\Delta_p(Y)\|_2^2\ll_d p^{-2}.$
Since martingale differences are orthogonal
in $L^2$, we conclude that uniformly in $Y, \lambda$
\[
\begin{aligned}
\|U_{Y,\lambda}-U_{Y,\lambda}^{(z)}\|_2^2
&=\sum_{p>z}\|\Delta_p(Y)\|_2^2
\ll_d\sum_{p>z}p^{-2}
\ll z^{-1}.
\end{aligned}
\]
It remains to pass to higher moments.
For every fixed integer $r\ge2$,
\lemref{lem:partial-sum-moment-bound}
and the contraction property of
conditional expectation imply
\[
\|U_{Y,\lambda}-U_{Y,\lambda}^{(z)}\|_r
\le 2\|U_{Y,\lambda}\|_r
\ll_{d,r}1.
\]
For $q>2$, the $L^p$ interpolation inequality and \lemref{lem:partial-sum-moment-bound} give
\[
\begin{aligned}
\|U_{Y,\lambda}-U_{Y,\lambda}^{(z)}\|_q
&\le
\|U_{Y,\lambda}-U_{Y,\lambda}^{(z)}\|_2
^{1/(q-1)}\cdot
\|U_{Y,\lambda}-U_{Y,\lambda}^{(z)}\|_{2q}
^{(q-2)/(q-1)}\ll_{d,q}z^{-1/(2(q-1))}.
\end{aligned}
\]
For $q=1,2$, the $L^2$ estimate already
suffices. The bounds are uniform in
$Y$ and $\lambda$, proving the lemma.
\end{proof}

Lemmas \ref{lem:finite-prime-partial-sum-limit} and \ref{lem:prime-tail-uniform-approximation} now combine to
give the normalized limit of the full coefficient partial sums. This
is the intermediate random quantity from which $C(g)$ will be
constructed.

\begin{proposition}\label{prop:partial-sum-limits}
For every $\lambda\geq1$, there exists a non-negative random variable
$D_g(\lambda)$ such that, for every fixed integer $k\geq1$,
\[
Y^{-1/2}A_{g,\lambda}(Y)
\longrightarrow
D_g(\lambda)
\]
in $L^k$. Moreover,
\[
\sup_{\lambda\geq1}\sup_{Y\geq1}
\left\|
Y^{-1/2}A_{g,\lambda}(Y)
\right\|_{L^k}
\ll_{d,k}1,
\]
and consequently
\[
\sup_{\lambda\geq1}
\|D_g(\lambda)\|_{L^k}
\ll_{d,k}1.
\]
\end{proposition}

\begin{proof}
Fix $k\geq1$. By \lemref{lem:finite-prime-partial-sum-limit},
\[
Y^{-1/2}A_{g,\lambda}^{(z)}(Y)
\longrightarrow
D_g^{(z)}(\lambda)
\]
in $L^k$. If $z_1>z_2$, then
\begin{align*}
\left\|
D^{(z_1)}_g(\lambda)-D^{(z_2)}_g(\lambda)
\right\|_{L^k}
\leq{}&
\limsup_{Y\to\infty}
\left\|
Y^{-1/2}
\left(A_{g,\lambda}(Y)-A_{g,\lambda}^{(z_1)}(Y)\right)
\right\|_{L^k}\\
&+
\limsup_{Y\to\infty}
\left\|
Y^{-1/2}
\left(A_{g,\lambda}(Y)-A_{g,\lambda}^{(z_2)}(Y)\right)
\right\|_{L^k}
\end{align*}
By \lemref{lem:prime-tail-uniform-approximation}, the right-hand
side tends to zero as $z_1,z_2\to\infty$. Thus
$D_g^{(z)}(\lambda)$ is Cauchy in $L^k$; denote its limit by
$D_g(\lambda)$.
Now
\begin{align*}
\left\|
Y^{-1/2}A_{g,\lambda}(Y)-D_g(\lambda)
\right\|_{L^k}
\leq{}&
\left\|
Y^{-1/2}(A_{g,\lambda}(Y)-A_{g,\lambda}^{(z)}(Y))
\right\|_{L^k}\\
&+
\left\|
Y^{-1/2}A_{g,\lambda}^{(z)}(Y)-D_g^{(z)}(\lambda)
\right\|_{L^k}
+
\left\|
D_g^{(z)}(\lambda)-D_g(\lambda)
\right\|_{L^k}.
\end{align*}
First letting $Y\to\infty$ and then $z\to\infty$, using
\lemref{lem:finite-prime-partial-sum-limit} and \lemref{lem:prime-tail-uniform-approximation}, proves
\[
Y^{-1/2}A_{g,\lambda}(Y)
\longrightarrow
D_g(\lambda)
\]
in $L^k$. The limit is non-negative because the approximating
variables are non-negative. The limits obtained for different values
of $k$ agree by uniqueness of limits in probability, so
$D_g(\lambda)$ is independent of the chosen moment exponent.

For $k\geq2$, the uniform norm bound follows directly from
\lemref{lem:partial-sum-moment-bound}; for $k=1$, use its $q=2$ case and
$\|\cdot\|_{L^1}\leq\|\cdot\|_{L^2}$. Passing to the limit gives the
same uniform bound for $D_g(\lambda)$.
\end{proof}

We have now determined the asymptotic behavior of the coefficient
partial sums $A_{g,\lambda}(Y)$. The following lemma transfers this
asymptotic to the specific weight $f$ appearing in
\lemref{lem:one-pair-sum-decomposition}, following the
regular-variation argument of \cite[Proposition~3.5]{key-14}.
Recall that $f(t)=t^{-2}|V(1/t)|^2$, and put
\[
\mathcal T_{\lambda,T}(f)
:=
T^{-1/2}
\sum_{\substack{r\in\mathscr R\\(r,\lambda)=1}}
c_{g,\lambda}(r)f(r/T).
\]

\begin{lemma}\label{lem:weighted-coefficient-sum-limit}
For every fixed $k\ge1$ and $\lambda\ge1$,
\begin{equation}
\mathcal T_{\lambda,T}(f)
\longrightarrow
\frac{D_g(\lambda)}{2\pi},
\label{eq:weighted-coefficient-limit}
\end{equation}
in $L^k$, as $T\to\infty$. Moreover,
\[
\sup_{\lambda\ge1}\sup_{T\ge1}
\|\mathcal T_{\lambda,T}(f)\|_{L^k}
\ll_{d,k}1.
\]
\end{lemma}

\begin{proof}
For any function $h$ for which the following series
converges absolutely, set
\[
\mathcal T_{\lambda,T}(h)
:=
T^{-1/2}
\sum_{\substack{r\in\mathscr R\\(r,\lambda)=1}}
c_{g,\lambda}(r)h(r/T).
\]
By \propref{prop:partial-sum-limits},
$T^{-1/2}A_{g,\lambda}(T)\to D_g(\lambda)$ in $L^k$,
and the corresponding $L^k$ norms are uniformly bounded for
$T\ge1$ and $\lambda\ge1$. For $0\le u<v<\infty$,
\[
\mathcal T_{\lambda,T}(\mathds1_{(u,v]})
=
T^{-1/2}\bigl(A_{g,\lambda}(vT)-A_{g,\lambda}(uT)\bigr)
\longrightarrow
D_g(\lambda)(\sqrt v-\sqrt u)
\]
in $L^k$, and hence the same holds for compactly supported step
functions.

Since $|V(y)|\ll\min(1,|y|^{-1})$, the function $f$ is
continuous and satisfies $0\le f(t)\ll\min(1,t^{-2})$ for $t>0$.
Fix $0<\delta<1<M$ and $\varepsilon>0$, and choose a step
function $s$ such that
\[
\sup_{\delta<t\le M}|f(t)-s(t)|\le\varepsilon.
\]
By non-negativity of the coefficients and
\lemref{lem:partial-sum-moment-bound}, we obtain
\[
\begin{aligned}
\left\|
\mathcal T_{\lambda,T}
\bigl((f-s)\mathds1_{(\delta,M]}\bigr)
\right\|_{L^k}
&\le
\varepsilon T^{-1/2}
\|A_{g,\lambda}(MT)\|_{L^k}\\
&\ll_{d,k}\varepsilon M^{1/2},
\end{aligned}
\]
uniformly in $T\ge1$ and $\lambda\ge1$. The corresponding
Stieltjes integrals satisfy the same estimate, so step-function
approximation yields
\[
\mathcal T_{\lambda,T}\bigl(f\mathds1_{(\delta,M]}\bigr)
\longrightarrow
D_g(\lambda)\int_\delta^M f(t)\,d(t^{1/2})
\]
in $L^k$. The two tails satisfy
\[
\left\|\mathcal T_{\lambda,T}
\bigl(f\mathds1_{(0,\delta]}\bigr)\right\|_{L^k}
\ll_{d,k}\delta^{1/2}
\]
and, by a dyadic decomposition,
\[
\begin{aligned}
\left\|\mathcal T_{\lambda,T}
\bigl(f\mathds1_{(M,\infty)}\bigr)\right\|_{L^k}
&\ll_{d,k}
\sum_{j\ge0}(2^jM)^{-2}T^{-1/2}
\left\|A_{g,\lambda}(2^{j+1}MT)\right\|_{L^k}\\
&\ll_{d,k}M^{-3/2}.
\end{aligned}
\]
The corresponding Stieltjes integrals have the same tail bounds.
Letting $\delta\to0$ and $M\to\infty$ gives
\[
\mathcal T_{\lambda,T}(f)
\longrightarrow
D_g(\lambda)\int_0^\infty f(t)\,d(t^{1/2})
\quad\text{in }L^k.
\]
The same estimates, splitting at $t=1$, give the asserted uniform
$L^k$ bound. Finally,
\[
\begin{aligned}
2\int_0^\infty f(t)\,d(t^{1/2})
&=\int_0^\infty t^{-5/2}\left|V\left(\frac1t\right)\right|^2\,dt=\int_0^\infty y^{1/2}|V(y)|^2\,dy
=\frac1\pi,
\end{aligned}
\]
which identifies the limiting constant.
\end{proof}

By \lemref{lem:one-pair-sum-decomposition},
\[
x^{-1/2}S(g,x)
=
2\sum_{\lambda\ge1}
\lambda^{-3/2}\mathcal T_{\lambda,x\lambda}(f).
\]
For each fixed $\lambda$, \lemref{lem:weighted-coefficient-sum-limit}
identifies the $L^k$ limit of the corresponding summand. Its
uniform norm bound and $\sum_{\lambda\ge1}\lambda^{-3/2}<\infty$
allow us to sum the limits in $L^k$. Thus \eqref{eq:S-to-C}
holds with
\begin{equation}
C(g)
:=
\frac1\pi\sum_{\lambda\ge1}
\frac{D_g(\lambda)}{\lambda^{3/2}}.
\label{eq:C-from-D}
\end{equation}
The series converges in every fixed $L^k$ by
\propref{prop:partial-sum-limits}; in particular,
$C(g)\ge0$ and $C\in L^k$ for all $k\ge1$.
Since $V_\square^\sharp$ is a Borel subset of
$V_\square$, we henceforth fix a nonnegative
Borel representative of $C$, extended by zero
as in \eqref{eq:Cextend}.

It remains to translate this $L^k$-convergence into the moment
asymptotic required in \secref{sec:exact-frequency-pairings}. Put
\[
X_x(g):=x^{-1/2}S(g,x).
\]
Then
\[
\|X_x-C\|_{L^k}\longrightarrow0,
\quad
\sup_x\|X_x\|_{L^k}<\infty.
\]
Using
\[
|u^k-v^k|
\leq
k|u-v|
\left(u^{k-1}+v^{k-1}\right)
\quad (u,v\geq0),
\]
Hölder's inequality gives
\[
\|X_x^k-C^k\|_{L^1}
\leq
k\|X_x-C\|_{L^k}
\left(
\|X_x\|_{L^k}^{k-1}
+
\|C\|_{L^k}^{k-1}
\right)
\longrightarrow0.
\]
Therefore
\[
\int_{V_\square^\sharp}S(g,x)^k\,dg
=
x^{k/2}
\int_{V_\square^\sharp}C(g)^k\,dg
+
o\left(x^{k/2}\right).
\]
We have thus established \eqref{eq:S-to-C}, and hence
\eqref{eq:S-moment-asymptotic}. The convergence in
\eqref{eq:S-to-C}, together with
\propref{prop:repeated-denominators-negligible}, proved in the next
subsection, provides the two remaining ingredients needed to complete
the proof of \propref{prop:fixed-moment-asymptotics}.
\subsection{Completion of the exactly paired contribution}\label{sec:completion-paired-contribution}

By \propref{prop:local-model-exact-pairing-reduction},
only exactly paired frequency tuples remain.
As explained in \secref{sec:exact-frequency-pairings},
it remains to prove
\propref{prop:repeated-denominators-negligible},
which allows us to remove the distinctness condition
on the pair-denominators in
\eqref{eq:fixed-pairing-distinct-denominators}.
We begin by establishing norm bounds for the
one-pair contributions $A_r(g,x)$.

\begin{lemma}\label{lem:diagonal-block-norm-bounds}
Let $r=ab^{2}\in\mathscr{R}$. For every fixed $q\geq1$ and every $\varepsilon>0$, uniformly in $x\geq1$ and $\sigma\in\mathcal{R}\left(r\right)$,

\[
\left\Vert \Lambda\left(g,\sigma,r\right)\right\Vert _{L^{2q}\left(dg\right)}\ll_{d,q,\varepsilon}\frac{r^{\varepsilon}}{ra^{\frac{1}{q}}},
\]and\[
\left\Vert A_{r}\left(g,x\right)\right\Vert _{L^{q}\left(dg\right)}\ll_{d,q,\varepsilon}\frac{\min\left(r,x\right)}{r^{1-\varepsilon}a^{\frac{2}{q}}}.
\]
\end{lemma}
\begin{proof}
At a prime dividing $a$, the local frequency has denominator $p$, by \lemref{lem:root-exponential-sum-moment-bounds}\[
\left\Vert p^{-2}T_{g}\left(\frac{\xi}{p};p^{2}\right)\right\Vert _{2q}=p^{-2}\left(1+p^{2q-2}\right)^{\frac{1}{2q}}\ll_{d,q}p^{-1-\frac{1}{q}}.
\]At a prime dividing $b$, \lemref{lem:root-exponential-sum-moment-bounds} gives\[
\left\Vert p^{-2}T_{g}\left(\frac{\xi}{p^{2}};p^{2}\right)\right\Vert _{2q}\ll_{d,q}p^{-2}.
\]The local factors for $p\nmid r$ have modulus at most one. Multiplying the local estimates and absorbing $C^{\omega\left(r\right)}$ into $r^{\varepsilon}$ gives\[
\left\Vert \Lambda\left(g,\sigma,r\right)\right\Vert _{L^{2q}}\ll_{d,q,\varepsilon}r^{\varepsilon}a^{-1-\frac{1}{q}}b^{-2}=\frac{r^{\varepsilon}}{ra^{\frac{1}{q}}}.
\]Finally, Minkowski's inequality and \cite[Lemma 2.2]{key-14} give \[
\sum_{\sigma\in\mathcal{R}\left(r\right)}\left\lvert E_{x}\left(\sigma\right)\right\rvert ^{2}\ll r\min\left(r,x\right),
\]hence, on applying the first estimate with $\varepsilon/2$ in place of $\varepsilon$, we get \[
\begin{aligned} & \left\Vert A_{r}\left(g,x\right)\right\Vert _{L^{q}}\leq\sum_{\sigma\in\mathcal{R}\left(r\right)}\left\lvert E_{x}\left(-\sigma\right)\right\rvert ^{2}\left\Vert \Lambda\left(g,\sigma,r\right)\right\Vert _{L^{2q}}^{2}\\
 & \ll_{d,q,\varepsilon}r\min\left(r,x\right)\frac{r^{\varepsilon}}{r^{2}a^{\frac{2}{q}}}=\frac{\min\left(r,x\right)}{r^{1-\varepsilon}a^{\frac{2}{q}}}.
\end{aligned}
\]
\end{proof}

\begin{proof}[Proof of
\propref{prop:repeated-denominators-negligible}]
Since $A_r(g,x)\ge0$, we have
\[
\begin{aligned}
0&\le S(g,x)^k-
\sum_{\substack{
r_1,\ldots,r_k\in\mathscr R\\
r_1,\ldots,r_k>1\ \mathrm{distinct}
}}
\prod_{i=1}^k A_{r_i}(g,x)\le \binom{k}{2}R_2(g,x)S(g,x)^{k-2},
\end{aligned}
\]
where
\[
R_2(g,x):=\sum_{\substack{r\in\mathscr R\\r>1}}
A_r(g,x)^2.
\]
By \lemref{lem:diagonal-block-norm-bounds}
with $q=2$, writing $r=ab^2$, we have
\[
\|A_r(g,x)\|_2^2
\ll_{d,\varepsilon}
a^{-2}r^\varepsilon
\min\left(1,\frac{x}{r}\right)^2.
\]
For every $T\ge1$,
\[
\begin{aligned}
\sum_{\substack{T<r\le2T\\r=ab^2\in\mathscr R}}
a^{-2}
&\ll \sqrt T\sum_{a\ge1}a^{-5/2}
\ll \sqrt T.
\end{aligned}
\]
Consequently,
\[
\begin{aligned}
\int_{V_\square^\sharp}R_2(g,x)\,dg
&=\sum_{r>1}\|A_r(g,x)\|_2^2\ll_{d,\varepsilon}
\sum_{T\le x}T^{1/2+\varepsilon}
+x^2\sum_{T>x}T^{-3/2+\varepsilon}\ll_{d,\varepsilon}x^{1/2+\varepsilon}.
\end{aligned}
\]
On the other hand, since $A_r(g,x)\ge0$,
we have $R_2^{1/2}\le S$ pointwise.
Thus, by Cauchy--Schwarz and the uniform
moment bounds for $S$ established in
\eqref{eq:S-to-C},
\[
\begin{aligned}
\int_{V_\square^\sharp}R_2S^{k-2}\,dg
&\le \int_{V_\square^\sharp}
R_2^{1/2}S^{k-1}\,dg\le
\left(\int_{V_\square^\sharp}R_2\,dg\right)^{1/2}
\left(\int_{V_\square^\sharp}S^{2k-2}\,dg\right)^{1/2}\ll_{d,k,\varepsilon}
x^{k/2-1/4+\varepsilon},
\end{aligned}
\]
which is $o(x^{k/2})$ upon choosing $\varepsilon<1/4$.
This proves the proposition.
\end{proof}

Together with \eqref{eq:S-to-C} and
\propref{prop:local-model-exact-pairing-reduction},
\propref{prop:repeated-denominators-negligible}
completes the proof of
\propref{prop:fixed-moment-asymptotics},
as explained in \secref{sec:exact-frequency-pairings}.

\subsection{Mean and moments of the limiting variance}
\label{sec:mean-and-moments}

The case $m=2$ of the fixed-moment theorem gives\[
\mathrm{MP}_{2}\left(x,H\right)\sim x^{\frac{1}{2}}\int_{V_{\square}}C\left(g\right)dg.
\] For $d\ge3$, the identity
\[
\int_{V_\square}C(g)\,dg=\gamma_0
\]
follows by comparing the case $m=2$ of
\propref{prop:fixed-moment-asymptotics} with Sofos's
second-moment asymptotic \cite[Theorem 1.1]{key-1-1}.
For $d=1,2$, the same identity follows from a direct
evaluation of the local second moment.\footnote{
An alternative elementary proof of the variance asymptotic,
based on the method developed in this paper, is given in
\href{https://omersimhi.github.io/pdfs/Variance_of_squarfree_values_of_polynomials.pdf}
{this note}. In particular, it establishes
\[
\frac{1}{\zeta(2d+2)}
\int_{V_\square^\sharp}E_g(x)^2\,dg
=\gamma_0x^{1/2}+O_d(x^{1/3})
\]
for every $d\ge1$. Comparing this with the local-model
second-moment asymptotic from
\propref{prop:fixed-moment-asymptotics}
gives the claimed identity.
This approach also recovers Sofos's variance asymptotic
for $d\ge3$ and extends it to $d=1,2$.
Moreover, it yields the sharper error term $O(x^{1/3})$
under a slightly stronger restriction on $H$.
}
We can see directly that the limiting Gaussian
mixture is non-Gaussian. Indeed, by
\eqref{eq:Cextend}, $C(g)=0$ outside
$V_\square^\sharp$, while
\[
\mathbb P(G\in V_\square^\sharp)
=\prod_p(1-p^{-2d-2})
=\frac1{\zeta(2d+2)}.
\]
Since $Z$ is an independent standard Gaussian,
\begin{equation}\label{eq:mixture-atom-zero}
\mathbb P(\sqrt{C(G)}Z=0)
=\mathbb P(C(G)=0)
\ge 1-\frac1{\zeta(2d+2)}>0.
\end{equation}
On the other hand,
\[
\mathbb E[(\sqrt{C(G)}Z)^2]
=\mathbb E[C(G)]=\gamma_0>0.
\]
Thus the limiting distribution has a positive
atom at zero but is not degenerate, and hence
is not Gaussian. In particular, $C(G)$ is not
almost surely constant. In fact, by
Cauchy--Schwarz,
\[
\gamma_0^2
\le \frac1{\zeta(2d+2)}
\int_{V_\square}C(g)^2\,dg,
\]
so $\operatorname{Var}_{dg}(C)
\ge(\zeta(2d+2)-1)\gamma_0^2>0.$

\section{Conditional moments and convergence in distribution}\label{sec:conditional-moments-convergence}

\propref{prop:fixed-moment-asymptotics} provides the asymptotic formula of every fixed moment of $X_{x,H}$. These are exactly the moments of the candidate Gaussian mixture $\sqrt{C\left(G\right)}Z$; in particular, its even moments are\[
M_{2k}:=\mu_{2k}\int_{V_{\square}}C\left(g\right)^{k}\,dg.
\]Nevertheless, convergence of all fixed moments does not by itself determine a limiting distribution in our case. One must also know that the resulting moment sequence determines a unique probability law. For example, applying Carleman’s criterion requires\[
\sum_{k\geq1}M_{2k}^{-\frac{1}{2k}}=\infty.
\]Note that \secref{sec:paired-main-term-limiting-variance} only establishes the finiteness of $M_{2k}$ for each fixed $k$. In fact, in \secref{app:C-moment-growth}, we prove that there exists a constant $c_{d}>0$ such that\[
\int_{V_{\square}}C\left(g\right)^{k}\,dg\geq\exp\left(c_{d}\frac{k^{2}}{\left(\log k\right)^{2}}\right)
\]showing that the even moments grow too rapidly for Carleman’s criterion to apply. Thus the fixed-moment asymptotics alone do not identify the law of the mixture. We do not claim here that the mixture is moment-indeterminate; rather, we avoid the separate problem of proving its global moment determinacy.

The key to overcoming this difficulty is to retain the local
reduction vector $G_H$ rather than study $X_{x,H}$ alone.
The limiting distribution of $G_H$ is the Haar probability measure on
$V_\square$. For a fixed value $g\in V_\square$ of the local data, the
proposed limiting distribution of $X_{x,H}$ is the Gaussian
$N(0,C(g))$. Unlike the Gaussian mixture, this distribution
is uniquely determined by its moments
(see, e.g., \cite[Example 30.1]{key-18}).
Our strategy is therefore to identify the limiting
distribution of $X_{x,H}$ after conditioning on the
local data.

We first establish the basic properties of $V_\square$
and show that $G_H$ converges in distribution to Haar
measure. We then prove
\propref{prop:finite-coordinate-weighted-moments},
which extends the ordinary moment asymptotics by allowing
weights that depend on the reductions modulo $p^2$ for
finitely many primes. Using bounds for higher moments,
we show that these weighted identities remain valid
for any subsequential joint limit $(G,X)$ of
$(G_H,X_{x,H})$. Since the information contained in
$G$ is determined by its finite collections of coordinates,
these identities yield
\[
\mathbb{E}[X^{2k}\mid G]=\mu_{2k}C(G)^k,
\quad
\mathbb{E}[X^{2k+1}\mid G]=0.
\]
These are precisely the moments of the Gaussian
$N(0,C(G))$. Since this Gaussian is determined by its
moments, every subsequential joint limit has the
distribution stated in
\thmref{thm:gaussian-mixture-limit}.
Tightness and uniqueness of the subsequential limit
then establish the full joint convergence.

Throughout this section, the notation $G_H$, $X_{x,H}$,
$\mathbb{E}_H$ and $\mathcal{F}_w$ is as introduced in
\secref{sec:introduction}.
The range $x=H^{o(1)}$ is interpreted sequentially:
along every sequence $(x_j,H_j)$ with
$x_j,H_j\to\infty$ and $x_j=H_j^{o(1)}$.

\subsection{Finite-coordinate structure and the Haar limit}\label{sec:finite-coordinate-haar-limit}

Enumerate the primes as $2=p_{1}<p_{2}<\cdots$ and equip $V_{\square}$ with\[
d_{V}\left(g,h\right):=\sum_{n\geq1}2^{-n}\mathds{1}_{\left\{ g_{p_{n}}\neq h_{p_{n}}\right\} }.
\]It is immediate to verify that $d_{V}$ is indeed a metric. For $w\geq2$, put $Q_{w}=\prod_{p\leq w}p^{2}$, and let $\pi_{w}\colon V_{\square}\to V_{Q_{w}}$ be the natural projection. The finite-coordinate $\sigma$-algebra $\mathcal{F}_{w}$ defined in \eqref{eq:sigmaAlgbra} is equivalently $\sigma\left(\pi_{w}\right)$. Its atoms are the cylinder sets\[
\operatorname{Cyl}_{w}\left(a\right):=\pi_{w}^{-1}\left(\left\{ a\right\} \right)=\left\{ g\in V_{\square}:g\equiv a\bmod{Q_{w}}\right\} ,\quad a\in V_{Q_{w}}.
\]We begin with a lemma collecting several standard consequences of the preceding definitions that will be used throughout the remainder of the section. Throughout, for any topological space $X$, we denote its Borel $\sigma$-algebra by $\mathcal{B}\left(X\right)$.

\begin{lemma}\label{lem:finite-coordinate-haar-structure}
The metric $d_{V}$ induces the product topology on $V_{\square}$, and $V_{\square}$ is compact and Polish. For each $w\ge2$, the cylinder sets $\operatorname{Cyl}_{w}\left(a\right)$ are clopen and form the atoms of $\mathcal{F}_{w}$; hence every bounded $\mathcal{F}_{w}$-measurable function is continuous and is a finite linear combination of cylinder indicators. Moreover,\[
\mathcal{B}\left(V_{\square}\right)=\sigma\left(\bigcup_{w\geq2}\mathcal{F}_{w}\right).
\]With respect to the Haar probability measure $dg$ on $V_{\square}$,\[
\int_{\operatorname{Cyl}_{w}\left(a\right)}dg=Q_{w}^{-\left(d+1\right)}.
\]Furthermore, uniformly in $a\in V_{Q_{w}}$ for fixed $w$,\begin{equation}
\mathbb{E}_{H}\mathds{1}_{\left\{ G_{H}\in\operatorname{Cyl}_{w}\left(a\right)\right\} }=Q_{w}^{-\left(d+1\right)}+O_{d,w}\left(H^{-1}\right).\label{eq:ProbabilityOfCyl}
\end{equation}Consequently, the law of $G_{H}$ converges weakly to Haar measure $dg$ on $V_{\square}$ as $H\to\infty$.
\end{lemma}

\begin{proof}
The metric description is the standard metric for a countable product of finite discrete spaces. Agreement in the first $N$ coordinates is equivalent to lying in a sufficiently small $d_{V}$-ball, so $d_{V}$ induces the product topology. Compactness follows from compactness of the finite factors by Tychonoff’s theorem, while completeness and separability follow directly from the definition of the metric. The cylinder assertions are immediate from the discrete finite-coordinate topology. Since the cylinders form a basis, their union over $w$ generates $\mathcal{B}\left(V_{\square}\right)$. The Haar probability measure on $V_{\square}$ is uniform on every finite set of coordinates, and hence each $\operatorname{Cyl}_{w}\left(a\right)$ has the stated measure. For proving \eqref{eq:ProbabilityOfCyl}, put $q=Q_{w}$. For every fixed residue class $a\in V_{q}$, counting the $d+1$ coefficients in their fixed classes modulo $q$ gives\[
\#\left\{ P\in\mathscr{P}_{d}\left(H\right):P\equiv a\bmod q\right\} =\frac{\left|\mathscr{P}_{d}\left(H\right)\right|}{q^{d+1}}+O_{d,q}\left(H^{d}\right),
\]uniformly in $a$. Division by $\left|\mathscr{P}_{d}\left(H\right)\right|\asymp_{d}H^{d+1}$ proves \eqref{eq:ProbabilityOfCyl}. To prove weak convergence, let $f:V_{\square}\to\mathbb{R}$
be bounded and continuous. By compactness, $f$ is uniformly
continuous, so for every $\varepsilon>0$ there exists a function
$f_w$, depending only on the coordinates $g_p$ with $p\le w$,
such that
\[
\sup_{g\in V_{\square}}|f(g)-f_w(g)|<\varepsilon.
\] Such a function is a finite linear combination of cylinder indicators, for which convergence follows from \eqref{eq:ProbabilityOfCyl}. Letting $\varepsilon\downarrow0$ proves that the law of $G_{H}$ converges weakly to $dg$.
\end{proof}

\subsection{Proof of \propref{prop:finite-coordinate-weighted-moments}}\label{sec:proof-weighted-moment-asymptotics}

By \lemref{lem:finite-coordinate-haar-structure}, it suffices to consider cylinder indicators. Fix $g^{(w)}\in V_{Q_w}$ and put\[
\chi_w(g):=\mathds{1}_{\{g\equiv g^{(w)}\bmod Q_w\}}.
\]
The restriction to the cylinder $\chi_w$ fixes the
local coordinates at primes $p\le w$. We now
adapt the harmonic analysis of the ordinary moments
to this setting, beginning with the lcm truncation.

Throughout, $m,d,w$ are fixed, and all bounds are uniform in
the finite-prime cutoff and the chosen cylinder, with implied
constants depending on $m,d,w$.

\begin{proof}[Proof of \propref{prop:finite-coordinate-weighted-moments}]

{\bfseries Step 1. Reduction to the local-model moments.} If $g^{(w)}\equiv0\bmod p^2$ for some $p\le w$, then every polynomial in the corresponding coefficient class has $p^2$ dividing its content, so $E_P(x)=0$. The assertion is immediate since $C$ vanishes outside $V_\square^\sharp$. We may therefore assume that $g_p^{(w)}\in V_{p^2}^\sharp$ for every $p\le w$.

Applying \propref{prop:polynomial-local-moment-comparison} with
$\Phi=\chi_w$, whenever $H\ge x^{C_dm^2}$ we obtain
\[
\mathbb E_H\left[\chi_w(G_H)X_{x,H}^m\right]
=
\frac{x^{-m/4}}{\zeta(2d+2)}
\int_{V_\square^\sharp}\chi_w(g)E_g(x)^m\,dg
+o_{m,d,w}(1).
\]
Thus it
remains to evaluate the weighted local-model moment.

{\bfseries Step 2. LCM truncation and the shifted zero-sum relation.} Put $\beta:=(3m+1)/4$ and $z:=x^\beta$. Applying
\eqref{eq:weighted-local-lcm-truncation} with $\Phi=\chi_w$ gives
\begin{equation}
\begin{aligned}
\frac1{\zeta(2d+2)}
\int_{V_\square^\sharp}\chi_w(g)E_g(x)^m\,dg
&=
\frac1{\zeta(2d+2)}
\sum_{\substack{k_1,\ldots,k_m\ge1\\K\le z}}
\prod_{i=1}^m\mu(k_i)
\int_{V_\square^\sharp}\chi_w(g)
\prod_{i=1}^m\Delta_g(k_i;x)\,dg+O_{m,d,\varepsilon}(x^mz^{-1+\varepsilon}).
\end{aligned}
\label{eq:CylinderLCMTruncation}
\end{equation}
Our choice of $z$ makes the error $o(x^{m/4})$ for sufficiently small $\varepsilon>0$.

We expand the finite sum in \eqref{eq:CylinderLCMTruncation} using the local harmonic identity \eqref{eq:Delta-local-fourier}. For each resulting frequency tuple, write
\[
\sigma_i:=\frac{\ell_i}{k_i^2}\in\mathbb Q/\mathbb Z,
\quad 0<\ell_i<k_i^2,
\]
and let $\sigma_{i,p}$ denote its $p$-primary component. Since the M\"obius factors vanish unless each $k_i$ is squarefree, we restrict to such tuples. Recall that $K=[k_1,\ldots,k_m]\le z$, and put $Q:=Q_w$.

The cylinder restriction imposes the following condition on the frequencies.

\begin{lemma}\label{lem:cylinder-frequency-reduction}
In the Fourier expansion of
\eqref{eq:CylinderLCMTruncation}, a frequency tuple
$(\sigma_1,\ldots,\sigma_m)$ contributes only if
\begin{equation}
\theta:=\sum_{i=1}^m\sigma_i\in\mathcal C(Q).
\label{eq:CylinderShiftRelation}
\end{equation}
\end{lemma}

\begin{proof}
At primes $p>w$, the local polynomial is averaged over $V_{p^2}^{\sharp}$, so translation invariance forces
\[
\sum_{i=1}^m\sigma_{i,p}=0\quad(p>w).
\]
Hence $\theta$ has no $p$-primary component for $p>w$. Since each $k_i$ is squarefree, the remaining denominator of $\theta$ divides $Q$, as claimed.
\end{proof}

To recover an ordinary zero-sum relation, put $\tau_i:=Q\sigma_i$. By \lemref{lem:cylinder-frequency-reduction}, we have $\sum_i\tau_i=0$ in $\mathbb Q/\mathbb Z$. Multiplication by $Q$ removes the denominator factors
at primes $p\le w$. At primes $p>w$, since $p\nmid Q$,
it preserves the exact $p$-power in each reduced
denominator. Moreover,
\begin{equation}
F_x(\sigma)\le QF_x(Q\sigma),
\label{eq:CylinderMultiplicationBound}
\end{equation}
since $\|Q\sigma\|\le Q\|\sigma\|$.

Fix $\theta$ and, for each prime $p\le w$,
the $p$-primary components of $\sigma_1,\ldots,\sigma_m$
and the indices $i$ for which $p\mid k_i$.
There are $O_{m,w}(1)$ such choices. For each such choice, the low-prime factors arise
from the Fourier expansion of $\Delta_g(k_i;x)$.
At a prime $p\le w$ dividing $k_i$, the factor is
$p^{-2}T_{g_p^{(w)}}(-\sigma_{i,p};p^2)$;
if $p\nmid k_i$, no local condition is imposed,
and the factor is $1$. 
Indeed, by the definition of the root exponential sum,
\[
\left|p^{-2}T_{g_p^{(w)}}(-\sigma_{i,p};p^2)\right|
\le \frac{\rho_{g_p^{(w)}}(p^2)}{p^2}\le 1.
\]
Thus the product of all low-prime factors has absolute
value at most $1$, uniformly in the frequencies and
the chosen cylinder. These factors can therefore be
discarded when estimating the off-diagonal contributions.

Put $K_+:=\prod_{p\mid K,\ p>w}p$ and $q_+:=K_+^2$. Denote the product of the averaged local arithmetic factors at primes $p>w$ by $\mathfrak A_{\mathbf k}^{>w}(\boldsymbol\ell)$. We divide the surviving tuples into good and bad frequencies according as\[
|\mathfrak A_{\mathbf k}^{>w}(\boldsymbol\ell)|\le q_+^\eta
\quad\text{or}\quad
|\mathfrak A_{\mathbf k}^{>w}(\boldsymbol\ell)|>q_+^\eta.
\]The low-prime part of $K$ has only finitely many possibilities, so replacing $K$ by $K_+$ in the estimates that follow changes only the implied constants, which may depend on $w$. Crucially, on the good set, the truncation gives\begin{equation}
K_+\le K\le z,\quad
|\mathfrak A_{\mathbf k}^{>w}(\boldsymbol\ell)|
\le z^{2\eta}.
\label{eq:CylinderGoodArithmeticBound}
\end{equation}
We now bound the truncated contribution from frequencies that are not exactly paired in the original variables $\sigma_i$.

{\bfseries Step 3. The finite-shift off-diagonal estimate.} Assume $m\ge3$. We show that, for some $\delta_{m,w}>0$, the contribution of all not-exact-paired original-frequency configurations is $O_{m,d,w}(x^{m/4-\delta_{m,w}})$. We treat the high-prime bad and good sets separately.

{\bfseries Step 3a. Bad frequencies.} We first record the affine counterpart of the lcm-$L^2$ estimate \eqref{eq:reduced-frequency-zero-sum}. Write $\theta=a_0/r_0$ in lowest terms, where $r_0\mid Q$. Appending the fixed frequency $-a_0/r_0$ converts the affine relation $\sum_i a_i/r_i=\theta\bmod1$ into an ordinary zero-sum relation. Applying the Fundamental Lemma to $f_0,f_1,\ldots,f_m$,
where $f_0(a):=\mathds{1}_{a\equiv-a_0\bmod r_0}$,
gives, for nonnegative functions $f_i$,
\begin{equation}
\sum_{\substack{a_i\bmod r_i,\ (a_i,r_i)=1\\
\sum_i a_i/r_i=\theta\;(\bmod1)}}
\prod_{i=1}^m f_i(a_i)
\ll_{m,Q}
\frac{(r_1\cdots r_m)^{1/2}}{[r_1,\ldots,r_m]}
\prod_{i=1}^m\|f_i\|_2.
\label{eq:CylinderAffineL2}
\end{equation}
Indeed, $\|f_0\|_2=1$, and the additional denominator
contributes at most a factor $r_0^{1/2}\le Q^{1/2}$.

To adapt the proof of \eqref{eq:arithmetic-factor-moment}
using \eqref{eq:CylinderAffineL2}, write
$k_i=k_i^-k_i^+$, where $k_i^-$ is supported on
$p\le w$ and $k_i^+$ on $p>w$. The factors $k_i^-$ have $O_w(1)$ possibilities. In the proof of \eqref{eq:arithmetic-factor-moment}, the local-pattern calculation \eqref{eq:local-pattern-exponent} applies unchanged at $p>w$, while \eqref{eq:CylinderAffineL2} supplies the lcm-$L^2$ bound and the finitely many low primes cost $O_{m,d,w}(1)$. 
For the normalized higher-moment estimate, fix the
low-prime components of each $\sigma_i$ and put
$N_i:=(k_i^+)^2$. We may write
\[
\sigma_i=\frac{c_i}{Q}+\frac{b_i}{N_i}\bmod 1,
\quad b_i\bmod N_i,
\]
where $c_i\bmod Q$ is fixed. Since $(Q,N_i)=1$,
multiplication by $Q$ sends the high-prime numerator
$b_i$ to $Qb_i\bmod N_i$. In particular, for every
$a\mid N_i$,
\[
a\mid b_i \quad\Longleftrightarrow\quad a\mid Qb_i.
\]
Thus multiplication by $Q$ preserves the high-prime
divisibility patterns used in
\eqref{eq:local-pattern-exponent}.

Moreover, restricting to $a\mid b_i$ gives the
shifted grid
\[
\frac{c_i}{Q}+\frac{j}{N_i/a}\bmod 1,
\quad 0\le j<N_i/a.
\]
Since $(Q,N_i/a)=1$, every nonzero point of this grid
has distance at least $1/(Q(N_i/a))$ from an integer.
Summing reciprocal distances therefore gives
\[
\frac1{N_i}
\sum_{\substack{b_i\bmod N_i,\ a\mid b_i\\
\sigma_i(b_i)\ne0}}
F_x(\sigma_i(b_i))
\ll_w \frac{\log(2N_i)}{a}.
\]
Consequently, the proof of
\eqref{eq:normalized-frequency-divisibility} and
the argument of
\lemref{lem:arithmetic-factor-higher-moment-bound}
apply to the shifted grids, with implied constants
depending additionally on $w$.
Hence, for each fixed choice of the frequency components and divisibility patterns at primes \(p\le w\), and every $\varepsilon>0$, the two resulting normalized estimates are
\[
\sum_{\substack{
0<\ell_i<k_i^2\ (1\le i\le m)\\
\sum_{i=1}^m\ell_i/k_i^2\equiv\theta\bmod{1}\\
(\ell_i/k_i^2)_p=\sigma_{i,p}\ (p\le w)
}}
\prod_{i=1}^m\frac{F_x(\sigma_i)}{k_i^2}
|\mathfrak A_{\mathbf k}^{>w}(\boldsymbol\ell)|^{1+\vartheta}
\ll_{m,d,w,\varepsilon}
\min\{K_+^\varepsilon,x^{m/2}K_+^{-2+\varepsilon}\}.
\]
where the sum ranges over the remaining high-prime frequencies. On the bad set, Rankin's trick contributes $q_+^{-\eta\vartheta}=K_+^{-2a}$, with $a:=\eta\vartheta$. Summing over the $O_{m,w}(1)$ low-prime choices and the $O_{m,\varepsilon}(K_+^{\varepsilon})$ tuples of high-prime moduli with fixed lcm, the argument of \eqref{eq:bad-small-lcm}--\eqref{eq:bad-large-lcm} gives that the bad set contribution is
\[
\ll_{m,d,w,\varepsilon}
Y^{1-2a+\varepsilon}+x^{m/2}Y^{-1-2a+\varepsilon}.
\]
Taking $Y=x^{m/4}$ and $\varepsilon>0$ sufficiently small gives a power saving, uniformly in $z$.

{\bfseries Step 3b. Good frequencies with some $\tau_i=0$.} We fix the local data at primes $p\le w$ and apply
the square-block lifting from
\secref{sec:nonpaired-frequency-analysis} at primes $p>w$.
This enlarges the frequency sums to the complete
square-block classes used by GMR. The additional
squarefree factors of the $k_i$ that do not appear
in the reduced denominators are summed absolutely
by \eqref{eq:squarefree-completion-summability}.
The resulting frequencies $\tau_i=Q\sigma_i$ satisfy
$\sum_i\tau_i=0$, and on the good set their
high-prime arithmetic factor is bounded by
$z^{2\eta}$, by \eqref{eq:CylinderGoodArithmeticBound}.

Suppose that $\tau_i=0$ for some $i$. Then $\sigma_i$ is supported only on the fixed primes, and since $\sigma_i\ne0$, its $F_x$ weight is $O_w(1)$. Removing these coordinates leaves $j<m$ nonzero frequencies satisfying the ordinary zero-sum relation. The ordinary paired estimate and \eqref{eq:GMR-nonexact-square-block-bound} bound their complete sum by $O_j(x^{j/4+o(1)})$. Thus this branch contributes\[
O_{m,d,w}\left(z^{2\eta}x^{(m-1)/4+o(1)}\right),
\]which has a power saving for sufficiently small $\eta$.

{\bfseries Step 3c. Good frequencies with non-paired high-prime tuples.} We may now assume that every $\tau_i$ is nonzero. If the high-prime square-block denominator tuple is repeated or non-paired in the GMR sense, \eqref{eq:CylinderMultiplicationBound} and \eqref{eq:GMR-nonexact-square-block-bound} give\[
O_{m,d,w}\left(z^{2\eta}x^{m/4-\Delta_m}\right).
\]Since $z=x^\beta$, this has a power saving when $2\beta\eta<\Delta_m$. Notice that GMR is applied only to the ordinary zero-sum tuple $(\tau_1,\ldots,\tau_m)$.

{\bfseries Step 3d. Good frequencies with paired high-prime tuples.} It remains to consider paired and non-repeated high-prime square-block denominators. Then $m=2k$ and, after relabeling, $R_{k+j}=R_j$ for $1\le j\le k$, with distinct $R_1,\ldots,R_k$. If $\tau_j+\tau_{k+j}\ne0$ for some $j$, there must be at least two nonzero pair blocks. The corresponding GMR nonzero-block estimate, as used in the proof of \eqref{eq:GMR-nonexact-square-block-bound}, gives the same bound as in Step 3c.

We may therefore assume that $\tau_{k+j}=-\tau_j$ for every $j$. The fixed low-prime components then determine $c_j\bmod Q$ such that\[
\sigma_j+\sigma_{k+j}=\frac{c_j}{Q}\bmod1.
\]The pair is an exact original-frequency pair precisely when $c_j=0$. For $c\ne0$, the points $0$ and $c/Q$ are
separated modulo $1$ by at least $1/Q$. Hence
at least one of $\|t\|$ and $\|c/Q-t\|$
is at least $1/(2Q)$, giving
\begin{equation}
F_x(t)F_x(c/Q-t)
\ll_Q F_x(t)+F_x(c/Q-t).
\label{eq:CylinderSeparatedSingularities}
\end{equation}
For a high-prime square-block denominator $R$, the fixed lifts lie in the nonzero grid $(QR)^{-1}\mathbb Z/\mathbb Z$. Since\[
\sum_{a=1}^{QR-1}F_x(a/(QR))\ll QR\log(2x),
\]enlarging the frequency sum to this grid and using \eqref{eq:CylinderSeparatedSingularities} gives\begin{equation}
\sum_{\substack{n\ge2\ \mathrm{squarefree}\\
p\mid n\Rightarrow p>w}}
\frac1{n^4}
\sum_{\tau\in\mathcal R_B(n^2)}
F_x(\sigma_1(\tau))F_x(\sigma_2(-\tau))
\ll_w\log(2x)
\label{eq:CylinderShiftedPair}
\end{equation}
whenever the pair shift is nonzero. Here $\sigma_1(\tau)$ and $\sigma_2(-\tau)$ are the
lifts determined by the fixed low-prime components.
By \eqref{eq:CylinderSeparatedSingularities} and the
preceding grid estimate, the contribution of each $R$
is $\ll_w R^{-1}\log(2x)$. Summing over squarefree $n\ge2$, with $R=n^2$,
gives the claimed bound, since
$\sum_{n\ge2}n^{-2}<\infty$.

For $c=0$, the ordinary second-moment pair estimate and \eqref{eq:CylinderMultiplicationBound} give $O_w(x^{1/2+o(1)})$. Once the distinctness condition on the denominators is removed, the nonnegative sum over paired frequencies factors into a product of independent one-pair sums. Unless every $c_j=0$, at least one pair satisfies the logarithmic bound \eqref{eq:CylinderShiftedPair}. Consequently, the contribution of paired high-prime tuples which are not exactly paired in the original frequencies is\[
O_{m,d,w}\left(z^{2\eta}x^{m/4-1/2+o(1)}\right).
\]This has a power saving for sufficiently small $\eta$. Combining Steps 3a--3d, and choosing $\eta$ sufficiently small in terms of $m$ and $\beta$, proves the asserted off-diagonal bound for $m\ge3$.

{\bfseries Step 4. The cases $m=1,2$.} For $m=1$, the frequency relation forces
$\tau_1=0$. Hence $\sigma_1\ne0$ is supported
on primes $p\le w$, so $F_x(\sigma_1)=O_w(1)$.
Absolute summability of the remaining squarefree
factors gives a first weighted local moment of
$O_{d,w}(1)=o(x^{1/4})$.

For $m=2$, \eqref{eq:Apbnd} gives
$|A_p(\boldsymbol{\lambda}_p)|\ll_d 1$
at every prime $p>w$, since $u_p\le2$.
Consequently, the high-prime arithmetic factor
is $O_{d,\varepsilon}(K_+^\varepsilon)$
for every $\varepsilon>0$. If both high-prime frequencies vanish, the contribution is $O_{d,w}(1)$. Otherwise, they are opposite. Exact original-frequency pairs form the zero-shift class, while nonzero shifts are bounded by \eqref{eq:CylinderShiftedPair}, with an additional harmless factor $R^\varepsilon$. Since $\sum_{n\ge2,\ n\ \mathrm{squarefree}}n^{-2+2\varepsilon}<\infty$
for sufficiently small $\varepsilon>0$, the nonzero-shift
contribution is $x^{o(1)}=o(x^{1/2})$. Hence only exact original-frequency pairings contribute on the scale $x^{1/2}$.

{\bfseries Step 5. Odd moments.} If $m$ is odd, Steps 3 and 4 show that every contribution is $o(x^{m/4})$, since an exact original-frequency pairing is possible only for even $m$. Together with \eqref{eq:CylinderLCMTruncation}, this proves the required odd-moment asymptotic.

{\bfseries Step 6. The diagonal contribution and conclusion.} Let $m=2k$. By the preceding steps, only exact original-frequency pairings can contribute on the scale $x^{k/2}$. We first justify replacing the lcm-truncated paired contribution by the complete paired contribution. The proof of \lemref{lem:paired-large-lcm-tail-local-model} remains valid with the cylinder weight inserted. Indeed, character orthogonality imposes the pair relations by averaging independent translations of the polynomial variable within each pair. Keeping $\chi_w(g)$ fixed during these translations, taking absolute values, and using $0\le\chi_w\le1$ bounds the resulting integrals by the same nonnegative products of shifted majorants $A_g(k;n+h)$ as in that proof. Since the proof of \lemref{lem:large-lcm-tail-local-model} is uniform in arbitrary integer arguments (including the shifts $n+h$), the weighted paired contribution with $K>Z$ is\[
O_{m,d,\varepsilon}(x^mZ^{-1+\varepsilon}),
\]uniformly in the finite-prime cutoff and the cylinder. Taking $Z=z=x^\beta$ makes this $o(x^{m/4})$.

For one fixed pairing, the complete exact-frequency diagonal is initially the distinct-denominator sum of \secref{sec:paired-main-term-limiting-variance}. When $k\ge2$, \propref{prop:repeated-denominators-negligible} bounds the repeated-denominator part by the nonnegative quantity $R_2(g,x)S(g,x)^{k-2}$. Inserting $\chi_w\le1$ cannot increase its integral, so this contribution remains $o(x^{k/2})$. For $k=1$, no distinctness restriction is needed. Thus, for every $k\ge1$, one fixed pairing contributes\[
\frac{x^{-k/2}}{\zeta(2d+2)}
\int_{V_\square^\sharp}\chi_w(g)S(g,x)^k\,dg+o(1).
\]By \eqref{eq:S-to-C}, $x^{-1/2}S(g,x)\to C(g)$ in $L^k(dg)$, so the $k$th powers converge in $L^1(dg)$. Hence the last expression tends to\[
\frac1{\zeta(2d+2)}
\int_{V_\square^\sharp}\chi_w(g)C(g)^k\,dg
=
\int_{V_\square}\chi_w(g)C(g)^k\,dg,
\]where the equality follows from the extension of $C$ by zero and \eqref{eq:C-sharp-to-Haar}. There are $\mu_{2k}=(2k-1)!!$ exact pairings, and the distinct-denominator condition makes their contributions disjoint. Summing over them proves the even-moment identity.

Finally, the odd case was treated in Step 5. Since every bounded $\mathcal F_w$-measurable function is a finite linear combination of cylinder indicators, linearity completes the proof of the proposition.
\end{proof}
\subsection{Subsequential limits and conditional moments}
\label{sec:subsequential-conditional-moments}

We now deduce the limiting law from the weighted-moment asymptotics.
The key point is that, once the local data $G$ are fixed, the proposed
limiting variance $C(G)$ is fixed as well. We first identify the
conditional moments of any subsequential joint limit.

Recall that, for an integrable random variable $Y$, the conditional
expectation $\mathbb E[Y\mid G]$ is the (almost surely unique)
integrable $\sigma(G)$-measurable random variable $W$ satisfying
\begin{equation}
\mathbb E[\mathds{1}_{\{G\in A\}}Y]
=\mathbb E[\mathds{1}_{\{G\in A\}}W]
\label{eq:defining_property_cond_exp}
\end{equation}

for every Borel set $A\subseteq V_\square$. Its existence follows,
for example, from the Radon--Nikodym theorem
(see \cite[Section 33]{key-18}). In the present setting, it is enough
to check this identity for sets depending on finitely many
prime coordinates: they form an algebra generating the Borel
$\sigma$-algebra by
\lemref{lem:finite-coordinate-haar-structure}, so the identity then
extends to all Borel sets by uniqueness of finite signed measures.

\begin{proposition}\label{prop:subsequential-conditional-moments}
Let $(x_j,H_j)$ satisfy $x_j,H_j\to\infty$ and
$x_j=H_j^{o(1)}$, and put
\[
G_j:=G_{H_j},\quad X_j:=X_{x_j,H_j}.
\]
Suppose, after passing to a subsequence, that
\[
(G_j,X_j)\overset{d}{\longrightarrow}(G,X).
\]
Then $G$ has Haar law $dg$, and for every integer $k\ge0$,
\[
\mathbb E[X^{2k}\mid G]=\mu_{2k}C(G)^k,
\quad
\mathbb E[X^{2k+1}\mid G]=0
\quad\text{a.s.}
\]
\end{proposition}

\begin{proof}
By \lemref{lem:finite-coordinate-haar-structure}, $G_j$ converges
in distribution to Haar measure on $V_\square$. The first coordinate
of the joint limit therefore has law $dg$.
Fix an integer $m\ge1$, and choose an integer $r$ with $2r>m$.
By \propref{prop:finite-coordinate-weighted-moments}, applied with
$\Phi\equiv1$ and moment order $2r$,
\[
\mathbb E|X_j|^{2r}
\longrightarrow
\mu_{2r}\int_{V_\square}C(g)^r\,dg<\infty.
\]
The integral is finite by \eqref{eq:C-from-D} and
\propref{prop:partial-sum-limits}. Consequently,
$\sup_j\mathbb E|X_j|^{2r}<\infty$ (absorbing finitely many
initial terms). In particular, for every $T\ge1$,
\[
\sup_j\mathbb E\bigl[|X_j|^m
\mathds{1}_{\{|X_j|>T\}}\bigr]
\ll_{m,r}T^{m-2r}.
\]
Thus the $m$th powers are uniformly integrable. The same
higher-moment bound passes to $X$ by weak convergence and
lower semicontinuity. Truncating the function $t\mapsto t^m$
and then removing the truncation therefore allows us to pass
weighted moments through the weak limit.

Specifically, fix $w\ge2$ and a bounded $\mathcal F_w$-measurable
function $\Phi$. By \lemref{lem:finite-coordinate-haar-structure},
$\Phi$ is continuous, so the preceding argument and
\propref{prop:finite-coordinate-weighted-moments} give
\begin{equation}\label{eq:limit-weighted-moment-suggestion9}
\mathbb E[\Phi(G)X^m]
=\mu_m\int_{V_\square}\Phi(g)C(g)^{m/2}\,dg
=\mathbb E[\Phi(G)\mu_m C(G)^{m/2}].
\end{equation}
Here we used the fact that $G$ has Haar law.
The variable $\mu_m C(G)^{m/2}$ is integrable and
$\sigma(G)$-measurable, since we have fixed a Borel
representative of $C$. Taking in
\eqref{eq:limit-weighted-moment-suggestion9} the indicator of
an arbitrary finite-coordinate cylinder, the defining property
of conditional expectation \eqref{eq:defining_property_cond_exp} yields
\[
\mathbb E[X^m\mid G]=\mu_m C(G)^{m/2}
\quad\text{a.s.}
\]
For odd $m$, $\mu_m=0$; for even $m$ this is the claimed formula.
The remaining case $m=0$ is immediate.
\end{proof}

\subsection{Identification of the limit law}
\label{sec:limit-law-identification}

The preceding conditional moments are those of a centered Gaussian
with variance $C(G)$. Unlike the unconditional Gaussian mixture,
a Gaussian with fixed variance is determined by its moments.
We now use this fact to identify the conditional distribution.

A conditional distribution of $X$ given $G$ is a family of
probability measures $(\nu_g)_{g\in V_\square}$ such that
\[
\mathbb P(G\in A,\,X\in B)
=\int_A\nu_g(B)\,dg
\]
for Borel sets $A\subseteq V_\square$ and $B\subseteq\mathbb R$.
Informally, $\nu_g$ describes the distribution of $X$ when the
local data take the value $g$. Such a family exists (up to a Haar-null
set) because $V_\square$ and $\mathbb R$ are standard Borel spaces;
see \cite[Theorem 33.3]{key-18}. Moreover, whenever
$\mathbb E|X|^m<\infty$, the $m$th moment of $\nu_g$ is a \emph{version}
of $\mathbb E[X^m\mid G=g]$; see
\cite[Theorem 34.5]{key-18}.

\begin{proposition}\label{prop:conditional-gaussian-identification}
Under the hypotheses of
\propref{prop:subsequential-conditional-moments}, the conditional
law of $X$ given $G=g$ is $N(0,C(g))$ for Haar-almost every $g$.
Consequently,
\[
\mathcal L(G,X)
=\mathcal L\bigl(G,\sqrt{C(G)}Z\bigr),
\]
where $Z\sim N(0,1)$ is independent of $G$.
\end{proposition}

\begin{proof}
The higher-moment bounds in the proof of
\propref{prop:subsequential-conditional-moments} show that all
moments of $X$ are finite. Thus, for each integer $m\ge1$, the
conditional-moment formula gives for Haar-almost every $g$
\[
\int_{\mathbb R}t^m\,\nu_g(dt)
=\mu_m C(g)^{m/2}.
\]
Since there are countably many moment orders, all these identities
hold simultaneously outside a single Haar-null set. Fix $g$
outside this set and put $c:=C(g)$.

If $c=0$, the second moment of $\nu_g$ vanishes, so
$\nu_g=\delta_0=N(0,0)$. If $c>0$, the image of $\nu_g$
under $t\mapsto t/\sqrt c$ has all the moments of $N(0,1)$.
The standard Gaussian is moment-determinate
(see \cite[Example 30.1]{key-18}), hence $\nu_g=N(0,c)$.

To identify the joint law, let $\gamma$ denote standard Gaussian
measure. The function $(g,z)\mapsto\sqrt{C(g)}z$ is Borel
measurable. Thus for Borel $A\subseteq V_\square$ and
$B\subseteq\mathbb R$, the conditional-distribution identity and
Tonelli's theorem give
\[
\begin{aligned}
\mathbb P(G\in A,\,X\in B)
&=\int_A\nu_g(B)\,dg\\
&=\int_A\int_{\mathbb R}
\mathds{1}_B\bigl(\sqrt{C(g)}z\bigr)
\,\gamma(dz)\,dg.
\end{aligned}
\]
This is exactly the joint probability for
$(G,\sqrt{C(G)}Z)$ when $Z$ is independent of $G$.
Since measurable rectangles generate the product Borel
$\sigma$-algebra, the asserted equality of joint laws follows.
\end{proof}

\subsection{Proof of \thmref{thm:gaussian-mixture-limit}}
\label{sec:proof-main-theorem}

Take an arbitrary sequence $(x_j,H_j)$ in the range of the theorem,
and put $G_j:=G_{H_j}$ and $X_j:=X_{x_j,H_j}$.
By \propref{prop:finite-coordinate-weighted-moments} with $m=2$
and $\Phi\equiv1$,
\[
\sup_j\mathbb E|X_j|^2<\infty.
\]
Since $V_\square$ is compact,
$V_\square\times[-T,T]$ is compact for every $T>0$, while
Markov's inequality gives
\[
\sup_j\mathbb P(|X_j|>T)\ll T^{-2}.
\]
Hence the joint laws of $(G_j,X_j)$ are tight on the Polish space
$V_\square\times\mathbb R$. By Prokhorov's theorem
\cite[Section 6, pp.~35--40]{key-19}, every subsequence has a
further weakly convergent subsequence, with some limit $(G,X)$.

By \propref{prop:subsequential-conditional-moments} and
\propref{prop:conditional-gaussian-identification}, every such
limit has law
\[
\mathcal L(G,X)
=\mathcal L\bigl(G,\sqrt{C(G)}Z\bigr),
\]
where $G$ has Haar law and $Z$ is an independent standard Gaussian.
Since every subsequence has a further subsequence with this same
limit, the full sequence converges in distribution to it.
The original sequence $(x_j,H_j)$ was arbitrary, so this proves
the claimed joint convergence. Marginal convergence follows
by projecting onto the second coordinate.

The positivity of $\operatorname{Var}_{dg}(C)$ and the
non-Gaussianity of the limiting law were established in
\secref{sec:mean-and-moments}.
\qed

\appendix
\section{Moment growth and coarse tail asymptotics for the limiting variance}
\label{app:C-moment-growth}

In this appendix, we establish matching upper and lower
bounds for the growth of the moments of $C(g)$ and
deduce a coarse asymptotic for its upper tail.
In particular, the moment lower bound shows that
Carleman's criterion does not hold for
$\sqrt{C(G)}Z$.
Throughout the appendix, $d\geq1$ is fixed.  We write $\mathbb E^\sharp$
and $\mathbb P^\sharp$ for expectation and probability on
$V_\square^\sharp$ with respect to its normalized product probability
measure.  Recall from \eqref{eq:Cextend} that $C$ is extended by
zero to $V_\square\setminus V_\square^\sharp$ and, by \eqref{eq:C-sharp-to-Haar},
\begin{equation}
 \int_{V_\square} C(g)^k\,dg
 =\frac{1}{\zeta(2d+2)}\,\mathbb E^\sharp C^k.
 \label{eq:app-sharp-Haar}
\end{equation}

\begin{theorem}\label{thm:C-two-sided-moment-growth}
There are constants $0<c_d<C_d<\infty$ and an integer $k_0(d)>0$ such that, for every
integer $k\geq k_0(d)$,
\[
 \exp\!\left(c_d\frac{k^2}{(\log k)^2}\right)
 \leq
 \int_{V_\square} C(g)^k\,dg
 \leq
 \exp\!\left(C_d\frac{k^2}{(\log k)^2}\right).
\]
Consequently, if $M_{2k}:=\mu_{2k}\int_{V_\square}C(g)^k\,dg$ then $\sum_{k\geq1}M_{2k}^{-1/(2k)}<\infty$.
\end{theorem}
The consequence is direct. We proceed with separately proving the lower and upper bounds.

\subsection{The lower bound} \label{sec:C(g)-lower-bound}

For a prime $p$ and $g_p\in V_{p^2}^\sharp$, put
\[
 Z_p(g_p):=
 \#\bigl\{y\bmod p:
 g_p(y)\equiv0\bmod{p^2},\quad g_p'(y)\equiv0\bmod p\bigr\},
\]
and let
\[
 \mathcal E_p:=\{Z_p=1\}.
\]
Ultimately, the lower bound arises by restricting to an explicit event contained in intersection of several $\mathcal{E}_p$s.

For the next lemma, recall the notation $X_{p,e}$ defined in \eqref{eq:Xpdef}.

\begin{lemma}\label{lem:app-singular-one-root}
There are constants $c_d>0$ and a positive integer $p_0(d)$ such that, for every prime
$p\geq p_0(d)$,
\[
 \mathbb P^\sharp(\mathcal E_p)\geq \frac{c_d}{p^2}.
\]
Moreover, provided $p\geq2d$, uniformly for
$\xi\in(\mathbb Z/p\mathbb Z)^\times$ and $g_p \in\mathcal E_p$,
\begin{equation}
 X_{p,1}(g_p;\xi)\geq\frac14
 \label{eq:app-Xp1-lower}
\end{equation}
while for every $g_p\in V_{p^2}^\sharp$,
\begin{equation}
 X_{p,0}(g_p)\geq \left(1-\frac d p\right)^2.
 \label{eq:app-Xp0-lower}
\end{equation}
\end{lemma}

\begin{proof}
For fixed $y\bmod p$, the conditions
\[
 g_p(y)\equiv0\bmod{p^2},\quad g_p'(y)\equiv0\bmod p
\]
impose three independent linear conditions on the $p$-adic digits of
the coefficients.  Averaging first on the full space $V_{p^2}$ and then
removing the zero polynomial therefore gives
\begin{equation}
 \mathbb E^\sharp Z_p
 =\frac{p^{-2}-p^{-2d-1}}{1-p^{-2d-2}}
 =p^{-2}+O_d(p^{-3}).
 \label{eq:app-EZp}
\end{equation}
We also have
\begin{equation}
 \mathbb E^\sharp\!\left[Z_p(Z_p-1)\right]\ll_d p^{-3}.
 \label{eq:app-factorial-Zp}
\end{equation}
Indeed, if all coefficients of $g_p$ are divisible by $p$, write
$g_p=ph$ modulo $p^2$.  For $d=1$, the nonzero polynomial $h\bmod p$
has at most one root, so this class contributes nothing to the
factorial moment.  For $d\geq2$, the class itself has probability
$O_d(p^{-d-1})=O_d(p^{-3})$, and $Z_p\leq d$.  Outside this class,
$g_p\bmod p$ is nonzero.  Two distinct elements counted by $Z_p$ would
then be two distinct double roots of $g_p\bmod p$, which is impossible
for $d\leq3$.  For $d\geq4$, a fixed ordered pair $y\neq y'$ imposes
four value-and-derivative conditions modulo $p$, together with the two
additional lifting conditions modulo $p^2$.  These six independent
$p$-adic linear conditions have probability $O_d(p^{-6})$; summing over
the $O(p^2)$ ordered pairs gives $O_d(p^{-4})$.  This proves
\eqref{eq:app-factorial-Zp}.
Note that $\mathbf 1_{\{Z_p=1\}}\ge Z_p-Z_p(Z_p-1)$ and hence
\[
 \mathbb P^\sharp(\mathcal E_p)
 \geq \mathbb E^\sharp Z_p
      -\mathbb E^\sharp[Z_p(Z_p-1)],
\]
and then \eqref{eq:app-EZp}--\eqref{eq:app-factorial-Zp} give the asserted lower
bound for $\mathbb P^\sharp(\mathcal E_p)$.

On $\mathcal E_p$, let $y$ be the unique singular residue class.
Since $g_p(y)\equiv0\bmod{p^2}$ and $g_p'(y)\equiv0\bmod p$,
all $y+pt$, $t\bmod p$, are roots modulo $p^2$.
Their contribution to $T_{g_p}(\xi/p;p^2)$ is
\[
\sum_{t\bmod p}e\left(\frac{\xi(y+pt)}p\right)
=p\,e\left(\frac{\xi y}p\right).
\]
Since the remaining roots are nonsingular, there are at most
$d$ of them. By the reverse triangle inequality,
\[
\left|T_{g_p}\left(\frac{\xi}{p};p^2\right)\right|
\ge p-d.
\]
Consequently, for $p\ge 2d$,
\[
X_{p,1}(g_p;\xi)
=p^{-2}\left|T_{g_p}\left(\frac{\xi}{p};p^2\right)\right|^2
\ge \left(1-\frac dp\right)^2
\ge\frac14,
\]
proving \eqref{eq:app-Xp1-lower}.
Finally, every $g_p\in V_{p^2}^\sharp$ has at most $dp$ roots modulo
$p^2$, i.e. $\rho_{g_p}(p^2)\le dp$, for $p\ge d$ we have
\[
X_{p,0}(g_p)
=\left(1-\frac{\rho_{g_p}(p^2)}{p^2}\right)^2
\ge \left(1-\frac dp\right)^2,
\]
proving \eqref{eq:app-Xp0-lower}.
\end{proof}

\begin{proof}[Proof of the lower bound in
\thmref{thm:C-two-sided-moment-growth}]
We first establish a mean-value estimate for the local factors at
primes larger than a cutoff $z\geq2$. Recall from
\lemref{lem:local-coefficient-moment-bounds} that
\[
\eta_{p,e}:=\mathbb E^\sharp X_{p,e},\quad
\eta_{p,0}=1+O_d(p^{-2}),\quad
\beta_p:=\frac{\eta_{p,2}}{\eta_{p,0}} 
=1-\frac1p+O_d(p^{-2}), \quad R_z:=\prod_{p>z}\eta_{p,0}.
\]
Define the multiplicative function
\[
m_z(b):=\mu^2(b)\mathds{1}_{p\mid b\Rightarrow p>z}
\prod_{p\mid b}\beta_p.
\]
The product $R_z$ is bounded above and below by positive constants
depending only on $d$. Moreover, $m_z$ has mean value
\begin{equation}
\vartheta_z
:=\prod_{p\leq z}\left(1-\frac1p\right)
\prod_{p>z}\left(1-\frac1p\right)
\left(1+\frac{\beta_p}{p}\right)
\asymp_d\frac1{\log z}
\label{eq:app-vartheta-z}
\end{equation}
by Mertens' theorem and since the local factors in the second product equal $1 + O_d(p^{-2})$. Applying the convolution argument from the proof of
\lemref{lem:finite-prime-partial-sum-limit}, we obtain, for fixed $z$,
\begin{equation}
\sum_{b\leq B}m_z(b)=\vartheta_z B+o_z(B).
\label{eq:app-mz-mean}
\end{equation}
Consequently, by \eqref{eq:app-mz-mean} and partial summation,
using $f(t^2)\ll\min(1,t^{-4})$ to control the tails,
we obtain, for every fixed squarefree $a$ supported on primes
at most $z$,
\begin{equation}
x^{-1/2}\sum_{b\geq1}m_z(b)
f\!\left(\frac{ab^2}{x}\right)
\xrightarrow{x\to\infty}
\frac{\vartheta_z}{\sqrt a}
\int_0^\infty f(t^2)\,dt.
\label{eq:app-weighted-mz}
\end{equation}
where $f$ is the function constructed in \secref{sec:limiting-variance-construction}. For convenience, write
\[
\gamma_f:=2\int_0^\infty f(t^2)\,dt>0.
\]
We now apply this estimate to the conditional variance
\[
C_z:=\mathbb E^\sharp[C\mid\mathcal F_z].
\]
By \eqref{eq:S-to-C} and the fact that conditional
expectation is a contraction on $L^k$, we have
\[
x^{-1/2}\mathbb E^\sharp[S(g,x)\mid\mathcal F_z]
\longrightarrow C_z
\quad\text{in }L^k.
\]
Since $\mathcal F_z$ is generated by finitely many finite-valued
coordinates, this $L^k$ convergence holds for every fixed
configuration $(g_p)_{p\le z}$.
Choose $p_0=p_0(d)\geq2d$ sufficiently large that
\lemref{lem:app-singular-one-root} applies for every $p>p_0$.
Fix the coordinates
\[
g_p=(1,0,\ldots,0)\in V_{p^2}^\sharp,
\quad p\leq p_0,
\]
and define the positive-probability event
\[
\mathcal B_0
:=\bigcap_{p\leq p_0}
\left\{g_p=(1,0,\ldots,0)\right\}.
\] For $z>2p_0$, define
\[
\mathcal P_z:=\{p:z/2<p\leq z\},\quad
\mathcal B_z:=\mathcal B_0
\cap\bigcap_{p\in\mathcal P_z}\mathcal E_p.
\]
For each $J\subseteq\mathcal P_z$, let
$a_J:=\prod_{p\in J}p$.
We use the nonnegative decomposition of $S(g,x)$ from
\lemref{lem:one-pair-sum-decomposition}, retaining only
$\lambda=1$ and the terms of the form
\[
r=a_Jb^2,\quad
J\subseteq\mathcal P_z,\quad
p\mid b\Rightarrow p>z.
\]
The case $J=\varnothing$, $b=1$ gives $r=1$, which is
excluded from $S(g,x)$. Including this term would contribute
$O_z(x^{-1/2})$ to the normalized conditional sum
$x^{-1/2}\mathbb E^\sharp[S(g,x)\mid\mathcal F_z]$,
and therefore would not affect its limit $C_z$.

To obtain a lower bound for $C_z$, we restrict the
nonnegative sum $S(g,x)$ in
\lemref{lem:one-pair-sum-decomposition} to the terms
with $\lambda=1$ and $r=a_Jb^2$ as specified above.
On $\mathcal B_z$, for every $p\in J$,
\eqref{eq:app-Xp1-lower} and $X_{p,0}\leq1$ imply
\[
X_{p,1}(g_p;\xi)\geq\frac14X_{p,0}(g_p).
\] Consequently,
\[
\begin{aligned}
x^{-1/2}\mathbb E^\sharp[S(g,x)\mid\mathcal F_z]
\geq{}&
\frac{2R_z}{\sqrt{x}}
\left(\prod_{p\leq z}X_{p,0}(g_p)\right)
\sum_{J\subseteq\mathcal P_z}4^{-|J|}
\sum_{b\geq1}m_z(b)
f\!\left(\frac{a_Jb^2}{x}\right)
-O_z(x^{-1/2}).
\end{aligned}
\]
Letting $x\to\infty$, the left-hand side converges to $C_z(g)$,
while \eqref{eq:app-weighted-mz} evaluates each sum over $b$.
We therefore obtain, for $g\in\mathcal B_z$,
\begin{align}
C_z(g)\geq
\gamma_f R_z\vartheta_z
\left(\prod_{p\leq z}X_{p,0}(g_p)\right)
\sum_{J\subseteq\mathcal P_z}
\frac{4^{-|J|}}{\sqrt{a_J}}=
\gamma_f R_z\vartheta_z
\left(\prod_{p\leq z}X_{p,0}(g_p)\right)
\prod_{p\in\mathcal P_z}
\left(1+\frac1{4\sqrt p}\right).
\label{eq:app-Cz-lower-product}
\end{align}
To estimate this expression, recall that our choice of
$\mathcal B_0$ ensures $X_{p,0}(g_p)=1$ for $p\leq p_0$.
For $p>p_0$, \eqref{eq:app-Xp0-lower} applies, and
Mertens' theorem gives
\[
\prod_{p\leq z}X_{p,0}(g_p)
\geq
\prod_{p_0<p\leq z}\left(1-\frac dp\right)^2
\gg_d(\log z)^{-2d}.
\]
Since $R_z\asymp_d1$ and
$\vartheta_z\asymp_d(\log z)^{-1}$ by
\eqref{eq:app-vartheta-z}, it follows that
\begin{equation}
C_z(g)\gg_d
(\log z)^{-(2d+1)}
\prod_{p\in\mathcal P_z}
\left(1+\frac1{4\sqrt p}\right),
\quad g\in\mathcal B_z.
\label{eq:app-Cz-final-lower}
\end{equation}
We now use this bound to give a lower bound for $\mathbb E^\sharp C^k$.
By independence of the prime coordinates and
\lemref{lem:app-singular-one-root}, for sufficiently large $z$,
\begin{equation}
\mathbb P^\sharp(\mathcal B_z)
=
\mathbb P^\sharp(\mathcal B_0)
\prod_{p\in\mathcal P_z}\mathbb P^\sharp(\mathcal E_p)
\geq
c_d\prod_{p\in\mathcal P_z}\frac1{2p^2}.
\label{eq:app-Bz-prob}
\end{equation}
Moreover, the conditional Jensen's inequality gives
\[
\mathbb E^\sharp C^k
\geq\mathbb E^\sharp C_z^k
\geq
\mathbb P^\sharp(\mathcal B_z)
\left(\inf_{g\in\mathcal B_z}C_z(g)\right)^k.
\]
Substituting \eqref{eq:app-Cz-final-lower} and
\eqref{eq:app-Bz-prob}, taking logarithms, and using
$\log(1+u)\geq u/2$ for sufficiently small $u$, we obtain
\begin{equation}
\log\mathbb E^\sharp C^k
\geq
c_d k\sum_{z/2<p\leq z}\frac1{\sqrt p}
-2\sum_{z/2<p\leq z}\log p
-C_d k\log\log z-C_d\pi(z).
\label{eq:app-lower-log}
\end{equation}
By Chebyshev's estimates,
\[
\sum_{z/2<p\leq z}\frac1{\sqrt p}
\gg\frac{\sqrt z}{\log z},
\quad
\sum_{z/2<p\leq z}\log p\ll z,
\quad
\pi(z)\ll\frac z{\log z}.
\]
Applying the lower bound for the first prime sum and the upper
bounds for the remaining two in \eqref{eq:app-lower-log},
and absorbing the implied constants, we obtain
\begin{equation}
\log\mathbb E^\sharp C^k
\geq
c_d\frac{k\sqrt z}{\log z}
-C_dz-C_dk\log\log z.
\label{eq:app-lower-optimized}
\end{equation}
Finally, choose $z=\alpha_d k^2/(\log k)^2$, where
$\alpha_d>0$ is a sufficiently small constant.
Since $\log z\asymp\log k$, the positive term in
\eqref{eq:app-lower-optimized} has size
$\gg_d\sqrt{\alpha_d}\,k^2/(\log k)^2$, whereas
the negative $z$-term is $O_d(\alpha_d k^2/(\log k)^2)$
and $k\log\log k=o(k^2/(\log k)^2)$.
Choosing $\alpha_d$ sufficiently small, we conclude that
\begin{equation}
\mathbb E^\sharp C^k
\geq
\exp\!\left(c_d\frac{k^2}{(\log k)^2}\right)
\label{eq:app-lower-final}
\end{equation}
for all sufficiently large $k$. By
\eqref{eq:C-sharp-to-Haar}, the same lower bound, after
adjusting $c_d$, holds for $\mathbb E C^k$. This completes the proof of the lower bound.
\end{proof}

\subsection{The matching upper bound}
Recall the $\mathcal F_z$-measurable random variable
$D_g^{(z)}(\lambda)$ constructed in
\lemref{lem:finite-prime-partial-sum-limit}. We begin by producing an upper bound on its moments; this is done by following the proof of that lemma and making it more explicit. For $k\geq2$ and a prime $p$, define
\[
 \mathcal L_p(k)
 :=
 \sum_{\mathbf e\in\{0,1,2\}^k}
 p^{-N_1(\mathbf e)/2-N_2(\mathbf e)}
 \sup_{\boldsymbol\xi}
 \mathbb E^\sharp
 \prod_{j=1}^k X_{p,e_j}(g_p;\xi_j),
\]
where $N_i(\mathbf e):=\#\{j:e_j=i\}$ and the supremum is taken over
$\xi_j\in(\mathbb Z/p^{e_j}\mathbb Z)^\times$
for all indices $j$ with $e_j\neq0$.

\begin{lemma}\label{lem:app-finite-prime-k-majorant}
There is a constant $A_d\geq1$ such that, uniformly for $z\geq2$,
$\lambda\geq1$, and integers $k\geq2$,
\begin{equation}
 \mathbb E^\sharp D_g^{(z)}(\lambda)^k
 \leq
 A_d^k
 \prod_{p\leq z}
 \left(1-\frac1p\right)^k\mathcal L_p(k).
 \label{eq:app-Dz-Euler-majorant}
\end{equation}
\end{lemma}

\begin{proof}
We revisit the proof of \lemref{lem:finite-prime-partial-sum-limit},
this time keeping track of the dependence on $k$.
Recall that the conditional sums have the representation
\[
A_{g,\lambda}^{(z)}(Y)
=R_z
\sum_{\substack{ab^2\leq Y\\
a,b\ \mathrm{squarefree}\\
(a,b)=(ab,\lambda)=1}}
\alpha(a)\beta(b)
\prod_{p\leq z}
X_{p,\nu_p(ab^2)}
\left(g_p;\xi_p(ab^2,\lambda)\right),
\]
where $R_z=\prod_{p>z}\eta_{p,0}\ll_d1$ and
$\alpha,\beta$ are the multiplicative weights supported on
primes greater than $z$, satisfying
\[
\alpha_p\ll_d p^{-2},
\quad
\beta_p=1-\frac1p+O_d(p^{-2}).
\]
Since $Y^{-1/2}A_{g,\lambda}^{(z)}(Y)\to D_g^{(z)}(\lambda)$
in $L^k$, it suffices to bound
\[
\mathbb E^\sharp D_g^{(z)}(\lambda)^k
=\lim_{Y\to\infty}Y^{-k/2}
\mathbb E^\sharp A_{g,\lambda}^{(z)}(Y)^k.
\]
Expanding the $k$th power, we separate in each copy the
small and large prime factors by writing
$a_j=a_{0,j}a_{1,j}$ and $b_j=b_{0,j}b_{1,j}$,
where the prime factors of $a_{0,j}b_{0,j}$ are at most $z$,
and those of $a_{1,j}b_{1,j}$ exceed $z$.

For fixed small-prime factors, the local frequencies
$\xi_p(ab^2,\lambda)$ may depend on the tail variables
$a_1,b_1$. We remove this dependence by bounding the
expectation over $g_p$, $p\leq z$, by the supremum over
the corresponding unit frequencies. Since all summands
are nonnegative, we may then discard the remaining
coprimality and congruence restrictions in the tail sums. For each fixed pair $a_0,b_0$, the resulting tail sum satisfies
\begin{align*}
&\limsup_{Y\to\infty}Y^{-1/2}
\sum_{\substack{a_1b_1^2\leq Y/(a_0b_0^2)\\
a_1,b_1\ \mathrm{squarefree}\\
p\mid a_1b_1\Rightarrow p>z}}
\alpha(a_1)\beta(b_1)\leq
\frac{A_d}{\sqrt{a_0}b_0}
\prod_{p\leq z}\left(1-\frac1p\right).
\end{align*}
Indeed, by \eqref{eq:app-mz-mean}, the $b_1$-weights
have mean value $\ll_d\prod_{p\leq z}(1-p^{-1})$,
while the $a_1$-sum is uniformly bounded since
\[
\sum_{\substack{a_1\ \mathrm{squarefree}\\
p\mid a_1\Rightarrow p>z}}
\frac{\alpha(a_1)}{\sqrt{a_1}}
\leq
\prod_{p>z}\left(1+\frac{A_d}{p^{5/2}}\right)
\ll_d1.
\]
Applying this estimate to each of the $k$ copies,
the contribution of the tail variables is bounded by
\[
A_d^k
\prod_{p\leq z}\left(1-\frac1p\right)^k
\prod_{j=1}^k\frac1{\sqrt{a_{0,j}}b_{0,j}},
\]
where the uniformly bounded factor $R_z^k$ has been
absorbed into $A_d^k$. If $e_{p,j}=\nu_p(a_{0,j}b_{0,j}^2)$,
then
\[
\prod_{j=1}^k\frac1{\sqrt{a_{0,j}}b_{0,j}}
=
\prod_{p\leq z}
p^{-N_1(\mathbf e_p)/2-N_2(\mathbf e_p)},
\]
where $\mathbf e_p=(e_{p,1},\ldots,e_{p,k})$.
Thus, summing over the small-prime exponent patterns
and using independence of the coordinates $g_p$,
we obtain
\[
\mathbb E^\sharp D_g^{(z)}(\lambda)^k
\leq
A_d^k
\prod_{p\leq z}\left(1-\frac1p\right)^k
\prod_{p\leq z}
\left(
\sum_{\mathbf e\in\{0,1,2\}^k}
p^{-N_1(\mathbf e)/2-N_2(\mathbf e)}
\sup_{\boldsymbol\xi}
\mathbb E^\sharp
\prod_{j=1}^k X_{p,e_j}(g_p;\xi_j)
\right).
\]
The expression in parentheses is precisely
$\mathcal L_p(k)$, proving
\eqref{eq:app-Dz-Euler-majorant}.
\end{proof}
We now estimate the local factors $\mathcal L_p(k)$.
Recall that
\[
Z_p:=
\#\left\{
y\bmod p:
g_p(y)\equiv0\bmod p^2,\quad
g_p'(y)\equiv0\bmod p
\right\}.
\]
The proof of \lemref{lem:root-exponential-sum-moment-bounds}
shows that $0\leq Z_p\leq d$ and, uniformly in the unit
frequencies,
\[
\left|T_{g_p}\left(\frac{\xi}{p};p^2\right)\right|
\leq d+pZ_p,
\quad
\left|T_{g_p}\left(\frac{\xi}{p^2};p^2\right)\right|
\leq d.
\]
Using $Z_p\le d$ and $(a+b)^2\le2(a^2+b^2)$,
we obtain, with $A_d=2d^2$
\begin{equation}
\begin{aligned}
\sup_\xi X_{p,1}(g_p;\xi)\leq p^{-2}(d+pZ_p)^2
&\leq A_d\left(p^{-2}
+\mathds1_{\{Z_p\geq1\}}\right),\\
\sup_\xi X_{p,2}(g_p;\xi)&\leq d^2.
\end{aligned}
\label{eq:app-pointwise-X12}
\end{equation}
The same proof gives $\mathbb E^\sharp Z_p\ll_d p^{-2}$,
so
\begin{equation}
\mathbb P^\sharp(Z_p\geq1)
\leq\mathbb E^\sharp Z_p
\ll_d p^{-2}.
\label{eq:app-singular-prob-upper}
\end{equation}

\begin{lemma}\label{lem:app-local-k-factor}
There exists $A_d\geq1$ such that, for $p\leq A_dk$,
\begin{equation}
\log\left[
\left(1-\frac1p\right)^k\mathcal L_p(k)
\right]
\leq A_d\frac{k}{\sqrt p}.
\label{eq:app-local-smallp}
\end{equation}
For $p>A_dk$,
\begin{align}
\log\left[
\left(1-\frac1p\right)^k\mathcal L_p(k)
\right]
\leq{}&
A_d\bigg(\frac{k^2}{p^2} + \frac{k}{p^{5/2}}\bigg)
+\log\left(
1+\frac{A_d}{p^2}
e^{A_dk/\sqrt p}
\right).
\label{eq:app-local-largep}
\end{align}
\end{lemma}

\begin{proof}
Write
\[
\mathcal L_p(k)=H_p(k)+J_p(k),
\]
where $H_p(k)$ contains the exponent patterns
$\mathbf e\in\{0,1,2\}^k$ with $N_1(\mathbf e)=0$,
and $J_p(k)$ contains those with $N_1(\mathbf e)\geq1$.
We first estimate $H_p(k)$. Since $X_{p,0}\leq1$,
$X_{p,2}\leq d^2$, and
\lemref{lem:local-coefficient-moment-bounds} that gives
\[
\mathbb E^\sharp X_{p,2}(g_p;\xi)
=1-\frac1p+O_d(p^{-2}),
\]
we obtain
\[
\begin{aligned}
H_p(k)
&\leq 1+\frac{k}{p}
\left(1-\frac1p+O_d(p^{-2})\right)
+\sum_{t=2}^k\binom kt\frac{d^{2t}}{p^t}\\
&\leq \exp\left(
\frac{k}{p}+A_d\frac{k^2}{p^2}
\right),
\quad p>A_dk.
\end{aligned}
\]
Hence
\begin{equation}
\left(1-\frac1p\right)^kH_p(k)
\leq \exp\left(A_d\frac{k^2}{p^2}\right).
\label{eq:app-no-e1}
\end{equation}
Thus the factor $(1-p^{-1})^k$ cancels the
leading contribution $k/p$.
To estimate $J_p(k)$, put
\[
U_p:=X_{p,0}
+\frac1p\sup_\xi X_{p,2}(g_p;\xi),
\quad
V_p:=\frac1{\sqrt p}
\sup_\xi X_{p,1}(g_p;\xi).
\]
By \eqref{eq:app-pointwise-X12},
\[
U_p\leq1+\frac{d^2}{p},
\quad
V_p\leq A_d\left(
p^{-5/2}
+p^{-1/2}\mathds1_{\{Z_p\geq1\}}
\right).
\]
Expanding the $k$th power and keeping the terms
containing at least one factor $V_p$, we get
\[
J_p(k)
\leq\mathbb E^\sharp
\left[(U_p+V_p)^k-U_p^k\right].
\]
On $\{Z_p=0\}$, \eqref{eq:app-pointwise-X12} gives
$V_p\ll_d p^{-5/2}$, while $U_p\leq1+d^2/p$.
Hence, using $(u+v)^k-u^k\leq kv(u+v)^{k-1}$,
\[
\begin{aligned}
(U_p+V_p)^k-U_p^k
\leq kV_p(U_p+V_p)^{k-1}
\ll_d\frac{k}{p^{5/2}}
\left(1+\frac{A_d}{p}\right)^{k-1}
&\ll_d\frac{k}{p^{5/2}}e^{A_dk/p}.
\end{aligned}
\]
On $\{Z_p\geq1\}$, we use
\eqref{eq:app-singular-prob-upper}. Thus, for $p>A_dk$,
\[
\begin{aligned}
J_p(k) \ll_d \frac{k}{p^{5/2}}e^{A_dk/p}
+\frac1{p^2}
e^{A_dk/p+A_dk/\sqrt p} \leq A_d\frac{k}{p^{5/2}}
+\frac{A_d}{p^2}e^{A_dk/\sqrt p}.
\end{aligned}
\]
Combining this with \eqref{eq:app-no-e1} yields
\[
\begin{aligned}
\left(1-\frac1p\right)^k\mathcal L_p(k) \leq
\exp\left(A_d\frac{k^2}{p^2}\right)
\left(
1+A_d\frac{k}{p^{5/2}}
+\frac{A_d}{p^2}e^{A_dk/\sqrt p}
\right).
\end{aligned}
\]
Taking logarithms proves \eqref{eq:app-local-largep}.
For $p\leq A_dk$, no cancellation is needed.
Indeed, the same pointwise bounds give
\[
\mathcal L_p(k)
\leq\mathbb E^\sharp(U_p+V_p)^k
\leq\left(1+\frac{A_d}{\sqrt p}\right)^k
\leq e^{A_dk/\sqrt p},
\]
which proves \eqref{eq:app-local-smallp}.
\end{proof}

To control the logarithmic term in \eqref{eq:app-local-largep},
we use the following elementary consequence of the prime number
theorem (in fact, Chebyshev's bound suffices).

\begin{lemma}\label{lem:app-prime-exponential-sum}
For every fixed $A>0$ and all sufficiently large $k$,
\[
\sum_{p>Ak}
\log\left(1+\frac{A}{p^2}e^{Ak/\sqrt p}\right)
\ll_A\frac{k^2}{(\log k)^2}.
\]
\end{lemma}

\begin{proof}[Proof of the upper bound in
\thmref{thm:C-two-sided-moment-growth}]
By \lemref{lem:app-finite-prime-k-majorant},
\[
\mathbb E^\sharp D_g^{(z)}(\lambda)^k
\leq A_d^k
\prod_{p\leq z}
\left(1-\frac1p\right)^k\mathcal L_p(k).
\]
Applying \lemref{lem:app-local-k-factor}, we obtain
\begin{align*}
\log\mathbb E^\sharp D_g^{(z)}(\lambda)^k
\leq{}& O_d(k)
+A_dk\sum_{p\leq A_dk}p^{-1/2}
+A_dk^2\sum_{p>A_dk}p^{-2}\\
&+A_dk\sum_{p>A_dk}p^{-5/2}
+\sum_{p>A_dk}
\log\left(1+\frac{A_d}{p^2}e^{A_dk/\sqrt p}\right)\\
\ll_d{}& k+\frac{k^{3/2}}{\log k}
+\frac{k^2}{(\log k)^2}
\ll_d\frac{k^2}{(\log k)^2}.
\end{align*}
Here we used \lemref{lem:app-prime-exponential-sum} for the logarithmic sum, and the prime number theorem
and partial summation for the remaining sums.
Consequently,
\[
\mathbb E^\sharp D_g^{(z)}(\lambda)^k
\leq
\exp\left(C_d\frac{k^2}{(\log k)^2}\right),
\]
uniformly in $z$ and $\lambda$.
By the proof of \propref{prop:partial-sum-limits},
$D_g^{(z)}(\lambda)\to D_g(\lambda)$ in $L^k$
as $z\to\infty$. Thus
\begin{equation}
\sup_{\lambda\geq1}
\mathbb E^\sharp D_g(\lambda)^k
\leq
\exp\left(C_d\frac{k^2}{(\log k)^2}\right).
\label{eq:app-D-upper}
\end{equation}
Recall from \eqref{eq:C-from-D} that
\[
C(g)=\frac1\pi
\sum_{\lambda\geq1}
\frac{D_g(\lambda)}{\lambda^{3/2}}.
\]
Minkowski's inequality and \eqref{eq:app-D-upper}
therefore yield
\[
\begin{aligned}
\|C\|_{L^k(\mathbb P^\sharp)}\leq\frac1\pi
\sum_{\lambda\geq1}\lambda^{-3/2}
\|D_g(\lambda)\|_{L^k(\mathbb P^\sharp)}\leq
\exp\left(C_d\frac{k}{(\log k)^2}\right),
\end{aligned}
\]
after enlarging $C_d$. Raising to the $k$th power, we obtain the required upper bound.
\end{proof}

\subsection{A coarse tail asymptotic for the limiting variance}
\label{sec:C-tail-asymptotic}

The moment estimates above also yield matching logarithmic bounds
for the upper tails of $C(G)$ and the limiting Gaussian mixture.

\begin{corollary}\label{cor:C-tail-asymptotic}
Let $G$ be distributed according to $dg$, and let
$Z\sim N(0,1)$ be independent of $G$. Then, as $t\to\infty$,
\[
\begin{aligned}
-\log\mathbb P(C(G)>t)
&\asymp_d (\log t)^2(\log\log t)^2,\\
-\log\mathbb P\bigl(|\sqrt{C(G)}Z|>t\bigr)
&\asymp_d (\log t)^2(\log\log t)^2.
\end{aligned}
\]
\end{corollary}

\begin{proof}
Write $m(k):=\log\mathbb E[C(G)^k]$. By
\thmref{thm:C-two-sided-moment-growth}, there are constants
$0<a\leq b<\infty$, depending only on $d$, such that
\[
a\frac{k^2}{(\log k)^2}\leq m(k)
\leq b\frac{k^2}{(\log k)^2}
\]
for all sufficiently large integers $k$.
Put $T:=(\log t)^2(\log\log t)^2$ and, for fixed $\lambda>0$, let
\[
k_\lambda:=\left\lceil\lambda(\log t)(\log\log t)^2\right\rceil.
\]
As $t\to\infty$,
\[
k_\lambda\log t=(\lambda+o(1))T,
\quad
\frac{k_\lambda^2}{(\log k_\lambda)^2}
=(\lambda^2+o(1))T.
\]
Choose $0<\alpha<1/b$. By Markov's inequality,
\[
\begin{aligned}
\log\mathbb P(C(G)>t)
&\leq \log\left(
\frac{\mathbb E[C(G)^{k_\alpha}]}{t^{k_\alpha}}
\right)
=m(k_\alpha)-k_\alpha\log t
\leq (b\alpha^2-\alpha+o(1))T.
\end{aligned}
\]
Since $b\alpha^2-\alpha<0$, this gives the required
upper bound for the probability.

For the lower bound, fix $\beta>1/a$ and put $k=k_\beta$.
Then
\[
m(k)-k\log t\geq(a\beta^2-\beta+o(1))T
\longrightarrow\infty.
\]
Thus $\theta:=t^k/\mathbb E[C(G)^k]\leq1/2$ for sufficiently
large $t$. Applying the Paley--Zygmund inequality to $C(G)^k$,
\[
\begin{aligned}
\mathbb P(C(G)>t)
&\geq (1-\theta)^2
\frac{\mathbb E[C(G)^k]^2}{\mathbb E[C(G)^{2k}]}
\geq\frac14\exp(2m(k)-m(2k))\\
&\geq\frac14\exp(-m(2k))
\geq\frac14\exp\bigl(-(4b\beta^2+o(1))T\bigr),
\end{aligned}
\]
where we used $m(k)\geq0$ for all sufficiently large $k$.
Combining the two bounds proves the first estimate.

For the second estimate, independence gives the two-sided bound 
\[
\begin{aligned}
\mathbb P(C(G)>t^2)\mathbb P(|Z|>1)
&\leq \mathbb P\bigl(|\sqrt{C(G)}Z|>t\bigr)
\leq \mathbb P(C(G)>t)
+\mathbb P(|Z|>\sqrt t).
\end{aligned}
\]
The first estimate and the Gaussian tail bound
$\mathbb P(|Z|>\sqrt t)\ll e^{-t/2}$
complete the proof.
\end{proof}

\subsection*{Acknowledgments} 
This research is supported by the Israel Science Foundation (grant no.~2088/24) and an Alon Fellowship. I would like to express my sincere gratitude to my advisor, Dr.~Ofir Gorodetsky, for introducing me to this problem and for his invaluable advice and suggestions. I am especially thankful for his endless patience and careful attention throughout the demanding process of editing and revising this manuscript.

\bibliography{references}
\bibliographystyle{abbrv}

\end{document}